\documentclass{amsart}

\usepackage{amsfonts,latexsym,rawfonts,amsmath,amssymb,amsthm, mathrsfs, lscape}
\usepackage{fullpage}
\usepackage{graphicx}
\usepackage[all]{xy}
\usepackage[english]{babel}
\usepackage[utf8]{inputenc}
\usepackage{xfrac}
\usepackage[shortlabels]{enumitem}
\usepackage{comment}

\usepackage{array, tabularx}

\usepackage{bm}
\bmdefine{\bX}{X}
\bmdefine{\ba}{$\alpha$}

\usepackage{textcomp}

\usepackage{color}

\usepackage{xparse}
\usepackage{soul}

\usepackage{graphicx}

\newcommand{\referenza}{}

\newtheorem{thm}{Theorem}
\newtheorem*{thm*}{Theorem \referenza}
\newtheorem*{thm**}{Theorem}
\newtheorem{corollario}[thm]{Corollary}
\newtheorem*{cor*}{Corollary \referenza}
\newtheorem{lem}[thm]{Lemma}
\newtheorem*{lem*}{Lemma \referenza}

\newtheorem*{clm*}{Claim \referenza}
\newtheorem{prop}[thm]{Proposition}
\newtheorem*{prop*}{Proposition \referenza}

\newtheorem*{ass*}{Assumption \referenza}
\newtheorem{prob}[thm]{Problem}

\newtheorem*{conj*}{Conjecture \referenza}

\newtheorem{defi}[thm]{Definition}

\theoremstyle{plain}

\theoremstyle{definition}
\newtheorem{exa}[thm]{Example}

\theoremstyle{remark}
\newtheorem{rmk}[thm]{Remark}

\def \N {\mathbb N}

\def \ezi {\varepsilon}
\def \Q {\mathbb Q}
\def \R {\mathbb R}
\def \C {\mathbb C}
\def \Z {\mathbb Z}

\def \hM {\hat{M}}
\def \hX {\hat{X}}
\def \lv {\lVert}
\def \rv {\rVert}
\def  \alphabq {C^{4, \alpha}_{b,\ezi}}
\def \p {\partial}
\def \bpa {\bar{\partial}}

\renewcommand{\Re}{\mathsf{Re}\,}

\allowdisplaybreaks

\usepackage[hypertexnames=false,
    pdfpagemode=UseNone,
    breaklinks=true,
    extension=pdf,
    colorlinks=true,
    linkcolor=blue,
    citecolor=blue,
    urlcolor=blue,
]{hyperref}
\usepackage{cleveref}

\makeatletter
\@namedef{subjclassname@2020}{
  \textup{2020} Mathematics Subject Classification}
\makeatother

\begin{document}
 
\title{Blowing up Chern-Ricci flat balanced metrics}

\author{Elia Fusi}
\address{	Dipartimento Di Matematica, Università di Parma, Parco Area delle Scienze 53/A, 43124, Parma, Italy}
\email{elia.fusi@unipr.it}
\author{Federico Giusti}
\address{Instituto de Ciencias Matemáticas, Calle Nicolás Cabrera 13-15, 28049 Madrid, Spain}
\email{federico.giusti@icmat.es}


\begin{abstract}
Given a compact Chern-Ricci flat balanced orbifold, we show that its blow-up at a finite family of smooth points admits constant Chern scalar  curvature balanced metrics, extending Arezzo-Pacard's construction to the balanced setting. Moreover, if the orbifold has isolated singularities and admits crepant resolutions, we show that they always carry Chern-Ricci flat balanced metrics, without any further hypothesis. Along the way, we study two Lichnerowicz-type operators originating from complex connections and investigate the relation between their kernel and holomorphic vector fields, with the aim of discussing the general constant Chern scalar curvature balanced case. Ultimately, we provide a variation of the main Theorem assuming the existence of a special $(n-2,n-2)$-form and we present several classes of examples in which all our results can be applied.
\end{abstract}

\maketitle

\section*{Introduction}

Over the last decades, great attention was reserved by many authors to the search for \emph{canonical} Kähler metrics as uniquely determined representatives of the K\"ahler class, by coupling the Kähler condition with several curvature restrictions, culminating with the introduction (in the compact setting) of \textit{extremal} Kähler metrics by Calabi in \cite{C} as critical points of the so called \textit{Calabi functional}. Great effort was put into producing explicit examples of these metrics through various constructions such as gluing approaches (see \cite{AP,Sz2,Sz1}), adiabatic limits (see \cite{DS, F1, F2}) and solving partial differential equations through the continuity method (see \cite{A, CDS, Y1, Y3}).\par
Following the K\"ahler setting, canonical non-K\"ahler metrics are believed to have to satisfy both a cohomological condition, generalizing  the K\"ahler one, and a curvature restriction with respect to a suitable Hermitian connection. Throughout the years, many cohomological constraints were introduced for Hermitian metrics such as  \emph{balanced}, \emph{SKT}, \emph{Vaisman} and many others, and they were widely investigated in the literature, see for example \cite{AB2,AB1, AB3,AI,Gau,M,Va}. All of these have been central in the study of non-Kähler geometry, since it is believed that compact complex manifolds split into \lq\lq orthogonal\rq\rq   \,worlds, identified by the existence of metrics satisfying such cohomological restrictions (for example it was shown by Angella and Otiman in \cite{AO} that Vaisman structures exclude the existence of balanced and SKT metrics), and are thus topic of several open problems such as the Fino-Vezzoni conjecture, see \cite{FV},  stating that SKT and balanced metrics cannot coexist in the non-Kähler case (see \cite{CRS, FG, FLY, GiP} for some families for which the Fino-Vezzoni conjecture was proved to hold true).  \par
Regarding instead curvature conditions, since we will be focusing on balanced metrics, we will only discuss those that are proved to be more interesting in this case, meaning conditions on the Ricci tensors associated to the Chern connection. The most natural constraint to consider is the Einstein condition, which in the balanced case can be referred to two different tensors, as seen in \cite{LiY}, namely the first and second Chern-Ricci tensors. The latter both provide generalizations of the Riemannian Ricci tensor in the Kähler case, respectively corresponding to the Ricci tensor and the Yang-Mills tensor (with respect to the metric itself) of the Chern connection induced on the holomorphic tangent bundle. These were inspected by Angella, Calamai and Spotti  (in \cite{ACS2}), resulting in the fact that, on a compact non-K\"ahler manifold, among first Chern-Einstein metrics, only first Chern-Ricci flat metrics can exist. The combination of this last condition with the balanced one appears to be relevant in the study of the Hull-Strominger system, which is a system of PDEs arised from heterotic string theory (introduced independently in \cite{Hu, S}) that is believed to identify canonical metrics in the non-Kähler Calabi-Yau setting (see for example \cite{AGF, CPY2, FeY, GF, GFGM, PPZ} for some explicit solutions and properties of the system). Nevertheless, in order to seek for canonical metrics when Chern-Ricci flat ones cannot be obtained, different constraints on the curvature need to be considered. Indeed, looking again at the K\"ahler realm, constant scalar curvature Kähler metrics (cscK, for short) have been central in the geometrization problem of polarized varieties stated in the Yau-Tian-Donaldson conjecture (see \cite{D1,Ti, Y2}). The latter asserts that the existence of such metrics on a fixed K\"ahler class is equivalent to the so called \emph{K-polystability} of the polarization, which was shown to be true in \cite{CDS, Ti1} when the polarization is given by the anticanonical bundle and in \cite{A,Y3} in the canonically polarized case. For general polarizations the problem is still open, as only one implication was proved in \cite{Sto} using \cite{AP}. The last mentioned result  provides the existence of cscK metrics on the blow-up at a finite number of smooth points of cscK orbifolds with isolated singularities via gluing techniques.\par 
This suggests that the Chern scalar curvature might be relevant in identifying canonical metrics in the non-Kähler setting. On top of this, it is also clear that Chern-Einstein conditions automatically force the Chern scalar curvature to be constant, hence its pairing with the balanced constraint can be seen as a step in the direction of \emph{canonical balanced metrics}. Without imposing cohomological restrictions, the problem of producing constant Chern scalar curvature metrics was investigated by Angella, Calamai and Spotti in \cite{ACS} resulting in the existence of a unique constant Chern scalar curvature metric in any conformal class with non-positive Gauduchon degree (see also \cite{LeU} in the almost-complex case). However, as far as the authors know, there are no known result concerning the constant Chern scalar curvature condition when paired with a cohomological one.

The main goal of the present paper is to construct balanced constant Chern scalar curvature metrics on the blow-up of a Chern-Ricci flat balanced manifold or orbifold with isolated singularities via gluing construction, partially extending the result of Arezzo and Pacard in \cite{AP} to the balanced case.
\\
Recall that a $n$-dimensional Hermitian manifold $(M^n, J,  \omega)$ is a triple consisting of a complex manifold, an integrable complex structure $J$ and a positive $(1,1)$-form $\omega$; the latter  is said to be \emph{balanced} \cite{M} (or \emph{semi-K\"ahler} as in \cite{G})   if 
$
d\omega^{n-1}=0,
$ from which it is clear that a cohomology class in $H^{n-1, n-1}_{BC}(M, \R)$ is naturally associated to it, usually called \emph{balanced class}.
\\
With these notations, we are in the position to state our main result as follows.
 \begin{thm}\label{ansbal} Let $(M^n, \tilde\omega)$  be a compact balanced Chern-Ricci flat  manifold or orbifold with isolated singularities. 
Then, given $p_1, \ldots, p_k\in M$ and $a_1, \ldots, a_k>0$, there exists $\ezi_0>0$ such that the blow-up of $M$ at $p_1, \ldots, p_k$ admits a   negative constant Chern scalar curvature  balanced metric 
$$
\omega_{\ezi}^{n-1}\in \pi^*[\tilde\omega^{n-1}]_{BC}- \ezi^{2n-2}\left(\sum_{i=1}^ ka_i[E_i]_{BC}\right)^{n-1}\,,
$$
 where $[E_i]_{BC}$ is the first Bott-Chern class of the  line bundle associated to the  exceptional divisor $E_i$ of the blow-up at $p_i$ and $\ezi\in (0,\ezi_0)$.
\end{thm}
 The construction of the desired metric is a standard gluing technique and it follows closely those in \cite{AP,BM, GS, Sz, Sz1, Sz2}. In particular, in the first part,  through a cut-off procedure on the starting metric on $M$, we obtain an approximate constant Chern scalar curvature  balanced metric on the blow-up of $M$ gluing it with a suitable metric on ${\rm Bl}_0\C^n$ built by cutting-off  the Burns-Simanca metric.
  Once this ingredient is available,  the  constant Chern scalar curvature  balanced metric is found  via a deformation argument within the balanced class, as introduced by \cite{FWW}, adopting the ansatz  used in \cite{GS}, called the \emph{balanced ansatz}, to reduce the equation to a scalar one. This allows us to perform weighted analysis as in literature. The novelty of  this case with respect to the ones in the literature is that we need to work orthogonally to the kernel of the operator
$$
F_{\omega}=\Delta_{\omega}+\frac{1}{n-1}|\partial \omega|^2{\rm Id}
$$ in order to obtain the invertibility of the linearized operator (in Section \ref{examples} we show that on the same manifolds we can find metrics that have vanishing kernel for the corresponding $F_\omega$ operator, as well as metrics for which $\ker F_\omega\neq \{0\}$). To do this,  we observe that any function  in $ \ker F_{\omega} $ on $M$ gives rise to a family of functions on the  blow-up of $M$ with suitable properties, see  Lemma \ref{deformann} for the precise statement.  This family allows us to use the fact that we are working orthogonally to $\ker F_{\omega}$ and conclude the invertibility. Once we have this, we can rephrase our equation as a fixed point problem, for which we are able to use Banach's fixed point Theorem and conclude.\\
 Prior to the proof of Theorem \ref{ansbal}, we obtain formulas for the codifferentials of the Chern-Ricci form, allowing us to produce alternative expressions of the linearized operator in the general constant Chern scalar curvature  balanced setting. Moreover,   they suggest us a possible relation between the Bott-Chern harmonicity of the Chern-Ricci form and the constant Chern scalar curvature balanced condition, see Problem \ref{probharmcscb}. Furthermore, in the attempt of understanding the geometric nature of the kernel of the linearized operator, we study two different Lichnerowicz-type  operators, arising from the non-Kähler setting, namely the Chern-Lichnerowicz, whose kernel consist in those functions generating a holomorphic vector field, and a torsion-twisted version of the latter obtained with a connection introduced by Ustinovsky in \cite{Us}. Throughout the analysis, we give several interesting formulas, a necessary and sufficient conditions for a holomorphic vector field to be Chern-parallel (see Lemma \ref{Chernparallel}) and we investigate the relation of the kernels with holomorphic fields, obtaining along the way a variational characterization for the kernel of the \emph{torsion-twisted Lichnerowicz-type} operator (see Lemma \ref{varnostrolich}).
 
With the same approach used for the proof of \cite[Theorem 1]{GS}, we are able to prove also the following.
\begin{thm}\label{glueorb}
    Let $(M^n, \tilde\omega)$ be a compact Chern-Ricci flat  balanced orbifold with isolated singularities. 
Furthermore, assume  that $M$ admits a crepant resolution $\hM$.
Then,  there exists $\ezi_0>0$  such that $\hM$ admits a Chern-Ricci flat balanced metric $\hat{\omega}_{\ezi}$ such that
$$
\hat{\omega}_{\ezi}^{n-1}\in \pi^*[\tilde\omega^{n-1}]_{BC}- \ezi^{2n-2}\left(\sum_{i=1}^k\sum_{j=1}^{k_i}a_i^j[E_i^j]_{BC}\right)^{n-1}\,,
$$
 where $[E_i^j]_{BC}$ is the first Bott-Chern class of the line bundle associated to the $j$-th irreducible component of the exceptional divisor $E_i$ of the resolution and $\ezi\in (0,\ezi_0)$.
\end{thm}

In this last case, the main difference is that  Joyce's ALE Calabi-Yau metrics on crepant resolutions of orbifold isolated singularity models (see \cite{J}) take on the role that was of the Burns-Simanca metric in the blow-up case (see Subsection \ref{preli} for more precise details).\par

Mainly due to the great amount of possible deformations in a fixed balanced class, it is natural to wonder if the space of Chern-Ricci flat or constant Chern scalar curvature balanced metrics within it might be, at least, finite-dimensional.
In Section \ref{nonpostr}, we explore the construction of such metrics adopting a different ansatz, resulting in, possibly, different metrics to that constructed in Theorem \ref{ansbal} and Theorem \ref{glueorb}, taking a first step towards the understanding of the moduli space of Chern-Ricci flat or constant Chern scalar curvature balanced metrics in a prescribed balanced class. More in detail,  we analyse an ansatz involving a $(n-2,n-2)$-form on the manifold $M$, satisfying mild assumptions. With this at hand, we are able to prove the following.  
\begin{thm}\label{immain}
Let $(M^n, \tilde\omega)$ be a compact balanced Chern-Ricci flat  manifold or orbifold with isolated singularities endowed with $\tilde\Omega\in \Lambda_{\R}^{n-2,n-2}(M)$   such that

\begin{equation}\label{hypOm}
\tilde\omega \wedge \tilde\Omega>0 \quad \mbox{ and }\quad \Lambda_{\tilde\omega}^{n-1}(i\partial \bar \partial \tilde\Omega)\le0\,.
\end{equation} 
Then, given $p_1, \ldots, p_k\in M$ and $a_1, \ldots, a_k>0$, there exists $\ezi_0>0$ such that the blow-up of $M$ at $p_1, \ldots, p_k$ admits a   negative constant Chern scalar curvature  balanced metric 
$$
\omega_{\ezi}^{n-1}\in \pi^*[\tilde\omega^{n-1}]_{BC}- \ezi^{2n-2}\left(\sum_{i=1}^ ka_i[E_i]_{BC}\right)^{n-1}\,,
$$
 where $[E_i]_{BC}$ is the first Bott-Chern class of the  line bundle associated to the  exceptional divisor $E_i$ of the blow-up at $p_i$ and $\ezi\in (0,\ezi_0)$.
\end{thm}
The strategy of the proof of Theorem \ref{immain} follows closely that of Theorem \ref{ansbal} and of \cite[Theorem 1]{GS}. 
The novelty in this case is that  we have to produce, starting from $\tilde \Omega$, a suitable $(n-2, n-2)$-form $\Omega$ on the  blow-up of $M$. This is done by gluing  a  cut-off version of $\tilde \Omega$ with the $(n-2)$-th power of a well chosen metric on ${\rm Bl}_0\C^n$, obtained again by cutting-off the Burns-Simanca metric. 

Adapting the ideas used in Theorem \ref{immain}, as for Theorem \ref{glueorb}, we can also prove the following.
\begin{thm}\label{glueorbOm}
    Let $(M^n, \tilde\omega)$ be a compact Chern-Ricci flat  balanced orbifold with isolated singularities endowed with $\tilde \Omega\in \Lambda^{n-2, n-2}_{\R}( M)$ satisfying \eqref{hypOm} on the smooth part of $ M$. 
Furthermore, assume  that $M$ admits a crepant resolution $\hM$.
Then, there exists $\ezi_0>0$ such that  $\hM$ admits a Chern-Ricci flat balanced metric $\hat{\omega}_{\ezi}$ such that
$$
\hat{\omega}_{\ezi}^{n-1}\in \pi^*[\tilde\omega^{n-1}]_{BC}- \ezi^{2n-2}\left(\sum_{i=1}^k\sum_{j=1}^{k_i}a_i^j[E_i^j]_{BC}\right)^{n-1}\,,
$$
 where $[E_i^j]_{BC}$ is the first Bott-Chern class of the line bundle associated to the $j$-th irreducible component of the exceptional divisor $E_i$ of the resolution and $\ezi\in (0,\ezi_0)$.
\end{thm}

As it is shown in Subsection \ref{exOmega}, a large class of complex manifolds admits balanced Chern-Ricci flat metrics and a $(n-2, n-2)$-form satisfying \eqref{hypOm}. On the other hand, as Remark \ref{noOmnak} highlights, the existence of such $(n-2, n-2)$-form is not always guaranteed.\par
The rest of the paper is organized as follows.
In   Section \ref{genset}, we recall some basic definitions in Complex Geometry and describe  the properties of the Burns-Simanca metric. We then proceed to construct a balanced metric on the blow-up  which is approximately constant Chern scalar curvature and we study its behaviour. 
The rest of the Section is dedicated to the description of  the equation we want to solve, together with its linearization, corresponding to a fourth-order elliptic linear operator.\par
In Section \ref{baldef} we focus on the  study of the problem adopting the balanced ansatz. In particular, we obtain different explicit  expressions for the linearized operator given by this ansatz, highlighting the link with the K\"ahler case and hence its geometric relevance. Furthermore, we introduce and study the Chern-Lichnerowicz operator, and a  torsion-twisted Lichnerowicz-type version which fits nicely in the expression of the linearized operator of the problem. We then explore the properties of these operators, obtaining characterizations of their kernels, their relation with holomorphic vector fields, along with some properties of the latter in the case of compact balanced manifolds.

In Section \ref{proofimmain}, we derive the invertibility of the linearized operator obtained in the previous section in suitable weighted H\"older spaces, and we use it to rephrase our equation as a fixed point problem. Finally, through several weighted estimates, we show that the hypotheses of Banach's fixed-point Theorem are verified, allowing us to conclude the proof of Theorem \ref{ansbal} along with Theorem \ref{glueorb}.\par
Section \ref{nonpostr} is devoted to the proof of Theorem \ref{immain}. We quickly discuss how to achieve it  by observing the small differences with the proofs of the previous theorems.\par
Finally, Section \ref{examples} is dedicated to the showcase of several classes of examples on which our theorems can be applied, along with a couple of cases in which we can explicitly describe $\ker F_{\omega}$  for a family of metrics.

\subsection*{Acknowledgements}
The authors are deeply grateful to Cristiano Spotti for suggesting the problem, and for the many comments and suggestions throughout its completion.
Part of this work was conducted while the first author was visiting Northwestern University. For this,  he would like to thank the university and Gábor Székelyhidi for the cheerful hospitality and for many inspiring and insightful conversations. The second named author would  like to thank Institute of Advanced Study in Mathematics at Zhejiang University, along with Song Sun, for hosting him during his visit during which this paper was partially completed.\par
Both the authors are  supported by GNSAGA of INdAM. The first named author is  supported by University of Parma through the action Bando di Ateneo 2023 per la ricerca. The second named author partially supported by the Spanish Ministry of Science and Innovation, under grant CNS2022-135784.
The second named author is thankful to Villum Fonden for funding his Ph.D. position through Villum Young Investigator 0019098 grant, during which part of this work was carried out. 
\numberwithin{thm}{section}
\numberwithin{equation}{section}
\section{General setup of the problem}\label{genset}
In this first section, after recalling a few known objects which will be central in our construction, we will use them to produce an approximate solution, and see how it can be used to formulate an equation whose solution will provide a constant Chern scalar curvature balanced metric.

\subsection{Preliminaries and notations}\label{preli}
In this subsection,   we will briefly recall some classical definitions which will be needed in the next sections. \\
Let $(M^n, J, \omega )$ be a Hermitian manifold. It is known that  the Levi-Civita connection $\nabla^{LC}$ does not preserve the complex structure when the metric is non-K\"ahler. On the other hand, we can always consider the \emph{Chern connection} $\nabla$, namely the unique linear connection such that 
 $$
\nabla g =0\,, \quad \nabla J=0\,, \quad T^{1,1}=0
\,.$$
 With this at hand, we can define  the first Chern-Ricci form as 
 $$
 {\rm Ric}^{ch}(\omega):=-\frac12{\rm tr}(JR)\,, 
$$ where $R$ is the curvature tensor of Chern connection $\nabla$. It turns out that ${\rm Ric}^{ch}(\omega)\in\Lambda_{\R}^{1,1}(M)$ and it can be locally expressed as 
$$
{\rm Ric}^{ch}(\omega)=-i\partial \bar \partial \log \omega^n\,.
$$
 This local expression gives us immediately that ${\rm Ric}^{ch}(\omega)$ is $d$-closed and,  in particular,  it represents the \emph{first Bott-Chern class}  $c_1^{BC}(M)$ of $M$.
 Moreover,  the Chern scalar curvature of $\omega$ is defined as:
 \begin{equation}\label{scal}
s^{ch}(\omega):=n\frac{{\rm Ric}^{ch}(\omega)\wedge \omega^{n-1}}{\omega^n}\,.
\end{equation}

Another important $(1,1)$-form, related to the curvature of the Chern connection, is the second Chern-Ricci form,  defined locally as:
$$
{\rm Ric}^{(2)}(\omega)_{i\bar j}:=g^{k\bar l }R_{k\bar l i\bar j }\,.
$$
In what follows, in order to avoid any confusion, we will denote with: 
$$
{\rm Ric}^{ch}(g)(X, Y):={\rm Ric}^{ch}(\omega)(X, JY)\,, \quad {\rm Ric}^{(2)}(g)(X, Y):={\rm Ric}^{(2)}(\omega)(X, JY)\,, \quad X, Y\in TM\,,
$$  the $(1,1)$-symmetric tensors associated, respectively, to ${\rm Ric}^{ch}(\omega)$ and ${\rm Ric}^{(2)}(\omega)$.
Furthermore, we will make use of another tensor, firstly introduced in \cite{STi} in the context of geometric flows on Hermitian manifolds, which is a quadratic expression in the torsion $T$ of the Chern connection. Following \cite{STi}, it is defined locally as:
$$
Q^2_{i\bar j }:=g^{l\bar k }g^{p\bar q }T_{lp\bar j}T_{\bar k \bar q i}\,.
$$
It turns out that the tensor $Q^2$ is a symmetric $(1,1)$-tensor. In view of this, we will denote its associated $(1,1)$-form with $\Xi$, defined as:
\begin{equation}\label{Q2malefico}
\Xi(X, Y):=Q^2(JX, Y)\,, \quad X, Y\in TM\,.
\end{equation}


A fundamental tool in the proof of  \cite[Theorem 1.1]{AP} is the Burns-Simanca metric $\omega_{BS}$ on ${\rm Bl}_0\C^n$. This metric was defined in \cite{LeB} and \cite{Sim} and it is an asymptotically flat, scalar flat K\"ahler metric. Moreover, using the standard coordinates $\zeta$  on ${\rm Bl}_0\C^n\backslash E\simeq \C^n\backslash \{0\}$, $\omega_{BS}$ can be written as:
\begin{equation}\label{AF}
\omega_{BS}=i\partial \bar \partial (|\zeta|^2+ \gamma(|\zeta|))\,, \quad \gamma(|\zeta|)= O(|\zeta|^{4-2n}) \,, \quad |\zeta|\to \infty\,,
\end{equation} where $E$ is the \emph{exceptional divisor} of  ${\rm Bl}_0\C^n$.

The focus of the present note is to extend the results in \cite{AP} weakening the K\"ahler condition and focusing on balanced Hermitian metrics.

A Hermitian metric $\omega$ is said to be  balanced if $d\omega^{n-1}=0.$ The notion of balanced metric was firstly introduced by Michelsohn in \cite{M} where many properties of balanced manifolds,  such as their closedness under products and holomorphic submersions, were studied. Later, Alessandrini and Bassanelli in \cite{AB1} proved that compact balanced manifolds are closed under proper modifications. As  a difference from the K\"ahler condition which is preserved under small deformations, Alessandrini e Bassanelli showed that the balanced condition does not satisfy the same property, see \cite{AB3} for a counterexample.

As one can directly see from the definition, balanced metrics on complex surfaces are K\"ahler. Motivated by this, our focus will be on complex manifolds of dimension $n\ge 3$.

From a curvature point of view, balanced metrics are the unique non-K\"ahler metrics such that the  Gauduchon Ricci forms agree. In \cite{gauconn}, Gauduchon introduced a line of \emph{canonical} Hermitian connections $\{\nabla^t\}_{t\in \R}$, also known as \emph{Gauduchon connections}, defined by:
$$
g(\nabla^t_XY, Z)=g(\nabla_X^{LC}Y, Z)+ \frac{t-1}{4}(d^c\omega)(X, Y, Z)+ \frac{t+1}{4}(d^c\omega)(X, JY, JZ)\,, \quad X, Y, Z \in  TM\,.
$$
Most notably, for $t=1$, we recover the definition of the Chern connection while, for $t=-1$, the corresponding connection is also known as the \emph{Bismut connection}, see \cite{Bis}, defined as the unique Hermitian connection with totally skew-symmetric torsion. Thus, one can define 
$$
{\rm Ric}^t(\omega):=-\frac12{\rm tr}(J R^{\nabla^t})\,.
$$
Then, $\omega$ is balanced if and only if ${\rm Ric}^t(\omega)={\rm Ric}^{ch}(\omega), $ for all $t\in \R$, see \cite[formula (8)]{GGP}. This result holds true even in the non-integrable case, see \cite[Corollary 3.3]{vez}.

From a  more analytic point of view, balanced metrics can be characterized by the coincidence of the various Laplacians operators on smooth functions. More precisely, 
$\omega$ is balanced if and only if 
\begin{equation}\label{laplacians}
\Delta_{\partial }=\Delta_{\bar \partial }=\frac12\Delta_d=-\Delta_{\omega}\,, \quad \mbox{ on } C^{\infty}(M, \R)\,,
\end{equation} where $\Delta_{\partial}:=[\partial, \partial^*]$, $\Delta_{\bar \partial }:=[\bar \partial , \bar \partial ^*]$, $\Delta_d$ is the standard Hodge Laplacian while $\Delta_{\omega}:={\rm tr}_{\omega}(i\partial \bar \partial)$, see \cite[Proposition 1]{G}. For general Hermitian metrics, \cite[Formula (4)]{Gaulaplace} ensures that 
\begin{equation}\label{chvshodge}
2\Delta_{\omega}f=-\Delta_df+ g(df, \theta_{\omega})\,, \quad  f\in C^{\infty}(M, \R)\,,
\end{equation} where $\theta_{\omega}$ is \emph{Lee form} of $\omega$ implicitly defined  by $d\omega^{n-1}=\theta_{\omega} \wedge \omega^{n-1}$. Recalling that the balanced condition is equivalent to ask $\theta=0$, we recover the last equality in \eqref{laplacians} in the balanced case.  \\
As K\"ahler metrics, balanced metrics define a natural cohomology class:
$$
[\omega^{n-1}]_{BC}\in H^{n-1,n-1}_{BC}(M, \R):=\frac{\ker \partial \cap \ker \bar \partial }{{\rm Im}\partial \bar \partial }\,.
$$
Mimicking  the nomenclature in the K\"ahler case, we will refer to $[\omega^{n-1}]_{BC}$ as the \emph{balanced class} of $\omega$.
 As a direct consequence of the definition of the balanced class, any other form $\alpha \in [\omega^{n-1}]_{BC}$ will be of the form 
\begin{equation}\label{gendef}
\alpha =\omega^{n-1}+ i \partial \bar \partial \varphi\,, \quad \varphi\in \Lambda_{\R}^{n-2, n-2}(M).
\end{equation} In particular, \eqref{gendef} highlights a striking difference from the K\"ahler case. Indeed, in the latter case, the  $\partial \bar \partial$-Lemma ensures that  all the  deformations of a fixed K\"ahler metric in its K\"ahler  class are fully determined by a smooth function, up to renormalization. On the opposite side, in the balanced case, we have much more room for deformations. In particular, in what follows, this will force us to consider ansatz when deforming balanced metrics within their balanced class.
 
Furthermore,  the total Chern scalar curvature, also known as the Gauduchon degree of the conformal class of $\omega$, \cite[I.17]{Gau}, of a balanced metric
\begin{equation}\label{gaudegree}
\Gamma(\{\omega\})=\int_Ms^{ch}(\omega)\frac{\omega^n}{n!}=\int_M{\rm Ric}^{ch}(\omega)\wedge \frac{\omega^{n-1}}{(n-1)!}=\frac{2\pi}{(n-1)!}c_1^{BC}(M)\cdot[\omega^{n-1}]_{BC}\,
\end{equation}   depends only on the balanced class of $\omega$ and on the topology of $M$. Moreover, as highlighted in \cite[Proposition 2.6]{ACS}, 
  the sign of the Chern scalar curvature of a balanced constant Chern scalar curvature metric $\omega$  must be equal to that of  $\Gamma(\{\omega\})$.

Additionally, we recall that, if  $\tilde M:={\rm Bl}_{p_1, \ldots, p_k}M$,  where  $p_1, \ldots, p_k\in M$, the following identity on the first Bott-Chern classes holds:
    \begin{equation}\label{BCclass}
    c_1^{BC}(\tilde M)=\pi^*c_1^{BC}(M)-(n-1)\sum_{i=1}^{k}[E_i]_{BC},
    \end{equation} where  we denoted with $\pi $ the blowdown map and  $[E_i]_{BC}$ is the Bott-Chern class of the Poincarè dual of the $(2n-2)$-homology class defined by the exceptional divisor of the blow-up of $M$ at $p_i$. Equivalently, we can define $[E_i]_{BC}$ as the first Bott-Chern class of the line bundle $\mathcal O_{\tilde M }(E_i)$ associated  to $E_i$ in $\tilde M$.

To conclude this preliminary part, we want to recall some basic ingredients concerning resolutions of orbifolds. 

Let $M$ be a complex orbifold. A map $\pi\colon \hM\to M$  is a resolution if $\hM$ is a complex manifold and $\pi$ is a proper biholomorphic map outside suitable dense subsets. The resolutions we will be concerned the most with are the so called \emph{crepant resolutions}.
\begin{defi}  
Let $M$ be a complex orbifold and $\pi \colon \hM \to M$ be a resolution. $\pi$ is called crepant if $K_{\hM}=\pi^*K_{M}$.
\end{defi}
Admitting crepant resolutions easily ensures that the orbifold group $G_x\subseteq {\rm SL}(n, \C)$. For $n=2,3$, the viceversa is also true, see for instance \cite[Theorem 6.6.1]{J},  but,   for $n\ge4 $, it does not hold in general. Besides having a nice expression for the canonical bundle, crepant resolutions  admit ALE Ricci-flat K\"ahler metrics with a suitable decay at infinity. More precisely,  we have the following.
\begin{thm}[Joyce, \cite{J}]\label{thmALE}
Let $G\subseteq {\rm SU}(n)$ be finite subgroup acting freely on $\C^{n}\backslash \{0\}$ and let $X$ be a crepant resolution of $\C^n/G$. Then, there exists a Ricci-flat K\"ahler metric $\omega_{{\rm ALE}}$ such that, away from the singularity, 
$$
\omega_{{\rm ALE}}=\omega_o-Ai\partial \bar \partial (r^{2-2n}+ {\rm o} (r^{2-2n}))\,,
$$
where $A>0$ is a constant, $r$ is the flat distance from the singularity while $\omega_o$ is the flat metric on $\C^n/G$.
\end{thm}
We refer to \cite[Theorem 8.2.3]{J} for the precise statement and the proof of  Theorem \ref{thmALE}.  
 We conclude this subsection  remarking that the better asymptotic behaviour of the Joyce's  ALE  Ricci-flat K\"ahler metric on crepant resolutions is the key ingredient to  prove Theorem \ref{glueorb}. 
\subsection{The approximate solution}\label{apprsol}

Let now $(M, \tilde \omega)$ be a compact Chern-Ricci flat manifold of dimension $n\geq 3$ and let $\hM$ be the blow-up at a point $x \in M$ (for simplicity, we will focus in the case of the  blow-up of a point, but the argument applies in the same way when blowing up a finite family of points). Following the usual strategy of gluing constructions (see \cite{AP} or \cite{Sz}), the first step will be to construct an \textit{approximate solution} to the problem on $\hM$, which in our case will consist of an \textit{approximately  constant Chern scalar curvature balanced metric}. In order to do this, we shall implement the cut-off argument for balanced metrics introduced in \cite{GS}. With this at hand,  we can glue together the background metric $\tilde{\omega}$ to the Burns-Simanca metric $\omega_{BS}$ on the  blow-up, using a flat region as bridge, making sure that it also satisfies suitable properties for the following deformation argument.\par 
Let us then start to describe this gluing process by seeing how the balanced property intervenes, and this happens with the following lemma from \cite{GS}.
\begin{lem}[Lemma 2.7, \cite{GS}]\label{flatcutoff}
Let  $(M^n,\tilde \omega)$ be  a balanced manifold. Then, for every $x \in M$ and $p>0$, there exist a sufficiently small $\varepsilon>0$, coordinates $z$ centred at $x$ and a balanced metric $\tilde\omega_\varepsilon$ such that
\[
\tilde\omega_\varepsilon=\begin{cases} \omega_o & \text{if} \>\> |z|<\varepsilon^p\, ; \\ \tilde \omega & \text{if} \>\> |z|> 2\varepsilon^p\, , \end{cases}
\]
where $\omega_o$ is the flat metric around $x$, and such that $|\tilde\omega_\varepsilon-\omega_o|_{\omega_o}<c\varepsilon^p$ on $\{\varepsilon^p \leq |z| \leq 2\varepsilon^p\}$.
\end{lem}
\begin{proof}
    For each point $x \in M$, we can choose coordinates $z$ around $x$ such that, in a sufficiently small neighbourhood of the point, it holds $\tilde{\omega}=\omega_o+O(|z|)$, and hence $\tilde{\omega}^{n-1}=\omega_o^{n-1}+\alpha$, with $\alpha$ a closed $(n-1,n-1)$-form. Thus, up to shrinking the neighbourhood, we can choose a real $(n-2,n-2)$ form $\beta$ such that $\alpha=i\p\bpa\beta$ such that $\beta=O(|z|^3)$. Thus, if we introduce a cut-off function
\[
\chi(y):=
\begin{cases} 
0 & \text{if} \>y\leq 1\, ; \\ 
\text{non decreasing} & \text{if} \> 1<y<2\, ; \\ 
1 & \text{if} \> y \geq 2\, 
\end{cases}
\]
and call $r(z):=|z|$ the (flat) distance from $x$, we can consider $\chi_\varepsilon(y):=\chi(y/\varepsilon)$ and define
\[
\tilde{\omega}_\varepsilon^{n-1}:=\omega_o^{n-1}+i\p\bpa(\chi_\varepsilon(r)\beta),
\]
which is easily seen to correspond to the metric of the statement.
\end{proof}
Thus, starting from $\tilde{\omega}$, we can obtain the corresponding $\tilde \omega_\ezi$, which is exactly flat in a neighbourhood of $x$.\\
On the other hand, we can consider the standard coordinates $\zeta$ on $\hX:={\rm Bl}_0\mathbb{C}^n\backslash E\simeq\mathbb{C}^n\backslash\{0\}=:X\setminus\{0\}$. Recalling that $\omega_{BS}$ admits the expansion \eqref{AF} away from the exceptional divisor, we can introduce a cut-off function
\[
\psi(y):=
\begin{cases} 
1 \qquad & \text{if} \> y\leq \frac{1}{4}; \\ 
\text{non increasing} \quad & \text{if} \> \frac{1}{4} < y < \frac{1}{2} ; \\ 
0 \qquad & \text{if} \> y\geq \frac{1}{2}, 
\end{cases}
\]
which, for all $q>0$, can be rescaled to
\[
\psi_\ezi(y):=\psi(\ezi^q y),
\]
making the cut-off happen \textit{far away} from the exceptional divisor (i.e. in the asymptotically flat part). This allows us to introduce the family of closed $(1,1)$-forms:
\[
\omega_{BS,\ezi}:=i\p\bpa(|\zeta|^2+\psi_\ezi(|\zeta|)\gamma(|\zeta|))\,.
\]
From here, it is easily seen that, on the cut-off region $\{\frac{1}{4}\ezi^{-q}\leq |\zeta| \leq \frac{1}{2}\ezi^{-q}\}$, it holds
\begin{equation}\label{cutoffbs}
\omega_{BS,\ezi}=\omega_o+O(|\zeta|^{2-2n})\,,
\end{equation}
where now $\omega_o$ denotes the flat metric on $\mathbb{C}^n\backslash \{0\}$ induced by the coordinates $\zeta$. This ensures that for sufficiently small $\ezi$, $\omega_{BS,\ezi}$ is an asymptotically exactly flat Kähler metric on $\hX$.\\
If we then consider the biholomorphism
\[
z=\ezi^{p+q}\zeta,
\]
we get the identification
\[
\left\{\frac{1}{4}\ezi^{-q}\leq |\zeta|\leq 2\ezi^{-q}\right\}\equiv \left\{\frac{1}{4}\ezi^p\leq |z|\leq 2\ezi^p\right\}
\]
 which guarantees that we can topologically realize $\hM$. Moreover, it also allows us to  obtain that
\[
|z|^2=\ezi^{2(p+q)}|\zeta|^2,
\]
 telling that on $\hM$, the metrics $\ezi^{2(p+q)}\omega_{BS,\ezi}$ and $\tilde \omega_\ezi$ coincide (with the flat metric), on the region
\begin{equation}\label{regione}
\left\{\frac{1}{2}\ezi^{-q}\leq |\zeta|\leq \ezi^{-q}\right\}\equiv \left\{\frac{1}{2}\ezi^p\leq |z|\leq \ezi^p\right\}\,.
\end{equation}
This last fact gives us the possibility to glue $\tilde \omega_\ezi$ and $\ezi^{2(p+q)}\omega_{BS,\ezi}$ to a global balanced metric $\omega_\ezi$ on $\hM$. In the following, as it will not create any confusion, we will avoid writing explicitly the dependence on $\ezi$, thus we will always write just $\omega=\omega_\ezi$.\par
\begin{rmk}
The metric $\omega$ is a suitable approximate solution. Indeed, it is clear that the metric is unaltered on $\{\varepsilon^p \leq |z| \leq 2\varepsilon^p\}$, on which we still have
\[
|\nabla_{\omega_o}^k (\omega -\omega_o)|_{\omega_o}\leq c|z|^{1-k},
\]
for all $k \geq 0$.\\
On the other hand, since to obtain $\omega$ we had to rescale the metric $\omega_{BS,\ezi}$ on $\hX$, we have to check how it has affected the distance from the flat metric. To have clearer estimates, we will express also this one in terms of the \lq\lq small\rq \rq\, coordinates $z$. The main thing to observe is that, on $\{\frac{1}{4}\ezi^{-q}\leq |\zeta| \leq \frac{1}{2}\ezi^{-q}\}$, it holds
\[
\begin{aligned}
\langle \omega-\omega_o,\omega-\omega_o\rangle_{\omega_o}(z)= & \> \ezi^{-4(p+q)}\langle \ezi^{2(p+q)}(\omega_{BS,\ezi}-\omega_o),\ezi^{2(p+q)}(\omega_{BS,\ezi}-\omega_o)\rangle_{\omega_o}(\zeta) \\
= & \>\langle \omega_{BS,\ezi}-\omega_o,\omega_{BS,\ezi}-\omega_o\rangle_{\omega_o}(\zeta)\,,
\end{aligned}
\] 
implying that $|\omega-\omega_o|_{\omega_o}(z)=|\omega_{BS,\ezi}-\omega_o|_{\omega_o}(\zeta)$. From here, we can recall the expansion \eqref{cutoffbs} and obtain
\[
\begin{aligned}
|\omega-\omega_o|_{\omega_o}(z)\leq & |\omega_{BS,\ezi}-\omega_o|_{\omega_o}(\zeta)  
 \leq  c|\zeta|^{2-2n} \leq c\varepsilon^{(2n-2)q} \leq c |z|^{(2n-2)q/p}\,,
\end{aligned}
\]
which implies, on the whole gluing region, that, for all $k \geq 0$,  it  holds
\[
|\nabla_{\omega_o}^k (\omega - \omega_o)|_{\omega_o} \leq c|z|^{m-k},
\]
where $m=\min\{1, (2n-2)q/p\}$, showing again that $\omega$ is indeed a metric on $\hM$. Moreover, the closeness between the metric $\omega$ and the flat metric $\omega_o$ shows us that $\omega$ is suitable to perform analysis with, and hence we can try to search for a  constant Chern scalar curvature balanced metric through a deformation argument.
\end{rmk}

\begin{rmk}
As we will see, the deformation argument we will introduce in the following subsection will allow us to work within the balanced class of $\omega$. Hence, in light of \cite[Proposition 2.6]{ACS}, we can predict the sign of the Chern scalar curvature of a genuine solution (which we will obtain in the next sections) by describing the cohomology class of $\omega$.
Indeed, we have that
    \[
    [\omega^{n-1}]_{BC}=\pi^*[ \tilde{\omega}^{n-1}]_{BC}+[\ezi^{2(p+q)}\omega_{BS}]_{BC}^{n-1},
    \]
    and since $[\omega_{BS}]_{BC}=-[E]_{BC}$, we get
    \[
    [\omega^{n-1}]_{BC}=\pi^*[ \tilde{\omega}^{n-1}]_{BC}+(-1)^{n-1}\ezi^{(2n-2)(p+q)}[E]_{BC}^{n-1}.
    \]
Now, recalling \eqref{gaudegree},  \eqref{BCclass}  and that $[E]_{BC}^n=(-1)^{n-1}$,   it follows that
  \begin{equation}\label{ilmentiroso}
    \begin{aligned}
    \Gamma(\{\omega\})= &\,\frac{2\pi}{(n-1)!} (\pi^*c_1^{BC}(M)-(n-1)[E]_{BC})\cdot(\pi^*[ \tilde{\omega}^{n-1}]_{BC}+(-1)^{n-1}\ezi^{(2n-2)(p+q)}[E]_{BC}^{n-1})\\
    = &\,  \Gamma(\{\tilde{\omega}\})-\frac{2\pi}{(n-2)!}\ezi^{(2n-2)(p+q)}.
    \end{aligned}
\end{equation}
\end{rmk}

\subsection{Setting up the equation}\label{seceq}
We now wish to obtain a  constant Chern scalar curvature balanced metric starting from the approximate solution,  as done in \cite{BM}, \cite{AP}, \cite{Sz} and many others. We plan to do it through a deformation argument. Since we wish to work inside the balanced class of $\omega$, we will consider the general deformation (considered first by \cite{FWW}):
\begin{equation}\label{ansatz}
\omega_\varphi^{n-1}:=\omega^{n-1}+i\partial\bar {\partial}\varphi, \quad \varphi \in \Lambda_{\R}^{n-2, n-2}(\hM) \>\> \text{such that} \>\> \omega_\varphi^{n-1} >0.
\end{equation}
Thus, the problem we are interested in solving, following what was done in \cite{Sz}, is the equation
\begin{equation}\label{eqs1}
s^{ch}(\omega_\varphi)=const.
\end{equation}
for $\varphi \in \Lambda_{\R}^{n-2, n-2}(\hM)$ such that $\omega_\varphi^{n-1} >0$.
Now, as showed by \eqref{ilmentiroso}, we can expect the solution to have the Chern scalar  curvature near to the one of $\tilde \omega$, thus we can rephrase equation \eqref{eqs1} as:
\begin{equation}\label{eqs2}
\mathcal{S}(\varphi):=s^{ch}(\omega_\varphi)-s^{ch}( \tilde{\omega})=c
\end{equation}
for $\varphi\in \Lambda_{\R}^{n-2, n-2}(\hM)$ and $c\in \mathbb R$. Moreover, we can get rid of the unknown constant by rewriting the equation as:
\begin{equation}\label{eqs3}
\tilde{\mathcal{S}}(\varphi):=s^{ch}(\omega_\varphi)-s^{ch}( \tilde{\omega})-\int_{\hM} f(\varphi)\frac{\omega^n}{n!}=0,
\end{equation} where $f\colon \Lambda_{\R}^{n-2, n-2}(\hM)\to C^{\infty}(\hM, \R)$ is a suitable operator to be evaluated on $\varphi$ which will be chosen later to help us to get rid of some operator's kernel. 
This last version of the equation encodes the unknown constant from equation \eqref{eqs2}. This will help us in obtaining the invertibility of the linearization of $\tilde{\mathcal{S}}$, which is a key ingredient to allow us to turn the problem of solving equation \eqref{eqs3} into a fixed point problem to be solved with Banach's fixed-point Theorem in a  suitably chosen neighbourhood of zero. Hence, our next step will be to compute the linearization at $0$ of $\tilde{\mathcal{S}}$.\par
First of all, we observe that, directly from \eqref{eqs3},  the linearized operator has the following form:
$$
\tilde {\mathcal L}(\varphi):=d_0\tilde {\mathcal S}(\varphi)=\mathcal L (\varphi)- \int_{\hM }d_0f(\varphi)\frac{\omega^n}{n!}\,, \quad \varphi\in \Lambda_{\R}^{n-2, n-2}(\hM)\,,
$$ where $\mathcal L(\varphi):=d_0s^{ch}(\varphi)$.
We thus need to obtain an explicit expression for the operator $\mathcal L(\varphi)=\frac{d}{dt}\big|_{t=0}s^{ch}(\omega_{t, \varphi})$, where $\omega_{t, \varphi}$ is an arbitrary curve of Hermitian metrics lying in $[\omega^{n-1}]_{BC}$ and   such that $\omega^{n-1}_{0,\varphi}=\omega^{n-1}$ and $(\omega_{t, \varphi}^{n-1})'(0)=\varphi$. 
Thus, we consider the curve of Hermitian metrics defined by $\omega_{t, u}^{n-1}=\omega^{n-1}+ ti\partial \bar \partial \varphi$   and we observe that 
  \begin{equation}\label{gigiKahler}
  \begin{aligned}
\frac{d}{dt}\Big|_{t=0}\omega_{t, \varphi}^{n}=&\, n \frac{d}{dt}\Big|_{t=0}\omega_{t, \varphi}\wedge \omega^{n-1}\,,\\
\frac{d}{dt}\Big|_{t=0}\omega_{t, \varphi}^n=&\, \frac{d}{dt}\Big|_{t=0}\omega_{t, \varphi}\wedge \omega^{n-1}+ \omega\wedge i\partial \bar\partial \varphi\,.
\end{aligned}
\end{equation} Then, from \eqref{gigiKahler},  we obtain 
\begin{equation}\label{volumeAK}
\frac{d}{dt}\Big|_{t=0}\omega_{t, u}^n=\frac{n}{n-1}\omega\wedge i \partial \bar \partial \varphi\,.
\end{equation} 
For the sake of simplicity, we will denote with 
$$
F_{\omega}(\varphi):=\frac{d}{dt}\Big|_{t=0}\log\omega_{t, \varphi}^n=\frac{n}{n-1}\frac{\omega \wedge i \partial \bar\partial \varphi}{\omega^n}\,.
$$
From this, we easily obtain that
\begin{equation}\label{ricciAK}
\frac{d}{dt}\Big|_{t=0}{\rm Ric}^{ch}(\omega_{t, \varphi})=-i\partial \bar \partial F_{\omega}(\varphi)\,.
\end{equation}
Now, differentiating \eqref{scal} and using \eqref{volumeAK} and \eqref{ricciAK}, we obtain that
\begin{equation}\label{pippoL}
\mathcal L(\varphi)=- \Delta_{\omega}F_{\omega}(\varphi)+ n \frac{{\rm Ric}^{ch}(\omega)\wedge i \partial \bar \partial\varphi }{\omega^n}- s^{ch}(\omega)F_{\omega}(\varphi)\,.
\end{equation}
Clearly such an operator will not stand the possibility to have no kernel, as the input space is much larger than the target space, on top of being extremely complicated to understand.
For this reason, in the next three sections, we will introduce two different ansatz  with the objective  to turn the operator into one between two function spaces, allowing us to perform the usual analysis involved in gluing constructions.

\section{Balanced deformation}\label{baldef}
In this section we  consider  the balanced  ansatz for the deformation argument.
In particular, we  analyse the equation  \eqref{eqs1} assuming $\varphi=u\omega^{n-2}$, for  $u\in C^{\infty}(M, \R)$, as previously done in \cite{GS}, from which we get (as from \cite[Lemma 3.2]{GS}) that the operator $F_{\omega}$ takes the form
\begin{equation}\label{opbalanced}
F_{\omega}(u)=\frac{1}{n-1}\left(\Delta_{\omega}u + \frac{1}{n-1}|\partial \omega|^2 u \right)\,.\end{equation}
We will then choose $f(u\omega^{n-2})=u|\partial \omega|^2$. With these choices,  we are able to turn the operator $\tilde{\mathcal S}$ into an operator taking smooth functions in input defined as:
\begin{equation}\label{opfunctions1}
\tilde{\mathcal{S}}(u)=s^{ch}(\omega_u)-s^{ch}( \tilde{\omega})-\int_{\hM} u|\partial \omega|^2\frac{\omega^n}{n!},
\end{equation}
whose linearization is now suitable to be inverted, up to working in the correct functional spaces.
Let us conclude by writing again the linearized operator $\tilde{\mathcal{L}}$ implementing this  ansatz:
\begin{equation}\label{pippoLbal}
    \tilde{\mathcal L}(u):=\tilde{\mathcal L}(u \omega^{n-2})=-\Delta_\omega F_{\omega}(u)+n\frac{\text{Ric}^{ch}(\omega)\wedge i\p\bpa(u\omega^{n-2})}{\omega^n}-s^{ch}(\omega)F_{\omega}(u)- \int_{\hM}u|\partial \omega|^2\frac{\omega^n}{n!}\,.
\end{equation}

Before discussing the proof the main Theorem, we want to highlight and discuss the differences and similarities of this setting with the K\"ahler one.
\subsection{Comparison with the Kähler case}
As highlighted before, this deformation makes sense on every balanced manifold, hence it is worth analysing the linearized operator in a more general setting, aiming to understand something more about it in the case of constant Chern scalar curvature balanced metrics. As we will see, this linearization will show up as a perfectly fitting generalization of the Kähler case, as it will reduce to the Lichnerowicz operator whenever the metric on the base manifold is chosen to be cscK. In order to obtain this, we will first recall some ingredients from the Kähler setting, then obtain some significant formulas for the balanced setting, and finally put all the pieces together.\par

First of all, we recall that from the case of \cite{AP} for cscK metrics on blow-ups, a key role is played by the Lichnerowicz operator:
$$
\mathcal D^*\mathcal D\colon C^{\infty}(M, \C)\to C^{\infty}(M, \C)\quad \mbox{defined locally by } \mathcal D^*\mathcal Du:=g^{j\bar k }g^{p\bar q }\nabla_p\nabla_j\nabla_{\bar k }\nabla_{\bar q} u\,,
$$ where here $\nabla $ is the Levi-Civita connection of a given K\"ahler metric $g$. More globally, using the identity known as \textit{contracted second Bianchi identity}, we can write, see for instance \cite[Definition 4.3]{Sz},  
$$
\mathcal D^*\mathcal Du=\Delta_{\omega}^2u+ g(i\partial \bar \partial u, {\rm Ric}(\omega)) + g(\partial s(\omega), \partial \bar u)\,,\quad u\in C^{\infty}(M, \C)\,,
$$ where ${\rm Ric}(\omega)$ is the Ricci form of the metric $g$ while $s(\omega)={\rm tr}_{\omega}{\rm Ric}(\omega)$ is the Riemannian scalar curvature of $g$. The kernel of the Lichnerowicz operator is well known and it consists on those function $u$ such that 
$
(\nabla u)^{1,0}$ is a holomorphic vector field.\par
 Now, let $ \omega$ be a balanced manifold. Following the K\"ahler case, we would like to consider the operator $$
 \mathcal D:=\nabla_{\bar k}\nabla_{\bar q }\colon C^{\infty}(M, \mathbb R)\to \Lambda^{0,1}(M)\otimes \Lambda^{0,1}(M)\,,$$ where $\nabla $ is the Chern connection of $\omega$ (which is now different from the Levi-Civita connection), and compute, given  $u\in C^{\infty}(M, \C)$, the Chern-Lichnerowicz operator $\mathcal D^*\mathcal D$. As the contracted second Bianchi identity are only satisfied by the Levi-Civita connection, in order to obtain a reasonable expression for the operator, we are going to develop some formulas for the codifferentials of the Chern-Ricci forms of a given balanced metric $\omega$. We will then  see some interesting applications of them, not only in direct relation to our problem, but in the wider framework of constant Chern scalar curvature  balanced metrics.\par

To make things clear, let us quickly establish some local coordinates conventions. We will locally write any Hermitian metric $\omega$ as$$
\omega= \frac i2g_{i\bar j } dz^i\wedge d\bar z^j\,,
$$ where, as usual, $g_{i\bar j }=g(\frac{\partial }{\partial z_i}, \frac{\partial }{\partial \bar z_j})$ are the components of the Hermitian metric $g$ associated to $\omega$.  As we know, the presence of a Riemannian metric determines a natural metric on all  tensor bundles. In what follows, we will make use of that on  differential forms. Encoding directly the complex structure in the discussion, we will use the following convention: given $\alpha, \beta\in \Lambda^{p,q}(M)$, then, in local holomorphic coordinates, we have that 
$$
g(\alpha, \beta)=\frac{1}{p!q!}g^{i_1\bar j_1}\cdots g^{i_p\bar j_p}g^{k_1\bar l_1}\cdots g^{k_q\bar l_q}\alpha_{i_1\cdots i_p\bar l_1\cdots \bar l_q}\overline{\beta_{j_1\cdots j_p\bar k_1\cdots \bar k_q}}\,.
$$ Instead of using local coordinates, we can equivalently define the above inner product using the Hodge $*$-operator as follows:
$$
\alpha\wedge *\bar \beta=g(\alpha, \beta)\frac{\omega^{n}}{n!}\,. 
$$
 With these conventions, we can recall some well known 
 Riemannian relations we will use later, specified in the Hermitian case. The first one is the fact  that the interior product and the exterior one are formal adjoints with respect to the inner product above. More precisely, given $Z\in T^{1,0}M$, $\alpha\in \Lambda^{p,q}(M)$ and  $\beta\in \Lambda^{p-1, q}(M)$ we have that 
\begin{equation}\label{formulazza1}
g(\iota_Z\alpha, \beta)=g(\alpha, \bar Z^{\flat}\wedge \beta)\,,
\end{equation} where $Z^{\flat}(W)=g(Z, W)$,  for all $W\in T^{0,1}M$. 
From this equality, we derive the following: 
\begin{equation}\label{formulazza2}
\iota_{\bar Z}\beta=*(\bar Z^{\flat}\wedge *\beta) \quad \mbox{and } \quad \iota_Z\beta=*(Z^{\flat}\wedge *\beta)\,.
\end{equation}

We then have the following formulas, which can be interpreted as  contracted second Bianchi identities for the Chern connection.

\begin{lem}\label{lemmarigi}
Let $(M^n, \omega)$ be a compact  balanced manifold. We have,
$$
\begin{aligned}
\bar \partial^*{\rm Ric}^{ch}(\omega)=&\, i\partial s^{ch}(\omega)-\frac i2 \Lambda^2({\rm Ric}^{ch}(\omega)\wedge \partial\omega)\,,\\
\bar \partial^*{\rm Ric}^{(2)}(\omega)=&\, i\partial s^{ch}(\omega)-\frac i2 \Lambda^2({\rm Ric}^{(2)}(\omega)\wedge \partial\omega)-i\Lambda(\partial {\rm Ric}^{(2)}(\omega))\,.
\end{aligned}
$$
As a consequence, we have that 
$$
\begin{aligned}
i\partial^*\bar \partial^*{\rm Ric}^{ch}(\omega)=&\, \Delta_{\omega}s^{ch}(\omega)-g({\rm Ric}^{ch}(\omega), \partial^*\partial \omega)\,,\\
i\partial ^*\bar \partial^*{\rm Ric}^{(2)}(\omega)=&\, \Delta_{\omega}s^{ch}(\omega)-\frac12\Lambda^2(i\partial \bar \partial {\rm Ric}^{(2)}(\omega))+ 2 \Re(g(\partial {\rm Ric}^{(2)}(\omega), \partial \omega))- g({\rm Ric}^{(2)}(\omega),\partial^*\partial \omega )\,.
\end{aligned}
$$
\end{lem}
\begin{proof}
First of all, using \cite[Lemma 2.1]{BV}, we know that  
\begin{equation}\label{starrig}
*{\rm Ric}^{ch}(\omega)=s^{ch}(\omega)\frac{\omega^{n-1}}{(n-1)!}-{\rm Ric}^{ch}(\omega)\wedge \frac{\omega^{n-2}}{(n-2)!}\,.
\end{equation}
Now, using the balanced condition, we have that 
$$
\partial *{\rm Ric}^{ch}(\omega)=\partial s^{ch}(\omega)\wedge *\omega-{\rm Ric}^{ch}(\omega)\wedge \frac{\partial \omega^{n-2}}{(n-2)!}\,.
$$
Applying \eqref{formulazza2},  for any $Z\in T^{1,0}M$, we can infer that
$$
\begin{aligned}
\iota_Z\bar \partial^*{\rm Ric}^{ch}(\omega)=&\, *(Z^{\flat}\wedge  \partial*{\rm Ric}^{ch}(\omega))
= *(Z^{\flat}\wedge  \partial s^{ch}(\omega)\wedge *\omega)-*\left(Z^{\flat}\wedge {\rm Ric}^{ch}(\omega)\wedge \frac{\partial \omega^{n-2}}{(n-2)!}\right)\,.
\end{aligned}
$$ On the other hand, using  \cite[Proposition 1.2.31]{Huy}, we obtain 
$$
*\partial s^{ch}(\omega)=-i\partial s^{ch}(\omega)\wedge \frac{\omega^{n-1}}{(n-1)!}=-i\partial s^{ch}(\omega)\wedge *\omega\,, 
$$ which gives us that 
$$
*(Z^{\flat}\wedge  \partial s^{ch}(\omega)\wedge *\omega)=i\iota_Z\partial s^{ch}(\omega)\,.
$$
 Now, thanks to the balanced condition, we have that  $\partial \omega $ is primitive. So, making use again of \cite[Proposition 1.2.31]{Huy}, we can easily see that 
\begin{equation}\label{startors}
*\partial \omega=\frac{i\partial \omega^{n-2}}{(n-2)!}\,.
\end{equation} Then, 
\begin{equation}\label{nonsapre}
 \begin{aligned}
Z^{\flat}\wedge {\rm Ric}^{ch}(\omega)\wedge \frac{\partial \omega^{n-2}}{(n-2)!}=&\, -iZ^{\flat}\wedge {\rm Ric}^{ch}(\omega)\wedge *\partial\omega 
=-ig( {\rm Ric}^{ch}(\omega), \iota_{\bar Z}\bar \partial \omega)\frac{\omega^n}{n!}\,.
\end{aligned}
\end{equation}
On the other hand,  one can easily check that, for any $Z\in T^{1,0}M$, we have:
\begin{equation}\label{assurdooo}
 \begin{aligned}
\frac12\iota_Z\Lambda^2({\rm Ric}^{ch}(\omega)\wedge  \partial\omega)=
*\left( {\rm Ric}^{ch}(\omega)\wedge\iota_Z  \partial\omega\wedge \frac{\omega^{n-2}}{(n-2)!}\right)=-g({\rm Ric}^{ch}(\omega), \iota_{\bar Z}\bar \partial \omega)\,,
\end{aligned}
\end{equation}where we used that  
$$
*\iota_Z\partial \omega=-\iota_Z\partial \omega\wedge \frac{\omega^{n-2}}{(n-2)!}\,.
$$
 Hence,  we obtain the first claim.

 As regards the formula for the second Chern-Ricci form, since an analogue of \eqref{starrig} holds for ${\rm Ric}^{(2)}(\omega)$, we just need to analyse the term involving $\partial {\rm Ric}^{(2)}(\omega)\wedge \omega^{n-2}$.
 Applying again \cite[Proposition 1.2.31]{Huy}, we have that 
 \begin{equation}\label{gigiL}
*LZ^{\flat}=iZ^{\flat}\wedge\frac{\omega^{n-2}}{(n-2)!}\,, 
\end{equation} which implies that 
$$
*\left(Z^{\flat}\wedge \partial {\rm Ric}^{(2)}(\omega)\wedge \frac{\omega^{n-2}}{(n-2)!}\right)=ig(\partial{\rm Ric}^{(2)}(\omega), L\bar Z^{\flat})=i\iota_Z\Lambda(\partial {\rm Ric}^{(2)}(\omega))\,,
$$ giving the desired formula.

As regards the second part of the statement,  we know that 
 $$
i\partial^*\bar \partial^*{\rm Ric}^{ch}(\omega)=i*\partial \bar \partial *{\rm Ric}^{ch}(\omega)\,.
$$
Now, recalling \eqref{starrig},  using the balanced condition and the fact that $d{\rm Ric}^{ch}(\omega)=0$, we have that
\begin{equation}\label{ipprigi}
\begin{aligned}
i\partial^*\bar \partial^*{\rm Ric}^{ch}(\omega)=&\,*\left(i\partial \bar \partial s^{ch}(\omega)\wedge \frac{\omega^{n-1}}{(n-1)!}- {\rm Ric}^{ch}(\omega)\wedge \frac{i\partial \bar \partial\omega^{n-2}}{(n-2)!}\right)\\
=&\, \Delta_{\omega}s^{ch}(\omega)-*\left({\rm Ric}^{ch}(\omega)\wedge \frac{i\partial \bar \partial\omega^{n-2}}{(n-2)!}\right)\,.
\end{aligned}
\end{equation} Hence, the claim follows by observing that 
\begin{equation}\label{ricvspartstar}
\partial ^*\partial \omega=*\frac{i\partial \bar \partial \omega^{n-2}}{(n-2)!}\,.
\end{equation}
\\Finally, using that 
$$
*{\rm Ric}^{(2)}(\omega)=s^{ch}(\omega)  \frac{\omega^{n-1}}{(n-1)!}- {\rm Ric}^{(2)}(\omega)\wedge \frac{\omega^{n-2}}{(n-2)!}\,,
$$ we obtain 
\begin{equation}\label{nonloso}
\begin{aligned}
i\partial^*\bar \partial^*{\rm Ric}^{(2)}(\omega)=&\, *\left( i\partial \bar \partial s^{ch}(\omega)\wedge \frac{\omega^{n-1}}{(n-1)!}-i\partial \bar \partial\left({\rm Ric}^{(2)}(\omega)\wedge \frac{\omega^{n-2}}{(n-2)!}\right)\right)\\
=&\, \Delta_{\omega}s^{ch}(\omega)- *i\partial \bar \partial\left({\rm Ric}^{(2)}(\omega)\wedge \frac{\omega^{n-2}}{(n-2)!}\right)\,.
\end{aligned}
\end{equation}
 Moreover, we have that 
 $$
i\partial \bar \partial({\rm Ric}^{(2)}(\omega)\wedge \omega^{n-2})=i\partial \bar \partial {\rm Ric}^{(2)}(\omega)\wedge \omega^{n-2}+ 2\Re(i\partial {\rm Ric}^{(2)}(\omega)\wedge \bar \partial \omega^{n-2})+ {\rm Ric}^{(2)}(\omega)\wedge i\partial \bar \partial \omega^{n-2}\,.
$$
 Now, using again \eqref{ricvspartstar}, we can conclude that 
 $$
*\left({\rm Ric}^{(2)}(\omega)\wedge \frac{i\partial \bar \partial \omega^{n-2}}{(n-2)!}\right)= - g({\rm Ric}^{(2)}(\omega), \partial ^*\partial \omega)\,.
$$
 Furthermore, one can check that  
 $$
i\partial {\rm Ric}^{(2)}(\omega) \wedge \frac{\bar \partial \omega^{n-2}}{(n-2)!}=- \partial {\rm Ric}^{(2)}(\omega)\wedge *\bar \partial \omega =-g(\partial {\rm Ric}^{(2)}(\omega), \partial \omega)\,.
$$
 Then,
 $$
-*i\partial \bar \partial \left({\rm Ric}^{(2)}(\omega)\wedge \frac{\omega^{n-2}}{(n-2)!}\right)=-\frac12\Lambda^2(i\partial \bar \partial {\rm Ric}^{(2)}(\omega))+ 2 \Re(g(\partial {\rm Ric}^{(2)}(\omega), \partial \omega))- g({\rm Ric}^{(2)}(\omega),\partial^*\partial \omega )\,,
 $$
  which inserted in \eqref{nonloso} concludes  the proof.
\end{proof}
This result has several very interesting consequences which we shall now discuss (up to obtaining a nice expression for the Chern-Lichnerowicz operator). The very first formula in Lemma \ref{lemmarigi}, for example, guarantees a new way of proving a well known fact, i.e. that balanced first Chern-Ricci Einstein metrics are either flat or  K\"ahler-Einstein, see \cite{ACS2}.
\begin{corollario}
Let $(M^n, \omega)$ be a compact first Chern-Einstein balanced manifold. Then, $\omega$  is either  Chern-Ricci flat or  K\"ahler-Einstein. 
\end{corollario}
\begin{proof}
The first Chern-Einstein condition together  with the balanced condition ensures that 
$$
\bar \partial^*{\rm Ric}^{ch}(\omega)= i\partial s^{ch}(\omega)\,.
$$
 On the other hand,  if ${\rm Ric}^{ch}(\omega) = \lambda \omega$, for some $\lambda \in C^{\infty}(M, \R)$, we have that
 $$
\bar \partial^*{\rm Ric}^{ch}(\omega)=\bar \partial ^*(\lambda \omega)=-*\left(\partial\lambda \wedge \frac{\omega^{n-1}}{(n-1)!}\right)=i \partial \lambda\,.
$$ Then, we have that 
$$
i\partial \lambda =i \partial s^{ch}(\omega)= n i\partial \lambda
$$ which implies that $\lambda $ is constant. So, in the case in which $\lambda\ne 0 $ we  conclude the proof, using that ${\rm Ric}^{ch}(\omega)$ is $d$-closed. 
\end{proof}
Another consequence of Lemma \ref{lemmarigi} is that, on a compact balanced manifold $(M, \omega)$, the term $g({\rm Ric}^{ch}(\omega), \partial^*\partial \omega)$ appearing in the expression of $i\partial^*\bar \partial^*{\rm Ric}^{ch}(\omega)$ is such that 
 $$
 \int_{M}g({\rm Ric}^{ch}(\omega), \partial^*\partial \omega)\frac{\omega^n}{n!}=0\,.
$$ 
The above condition directly implies that if $(M, \omega )$ is a compact quotient of a Lie group endowed with a left-invariant balanced metric, then ${\rm Ric}^{ch}(\omega)$ is Bott-Chern harmonic, i.e.
$$
\partial {\rm Ric}^{ch}(\omega)=0\,, \quad \bar \partial {\rm Ric}^{ch}(\omega)=0 \,\, \mbox{ and } i \partial^*\bar \partial^*{\rm Ric}^{ch}(\omega)=0\,.
$$
On the other hand, a well known result, following directly from the famous K\"ahler identities, states that  a K\"ahler metric has harmonic Ricci form if and only if the metric  is cscK. Then, one can formulate the following problem:
\begin{prob}\label{probharmcscb}
Let $(M, \omega)$ be a compact balanced manifold. The first Chern-Ricci form of $\omega$ is Bott-Chern harmonic if and only if $\omega$ has constant Chern scalar curvature. 
\end{prob}
\begin{rmk}
In general, we cannot expect the first Chern-Ricci form to be $d$-harmonic. Indeed, in \cite{GiP}, the authors constructed  balanced compact quotient of Lie groups with vanishing first Chern class but non-zero Chern-Ricci form. Then, if the Chern-Ricci form was $d$-harmonic, it would readily imply that it is zero, which is not possible.
\end{rmk}

Finally, we provide an explicit expression for $\mathcal L $  which will allow us to compare it to the Kähler case. 
Recall that the operator $\mathcal{L}$ is given by
$$
    \mathcal L(u)=-\Delta_\omega F_\omega(u)+n\frac{\text{Ric}^{ch}(\omega)\wedge i\p\bpa(u\omega^{n-2})}{\omega^n}-s^{ch}(\omega)F_\omega(u)\,,\quad 
u \in C^{\infty}(M, \R)\,.
$$
\begin{prop}\label{thmveroL}
Let $(M^n,\omega)$ be a compact balanced manifold. Then, for all $u\in C^{\infty}(M, \R)$, 
\begin{equation}\label{veroL}
 \begin{aligned}
 \mathcal{L}( u)=&\,  -\frac{1}{n-1}\left(\Delta_{\omega}^2u+ g(i\partial \bar \partial u,  {\rm Ric}^{ch}(\omega))+ \frac{1}{n-1}(\Delta_{\omega}+ s^{ch}(\omega){\rm Id})(\lvert \partial\omega\rvert^2 u) \right)\\
 &\, - \frac{1}{n-1}\left(-\Re(g(i \partial u , i \Lambda^{2}({\rm Ric}^{ch}(\omega)\wedge \partial \omega))- u g({\rm Ric}^{ch}(\omega), \partial^*\partial\omega)\right)\,.
 \end{aligned}\end{equation} 
 In particular, if $\omega$ is a constant Chern scalar curvature balanced metric, then 
 $$
 \begin{aligned}
\mathcal{L}( u)=&\,  -\frac{1}{n-1}\left(\Delta_{\omega}^2u+ g(i\partial \bar \partial u,  {\rm Ric}^{ch}(\omega))+2\Re(g(\bar \partial ^*{\rm Ric}^{ch}(\omega), i \partial u ) \right)\\
 &\, - \frac{1}{n-1}\left(\frac{1}{n-1}(\Delta_{\omega}+ s^{ch}(\omega){\rm Id})(\lvert \partial\omega\rvert^2 u)+ ui\partial^*\bar \partial^*{\rm Ric}^{ch}(\omega)\right)\,.
 \end{aligned}
$$
Moreover, if $\omega$ is cscK, 
$$
\mathcal L(u)=-\frac{1}{n-1}\mathcal D^*\mathcal Du\,.
$$
\end{prop}
\begin{proof}
 We just need to analyse the term
 $$
n\frac{\text{Ric}^{ch}(\omega)\wedge i\p\bpa(u\omega^{n-2})}{\omega^n}\,.
$$
On the other hand, we have, using \cite[Lemma 4.7]{Sz}, that 
\begin{equation}\label{L1}
n\frac{\text{Ric}^{ch}(\omega)\wedge i\p\bpa u\wedge \omega^{n-2}}{\omega^n}=\frac{1}{n-1}(s^{ch}(\omega)\Delta_{\omega}u-g(i\partial \bar\partial u,{\rm Ric}^{ch}(\omega)))\,.
\end{equation}Moreover, one can check, repeating the computations in \eqref{nonsapre} and \eqref{assurdooo}, that 
\begin{equation}\label{L2}
n\frac{2\Re({\rm Ric}^{ch}(\omega)\wedge i \partial u\wedge \bar \partial \omega^{n-2} )}{\omega^n}=\frac{1}{n-1}\Re(g(i\partial u, i\Lambda^2({\rm Ric}^{ch}(\omega)\wedge \partial \omega)))\,.
\end{equation}Finally, using again \eqref{ricvspartstar}, we have that 
\begin{equation}\label{L3}
n\frac{u\text{Ric}^{ch}(\omega)\wedge i\p\bpa \omega^{n-2}}{\omega^n}=\frac{1}{n-1}ug({\rm Ric}^{ch}(\omega), \partial^*\partial \omega)\,.
\end{equation}
Summing \eqref{L1}, \eqref{L2} and \eqref{L3},  inserting the result in the expression of $\mathcal L$ and using \eqref{opbalanced}, we obtain \eqref{veroL}. The second assertion is straightforward using Lemma \ref{lemmarigi}. Finally, the last claim is again easily verifiable using the cscK condition. 
\end{proof}
Thus, we see that the operator $\mathcal{L}$ arising from the balanced deformation generalizes to the balanced case (with the choice of the Chern connection) the Lichnerowicz operator, giving further motivation to widen the understanding of the associated equation, and making it a good candidate to obtain in the future the result in full generality. 
\begin{rmk}
    It is very interesting to notice that, when starting with a Kähler metric $\tilde{\omega}$ and producing a Kähler pre-gluing metric, despite the fact that the deformation is not corresponding to a deformation in the Kähler class (indeed the balanced deformation does not preserve the Kähler condition in general), it produces an operator whose linearization is again the Lichnerowicz operator (up to a constant factor). This in particular shows that, along this ansatz, holomorphic vector fields appear again as an obstruction to successfully deform the pre-gluing metric to a genuine constant Chern scalar curvature balanced metric.
\end{rmk}

\subsection{Lichnerowicz-type operators on balanced manifolds}
In this subsection, we will study holomorphic vector fields on compact  balanced manifolds and another type of vector fields arising as the kernel of  a \emph{torsion-twisted Lichnerowicz-type} operator. We will see how the latter fits into the framework of constant Chern scalar curvature balanced manifolds.
 We  start by giving the expression of the Chern-Lichnerowicz operator for a balanced metric.
\begin{lem}\label{lemLich}
Let $(M^n, \omega)$ be a compact balanced manifold. Then, for any $ u  \in C^{\infty}(M, \R)$, we have that
$$
\mathcal D^*\mathcal D u =\Delta_{\omega}^2 u + g\left( {\rm Ric}^{ch}(\omega)-\frac12\Xi,  i\partial \bar \partial  u \right) + g\left(\bar \partial^*\left({\rm Ric}^{ch}(\omega)-\frac12\Xi\right), i\partial  u \right)\,,
$$ where the $(1,1)$-form $\Xi$ is defined as in \eqref{Q2malefico}.
Moreover, $u\in \ker \mathcal D^*\mathcal D$ if and only if 
\begin{equation}\label{kerlichforme}
\bar \partial \Delta_{\bar \partial }u+ \iota_{i(\bar \partial u )^{\sharp}}{\rm Ric}^{ch}(\omega)=0\,.
\end{equation}
In particular, if ${\rm  Ric}^{ch}(g)\le 0$, then $\ker \mathcal D^*\mathcal D=\R$. 
\end{lem}
\begin{proof}\label{proofLich}
    In order to have the explicit expression of $\mathcal D^*$ we observe that, if $\beta\in \Lambda^{0,1}(M)\otimes \Lambda^{0,1}(M)$, 
\begin{equation}\label{iddelirio}
\begin{aligned}
g^{i \bar j }g^{l \bar k}\nabla_{\bar j }\nabla_{\bar k } u \bar\beta_{il}
=&\, -{\rm tr}_{\omega}(\bar \partial\nabla^*( u \bar \beta) )+ 2 {\rm tr}_{\omega}(\bar \partial ( u \nabla^*\bar \beta))+  u  g^{i\bar j }g^{l\bar k}\nabla_{\bar j }\nabla_{\bar k }\bar\beta_{il}\,,\end{aligned}
\end{equation}
 where $(\nabla^*\bar \beta)_{m}=-g^{p\bar q }\nabla_{\bar q}\bar \beta_{pm}$. Then, using the balanced condition, we can  infer that
$$
\langle \mathcal D u , \beta\rangle =\int_{M} u  g^{i\bar j }g^{ l \bar k}\nabla_{\bar j }\nabla_{\bar k }\bar\beta_{il}\frac{\omega^n}{n!}\,,
$$ obtaining that $\mathcal D^*\beta=g^{j\bar i }g^{k\bar l  }\nabla_j\nabla_k\beta_{\bar i \bar l }$\,. This guarantees that 
$$
\mathcal D^*\mathcal D  u = g^{j\bar k }g^{p\bar q }\nabla_p\nabla_j\nabla_{\bar k }\nabla_{\bar q} u \,, 
$$  as in the K\"ahler case. On the other hand, we have that, for any $\beta \in \Lambda^{1,0}(M)\otimes \Lambda^{1,0}(M)$,  
\begin{equation}\label{comm1}
\nabla_p\nabla_j \beta_{\bar k\bar q}=\nabla_j\nabla_p\beta_{\bar k \bar q}+T_{jp}^s\nabla_s\beta_{\bar k \bar q}\,.
\end{equation}
We can then use \eqref{comm1} to infer that 
$$
\mathcal D^*\mathcal D u =g^{p\bar q }g^{j\bar k }\nabla_{j}\nabla_p\nabla_{\bar k }\nabla_{\bar q} u + g^{p\bar q }g^{j\bar k }T_{jp}^s\nabla_s\nabla_{\bar k}\nabla_{\bar q} u \,.
$$
Furthermore, we can  recall that, for any $\alpha\in \Lambda^{0,1}(M)$, 
\begin{equation}\label{commnabla}
\nabla_p\nabla_{\bar k}\alpha_{\bar q }=\nabla_{\bar k}\nabla_{p}\alpha_{\bar q } + R_{\bar kp\bar q }^{\bar t }\alpha_{\bar t }\,.
\end{equation} Hence, by making use of  \eqref{commnabla}, we obtain that 
$$
\begin{aligned}
\mathcal D^*\mathcal D u =&\,  g^{j\bar k  }g^{p\bar q  }\nabla_j\nabla_{\bar k }\nabla_{p}\nabla_{\bar q } u + g^{j\bar k  }g^{p\bar q  }\nabla_j(R_{\bar kp\bar q}^{\bar t}\nabla_{\bar t } u )+g^{p\bar q }g^{j\bar k }T_{jp}^s\nabla_s\nabla_{\bar k }\nabla_{\bar q } u \,.\\
\end{aligned}
$$ 
Now, recalling that, thanks to the second Bianchi identity, see for instance \cite[Proposition 1.6]{Us}, we have
$$
\nabla_jR_{\bar k p \bar q }^{\bar t}=\nabla_pR_{\bar k j \bar q }^{\bar t }+T_{pj}^sR_{\bar k s\bar q}^{\bar t}\,, 
$$
 we can conclude that 
 $$
\begin{aligned}
 g^{j\bar k  }g^{p\bar q  }\nabla_j(R_{\bar kp\bar q}^{\bar t}\nabla_{\bar t } u )=&\, g^{j\bar k  }g^{p\bar q  }\nabla_jR_{\bar kp\bar q}^{\bar t}\nabla_{\bar t } u + g^{j\bar k  }g^{p\bar q  }R_{\bar kp\bar q}^{\bar t}\nabla_j\nabla_{\bar t } u \\
 =&\, g({\rm Ric}^{(3)}(\omega), i\partial \bar \partial  u )+ g^{p\bar q  }\nabla_p{\rm Ric}^{(2)}(g)^{\bar t }_{\bar q}\nabla_{\bar t } u +g^{j\bar k  }g^{p\bar q  }T_{pj}^sR_{\bar k s \bar q }^{\bar t}\nabla_{\bar t} u\,, 
\end{aligned}
$$
where 
${\rm Ric}^{(3)}(\omega)$ is the third Chern-Ricci form, see, for instance, \cite{LiY} for its definition. On the other hand, using \cite[Theorem 4.1]{LiY}, we know that, in the balanced case, ${\rm Ric}^{(3)}(\omega)={\rm Ric}^{ch}(\omega)$. This gives us that
$$
\begin{aligned}
\mathcal D^*\mathcal D u =&\, 
 \Delta_{\omega}^2 u +g( {\rm Ric}^{ch}(\omega), i\partial \bar \partial u) + g^{p\bar q }\nabla_p{\rm Ric}^{(2) }(g)^{ \bar t}_{\bar q }\nabla_{\bar t } u +g^{j\bar k  }g^{p\bar q  }T_{pj}^sR_{\bar k s \bar q }^{\bar t}\nabla_{\bar t} u +g^{p\bar q }g^{j\bar k }T_{jp}^s\nabla_s\nabla_{\bar k}\nabla_{\bar q} u \,.
\end{aligned}
$$
Now,  we can use \eqref{commnabla} to infer that 
$$
g^{p\bar q }g^{j\bar k }T_{jp}^s\nabla_s\nabla_{\bar k}\nabla_{\bar q} u -g^{j\bar k  }g^{p\bar q  }T_{jp}^sR_{\bar k s \bar q }^{\bar t}\nabla_{\bar t} u =g^{p\bar q }g^{j\bar k }T_{jp}^s\nabla_{\bar k}\nabla_{s}\nabla_{\bar q} u \,.
$$
 On the other hand, one can easily check that 
 $$
g^{p\bar q }g^{j\bar k }T_{jp}^s\nabla_{\bar k}\nabla_{s}\nabla_{\bar q} u =\frac12g^{p\bar q }g^{j\bar k }T_{jp}^s(\nabla_{\bar k}\nabla_{s}\nabla_{\bar q} u - \nabla_{\bar q }\nabla_s\nabla_{\bar k } u )\,,
$$
 moreover,  using that, for any $\gamma\in \Lambda^{1,1}(M)$, we have that
 $$
\nabla_{\bar k }\gamma_{s\bar q}-\nabla_{\bar q}\gamma_{s\bar k }=(\bar \partial \gamma)_{\bar k s \bar q }+T_{\bar q \bar k }^{\bar t }\gamma_{s\bar t }\,, 
$$
 concluding that 
$$
g^{p\bar q }g^{j\bar k }T_{jp}^s\nabla_{\bar k}\nabla_{s}\nabla_{\bar q} u =\frac12g^{p\bar q }g^{j\bar k }T_{jp}^sT_{\bar q \bar k }^{\bar t }\nabla_s\nabla_{\bar t} u =-\frac12g(\Xi, i\partial \bar\partial  u  )\,.
$$ 
Then, putting all the pieces together we obtain that 
$$
\begin{aligned}
\mathcal D^*\mathcal D u =&\, \Delta_{\omega}^2u+g\left({\rm Ric}^{ch}(\omega)-\frac12\Xi, i\partial \bar \partial  u \right)+ g^{p\bar q }\nabla_p{\rm Ric}^{(2)}(g)^{\bar t }_{\bar q }\nabla_{\bar t } u 
\\
=&\, \Delta_{\omega}^2u+g\left({\rm Ric}^{ch}(\omega)-\frac12\Xi, i\partial \bar \partial  u \right)+ g(\bar \partial ^*{\rm Ric}^{(2)}(\omega), i \partial  u )\,,
\end{aligned}
$$ where the last equality is due to the fact that 
$$
\bar\partial ^*{\rm Ric}^{(2)}(\omega)_s=-g^{p\bar q }\nabla_{p}{\rm Ric}^{(2)}(\omega)_{s\bar q }\,,
$$ see, for instance \cite[Appendix D]{FP} for a more general statement. The claim is obtained making use of \cite[Proposition 4.3]{STi}.

In order to prove the second claim, we observe that:
\begin{equation}\label{formulina}
g\left( {\rm Ric}^{ch}(\omega)-\frac12\Xi, i\partial \bar \partial u\right) + g\left(\bar \partial^*\left({\rm Ric}^{ch}(\omega)-\frac12\Xi\right), i\partial u\right)=\bar \partial^*\left(\iota_{i(\bar \partial u)^{\sharp}}\left({\rm Ric}^{ch}(\omega)-\frac12\Xi\right)\right)\,.
\end{equation}
Indeed, using \eqref{formulazza1} and \eqref{formulazza2},  we have that
\begin{equation}\label{debbariota}
\begin{aligned}
\bar \partial^*\left(\iota_{i(\bar \partial u )^{\sharp }}\left({\rm Ric}^{ch}(\omega)-\frac12\Xi\right)\right)
=&\, *\left(i\partial \bar \partial u \wedge *\left({\rm Ric}^{ch}(\omega)-\frac12\Xi\right) - i\bar \partial u \wedge \partial *\left({\rm Ric}^{ch}(\omega)-\frac12\Xi\right)\right)\\
=&\, g\left({\rm Ric}^{ch}(\omega)-\frac12\Xi, i \partial \bar \partial u \right)+ g\left(\bar \partial^*\left({\rm Ric}^{ch}(\omega)- \frac12\Xi\right), i \partial u  \right)\,, 
\end{aligned}
\end{equation}
as claimed. Moreover, since  $u\in \ker \mathcal D^*\mathcal D$ is equivalent to require that $(\bar \partial u)^{\sharp }$ is a holomorphic vector field, then, we have 
\begin{equation}\label{holkillstorsion}
0=\left(\mathcal L_{(\bar \partial u )^{\sharp}}\omega\right)^{0,2}=\bar \partial \iota_{(\bar \partial u )^{\sharp}}\omega+ \iota_{(\bar \partial u )^{\sharp}}\bar \partial \omega=
\iota_{(\bar \partial u )^{\sharp}}\bar \partial \omega\,.\end{equation}
But, on the other hand,  we remark that 
\begin{equation}\label{Q^2}
\frac12Q^2_{i\bar j }=g(j_{Z_i}\bar T, j_{Z_j}\bar T)\,,
\end{equation} where $j_XT(Y, W):=g(T(Y, W), X)$, for any $X, Y, W\in TM$. On the other hand, $\iota_{(\bar \partial u )^{\sharp}}\bar \partial \omega=0$ clearly implies that $j_{(\bar \partial u )^{\sharp}}\bar T =0$ forcing \begin{equation}\label{holq2}
\iota_{(\bar \partial u )^{\sharp}}Q^2=0\,.
\end{equation} Then, putting together \eqref{formulina} and \eqref{holq2}, we have that $u\in \ker \mathcal D^*\mathcal D$ if and only if 
$$
\bar \partial \Delta_{\bar \partial }u + \iota_{i(\bar \partial u  )^{\sharp}}{\rm Ric}^{ch}(\omega)=0\,.
$$
Finally, let $u\in \ker \mathcal D^*\mathcal D$,   we have that 
\begin{equation}\label{nokerD}
\begin{aligned}
0= \langle\bar \partial \Delta_{\bar \partial }u + \iota_{i(\bar \partial u  )^{\sharp}}{\rm Ric}^{ch}(\omega),  \bar \partial u \rangle_{L^2}
=&\, ||\Delta_{\bar\partial }u||^2_{L^2}- \int_M {\rm Ric}^{ch}(g)((\bar \partial u )^{\sharp}, (\partial u)^{\sharp})\frac{\omega^n}{n!}\,,
\end{aligned}
\end{equation} which allows to conclude, using that ${\rm Ric}^{ch}(g)\le 0 $.
\end{proof}
 
Lemma \ref{lemLich} highlights the fundamental differences between the problem in our case and in the K\"ahler one. Indeed, formally, the role which in the K\"ahler case belongs to the Ricci form now is played by the form ${\rm Ric}^{ch}(\omega)-\frac12\Xi$. As one can see from the proof, the balanced condition is playing a crucial role both in computing $\mathcal D^*$ and in identifying the third and the first Chern-Ricci tensors. This last property is really exclusive of the balanced setting (it is indeed a characterization of the balanced condition, see \cite[Theorem 4.1]{LiY}). 

As we already noticed, $\ker \mathcal D^*\mathcal D$ consists of all those functions whose $(1,0)$-part of the gradient generates a holomorphic vector field. We will now give an equivalent condition, in terms of the dual form, for a vector field to be holomorphic in the balanced setting. Moreover, assuming that the first  Chern-Ricci tensor is non-positive, we give an equivalent condition for a holomorphic vector field to be parallel with respect to the Chern connection. \par

\begin{lem}\label{Chernparallel}
Let $(M^n, \omega)$ be a compact balanced  manifold.  Then,  $Z\in T^{1,0}M$ is  holomorphic if and only if  $\alpha:=Z^{\flat}$ satisfies the following:
$$
\Delta_{\bar \partial }\alpha+\iota_{iZ}{\rm Ric}^{(2)}(\omega)+\frac12\Lambda^2(i\partial \alpha \wedge \bar \partial \omega)+\frac12\Lambda^2(i\bar \partial \alpha \wedge \partial \omega)=0\,.
$$
Moreover, if ${\rm Ric}^{ch}(g)\le 0$,  then, a holomorphic vector field  $Z$ is Chern-parallel if and only if $\partial Z^{\flat}=0$.
\end{lem}
\begin{proof} Recall that $Z\in T^{1,0}M$ is holomorphic if and only if $\Delta_{\bar \partial }Z:=\bar \partial ^*\bar \partial Z=0$, where $\bar \partial $ is the Cauchy-Riemann  operator on $T^{1,0}M$. Easily,  we see that $(\Delta_{\bar \partial }Z)^k=-g^{i\bar j }\nabla_{i}\nabla_{\bar j }Z^k$. Furthermore, we have that, for any $\alpha\in \Lambda^{0,1}(M)$, 
$$
(\bar \partial \alpha)_{\bar j \bar m }
=\nabla_{\bar j }\alpha_{\bar m}- \nabla_{\bar m }\alpha_{\bar j }+ T_{\bar j \bar m }^{\bar q}\alpha_{\bar q}\,, 
$$
 which gives, together with \cite[Appendix D]{FP},  that 
 $$
(\bar \partial^*\bar \partial \alpha)_{\bar m}
=-g^{i\bar j }(\nabla_{i}\nabla_{\bar j}\alpha_{\bar m }- \nabla_{i}\nabla_{\bar m }\alpha_{\bar j } + \nabla_i(T_{\bar j \bar m }^{\bar q}\alpha_{\bar q}))+ \frac12g^{i\bar j }g^{p\bar q }(\bar \partial\alpha)_{\bar j\bar q}T_{pi\bar m }\,,\quad (\bar \partial \bar \partial^*\alpha)_{\bar m }=-g^{i\bar j}\nabla_{\bar m }\nabla_{i}\alpha_{\bar j }\,.
$$ 
Hence, we can infer that 
$$
(\Delta_{\bar \partial }\alpha)_{\bar m }=-g^{i\bar j }(\nabla_{i}\nabla_{\bar j}\alpha_{\bar m }- \nabla_{i}\nabla_{\bar m }\alpha_{\bar j } + \nabla_i(T_{\bar j \bar m }^{\bar q}\alpha_{\bar q})+\nabla_{\bar m }\nabla_{i}\alpha_{\bar j })+ \frac12g^{i\bar j }g^{p\bar q }(\bar \partial\alpha)_{\bar j\bar q}T_{pi\bar m }\,.
$$
Firstly, we observe that, following the computations in \eqref{assurdooo}, 
\begin{equation}\label{nonhoidea}
\begin{aligned}
\frac12g^{i\bar j }g^{p\bar q }(\bar \partial\alpha)_{\bar j\bar q}T_{pi\bar m }=&\, -g(i\bar \partial \alpha, \iota_{Z_m}\bar \partial \omega)
=-\frac12(\Lambda^2(i\bar \partial \alpha \wedge \partial \omega))_{\bar m }\,.
\end{aligned}
\end{equation}
 Using \eqref{commnabla}, we have that
 \begin{equation}\label{illaplace}
 \begin{aligned}
(\Delta_{\bar \partial }\alpha)_{\bar m }=&\, 
-g^{i\bar j }(\nabla_{i}\nabla_{\bar j}\alpha_{\bar m }+ \nabla_i(T_{\bar j \bar m }^{\bar q}\alpha_{\bar q}))+ {\rm Ric}^{ch}(g)_{\bar m }^{\bar t }\alpha_{\bar t}-\frac12(\Lambda^2(i\bar \partial \alpha \wedge \partial \omega))_{\bar m }\,.
\end{aligned}
\end{equation}
Now, using the first Bianchi Identity, see \cite[Proposition 1.6]{Us}, we obtain
$$
\begin{aligned}
g^{i\bar j }\nabla_i(T_{\bar j \bar m }^{\bar q}\alpha_{\bar q})
=&\, ({\rm Ric}^{ch}(g)_{\bar m}^{\bar q }- {\rm Ric}^{(2)}(g)_{\bar m }^{\bar  q} )\alpha_{\bar q }+ g^{i\bar j }g^{p\bar q  }T_{\bar j \bar m p }(\partial \alpha)_{i\bar q} \,,
\end{aligned}
$$ which put in \eqref{illaplace} gives
$$
\begin{aligned}
(\Delta_{\bar \partial }\alpha)_{\bar m } 
=&\, (\Delta_{\bar \partial }\alpha^{\sharp})_{\bar m }^{\flat}+ \iota_{\alpha^{\sharp}}{\rm Ric}^{(2)}(g)_{\bar m }+g^{i\bar j }g^{p\bar q  }T_{\bar m \bar j  p }(\partial \alpha)_{i\bar q}-\frac12(\Lambda^2(i\bar \partial \alpha \wedge \partial \omega))_{\bar m }\,.
\end{aligned}
$$
As similarly done  for \eqref{nonhoidea}, we conclude observing that 
$$
\begin{aligned}
g^{i\bar j }g^{p\bar q  }T_{\bar m \bar j  p }(\partial \alpha)_{i\bar q}
=-\frac12\Lambda^2(i\partial \alpha \wedge \bar \partial \omega)_{\bar m }\,.
\end{aligned}
$$
Hence, we obtain that 
$$
\Delta_{\bar \partial }\alpha=(\Delta_{\bar \partial }\alpha^{\sharp})^{\flat}+\iota_{\alpha^{\sharp}}{\rm Ric}^{(2)}(g)-\frac12\Lambda^2(i\partial \alpha \wedge \bar \partial \omega)-\frac12\Lambda^2(i\bar \partial \alpha \wedge \partial \omega)\,,
$$ which gives the first claim.
 As regards the  second claim, first of all, we observe,  as in \eqref{holkillstorsion}, that, if $Z$ is holomorphic, then, 
 \begin{equation}\label{holomorphic}
\iota_Z\bar \partial \omega=-\bar \partial \iota_Z\omega=-i\bar \partial Z^{\flat}\,.
\end{equation}
Using \eqref{formulazza2} and \eqref{startors}, we infer that 
$$
\begin{aligned}
\langle i\partial Z^{\flat}, \iota_Z\partial \omega \rangle_{L^2}
=&\, \int_M\partial Z^{\flat}\wedge \bar Z^{\flat}\wedge  \frac{\bar \partial\omega^{n-2}}{(n-2)!}\,.
\end{aligned}
$$
On the other hand, we have that 
$$
\begin{aligned}
\partial \left(Z^{\flat}\wedge \bar Z^{\flat}\wedge \frac{\bar \partial \omega^{n-2}}{(n-2)!}\right)=&\, \partial Z^{\flat}\wedge \bar Z^{\flat}\wedge\frac{\bar \partial \omega^{n-2}}{(n-2)!}-Z^{\flat}\wedge \partial\bar Z^{\flat}\wedge \frac{\bar\partial\omega^{n-2} }{(n-2)!}-iZ^{\flat}\wedge \bar Z^{\flat}\wedge \frac{i \partial \bar \partial \omega^{n-2}}{(n-2)!}\,.
\end{aligned}
$$
We can now use \eqref{ricvspartstar} and \eqref{formulazza1} and obtain
 \begin{equation}\label{CP1}
-iZ^{\flat}\wedge \bar Z^{\flat}\wedge \frac{i \partial \bar \partial \omega^{n-2}}{(n-2)!}
=-(\iota_{\bar Z}\iota_{iZ}\partial^*\partial \omega)\frac{\omega^n}{n!}\,.
\end{equation}
Moreover,   we can again use \eqref{startors}, \eqref{formulazza2} and \eqref{holomorphic} to conclude that 
\begin{equation}\label{CP2}
\begin{aligned}
Z^{\flat}\wedge \partial\bar Z^{\flat}\wedge \frac{\bar\partial\omega^{n-2} }{(n-2)!}
=&\, i\partial\bar Z^{\flat}\wedge*(\iota_Z\bar \partial \omega)=i\partial\bar Z^{\flat}\wedge *(-i\bar \partial Z^{\flat})=|\bar \partial Z^{\flat}|^2\frac{\omega^n}{n!}\,.
\end{aligned}
\end{equation}
Then, making use of \eqref{CP1} and \eqref{CP2},
$$
\begin{aligned}
 \partial Z^{\flat}\wedge \bar Z^{\flat}\wedge\frac{\bar \partial \omega^{n-2}}{(n-2)!}
 =&\, \partial \left(Z^{\flat}\wedge \bar Z^{\flat}\wedge \frac{\bar \partial \omega^{n-2}}{(n-2)!}\right)+( |\bar \partial Z^{\flat}|^2+ \iota_{\bar Z}\iota_{iZ}\partial^*\partial \omega)\frac{\omega^n}{n!}\,
 \end{aligned}
$$and thus 

\begin{equation}\label{laseggiola}
\langle i\partial Z^{\flat}, \iota_Z\partial \omega \rangle_{L^2}=||\bar \partial Z^{\flat}||_{L^2}^2 + \int_M(\iota_{\bar Z}\iota_{iZ}\partial^*\partial \omega)\frac{\omega^n}{n!}\,.
\end{equation}
On the other hand, we have that
$$
- i\partial Z^{\flat}\wedge \bar Z^{\flat}\wedge \frac{i\bar \partial \omega^{n-2}}{(n-2)!}=\bar \partial \left(i\partial Z^{\flat}\wedge i\bar Z^{\flat} \wedge  \frac{\omega^{n-2}}{(n-2)!}\right)+ i\partial \bar \partial Z^{\flat}\wedge i \bar Z^{\flat}\wedge \frac{\omega^{n-2}}{(n-2)!}-i \partial Z^{\flat}\wedge i \bar \partial \bar Z^{\flat}\wedge \frac{\omega^{n-2}}{(n-2)!}\,.
$$
 Now,  we can use \eqref{gigiL} to have 
 $$\int_Mi\partial \bar \partial Z^{\flat}\wedge i \bar Z^{\flat}\wedge \frac{\omega^{n-2}}{(n-2)!}
 =
 -\langle i\partial \bar \partial Z^{\flat}, LZ^{\flat}\rangle_{L^2}=-\langle i\bar \partial Z^{\flat}, \partial^*LZ^{\flat} \rangle_{L^2}\,.$$
On the other hand, using \eqref{holomorphic}, one can easily check that 
$$
\begin{aligned}
\partial^*LZ^{\flat}
=-*\left(i\bar \partial Z^{\flat}\wedge \frac{\omega^{n-2}}{(n-2)!}- iZ^{\flat}\wedge\frac{\bar \partial \omega^{n-2}}{(n-2)!}\right)=0\,,
\end{aligned}
$$  hence, 
$$
\langle i\partial \bar \partial Z^{\flat}, LZ^{\flat} \rangle_{L^2}
=0\,.
$$
Then, using \cite[Lemma 4.7]{Sz}, we obtain
$$
\int_M - i\partial Z^{\flat}\wedge \bar Z^{\flat}\wedge \frac{\bar \partial \omega^{n-2}}{(n-2)!}
= \int_M \partial Z^{\flat}\wedge  \bar \partial \bar Z^{\flat}\wedge \frac{\omega^{n-2}}{(n-2)!}=-||\partial Z^{\flat}||_{L^2}^2+ ||\Lambda(\partial Z^{\flat})||_{L^2}^2\,,
$$
which, together with \eqref{laseggiola}, gives that 
$$
\int_M (\iota_{\bar Z}\iota_{iZ}\partial^*\partial \omega)\frac{\omega^{n}}{n!}=-||\bar \partial Z^{\flat}||_{L^2}^2-||\partial Z^{\flat}||_{L^2}^2+ ||\Lambda(\partial Z^{\flat})||_{L^2}^2\,.
$$ Furthermore, 
one can easily verify that 
$
\bar \partial^*Z^{\flat}=
-\Lambda(\partial Z^{\flat})
$. 
However, thanks to \cite[Theorem 1.2]{gigis}, we know that if the Chern-Ricci tensor is non-positive,  $\bar \partial^* Z^{\flat}=0$ and ${\rm Ric}^{ch}(Z, \bar Z)=0$. From this, we conclude  that 
$$
\int_M (\iota_{\bar Z}\iota_{iZ}\partial^*\partial \omega)\frac{\omega^{n}}{n!}=-||\bar \partial Z^{\flat}||_{L^2}^2-||\partial Z^{\flat}||_{L^2}^2\,.
$$
Moreover, making use of \eqref{Q^2} and \eqref{holomorphic}, we observe that 
$$
\int_M\frac12\iota_{\bar Z}\iota_{iZ}\Xi\frac{\omega^n}{n!}=-||\iota_Z\bar\partial \omega||_{L^2}^2=-||\bar \partial Z^{\flat}||_{L^2}^2\,.
$$
Hence, using again \cite[Proposition 4.3]{STi}, 
$$
-\int_M\iota_{\bar Z}\iota_{Z}{\rm Ric}^{(2)}(g)\frac{\omega^n}{n!}=\int_M \iota_{\bar Z}\iota_{iZ}{\rm Ric}^{(2)}(\omega)\frac{\omega^n}{n!}=\int_M\iota_{\bar Z}\iota_{iZ}\left(\partial^*\partial \omega-\frac12\Xi\right)\frac{\omega^n}{n!}=-||\partial Z^{\flat}||_{L^2}^2\,. 
$$
Now, we recall that, for any holomorphic vector field $Z\in T^{1,0}M$,  the following Bochner formula, see for instance \cite{NZ}, holds:
$$
\Delta_{\omega}|Z|^2=|\nabla Z|^2-\iota_{\bar Z}\iota_{Z}{\rm Ric}^{(2)}(g)\,,
$$ which integrated over $M$ gives
$$
0=||\nabla Z||^2_{L^2}-\int_M \iota_{\bar Z}\iota_{Z}{\rm Ric}^{(2)}(g)\frac{\omega^n}{n!}=||\nabla Z||^2_{L^2}-||\partial Z^{\flat}||^2_{L^2}\,,
$$giving the claim.
\end{proof}
Unfortunately, the expression of the Chern-Lichnerowicz operator is not fitting well within that of $\mathcal L $. However, in order to understand a bit better the linearized operator, we can try to modify the Lichnerowicz-type operator we are considering. To do so,   we consider the following connections, introduced in \cite{Us}, obtained as variations  of the Chern connection:
 
$$
\begin{aligned}
 \nabla^1_XY=&\, \nabla_XY-T(X, Y)\,, \quad X, Y \in TM\,,
\\
\nabla^2_XY=&\, \nabla_XY +g (Y, T(X, \cdot))^{\sharp}\,, \quad X, Y\in TM\,.
\end{aligned}
$$
 It is fairly easy to prove that these connections  still preserve $J$, but, as  remarked in \cite{Us}, 
$\nabla^1$ and $\nabla^2$ are dual conjugate with respect to $g$, namely they satisfy the following:
$$
Xg(Y, W)=g(\nabla^1_XY, W)+ g(Y, \nabla^2_XW)\,, \quad X, Y, W\in TM\,.
$$An interesting thing about $\nabla^1$ is the following fact. Let  $u\in C^{\infty}(M, \R)$, we have
\begin{equation}\label{trick}
 \nabla^1_{\bar k} \nabla^1_{\bar j}u=\nabla_{\bar k }\nabla_{\bar j }u+ T_{\bar k \bar j }^{\bar m }\nabla_{\bar m}u=\nabla_{\bar j }\nabla_{\bar k }u\,.
\end{equation} Then, as done for the operator $\mathcal D$, we consider the operator $\tilde{\mathcal D}u= \nabla^1_{\bar k } \nabla^1_{\bar q}u$, for any $u \in C^{\infty}(M, \R)$. Now, using \eqref{trick} and the same computations as in \eqref{iddelirio}, we can infer that, for any $\beta\in \Lambda^{1,0}(M)\otimes \Lambda^{1,0}(M)$,
$$
\tilde{\mathcal D}^*\beta=g^{j\bar i }g^{k\bar l  }\nabla_k\nabla_j\beta_{\bar i \bar l }\,.
$$
\begin{lem}\label{inuovolich}
Let $(M^n, \omega)$ be a compact balanced manifold. Then, for any $u \in C^{\infty}(M, \R)$, we have
\begin{equation}\label{eqnuovolich}
\tilde{\mathcal D}^*\mathcal Du=\Delta_{\omega}^2u+ g({\rm Ric}^{ch}(\omega), i\partial \bar \partial u)+ g(\bar \partial^*{\rm Ric}^{ch}(\omega), i\partial u)\,.
\end{equation}
 In particular, if $\omega $ has constant Chern scalar curvature, we have 
 $$
\mathcal L (u)=-\frac{1}{n-1}\left(\Re(\tilde{\mathcal D}^*\mathcal Du + g(\bar \partial^*{\rm Ric}^{ch}(\omega), i\partial u)) + \frac{1}{n-1}(\Delta_{\omega}+ s^{ch}(\omega){\rm Id})(\lvert \partial\omega\rvert^2 u) +ui\partial^*\bar \partial^*{\rm Ric}^{ch}(\omega)\right)\,.
 $$
\end{lem}
\begin{proof}
The computations are the same as in the proof of  Lemma \ref{lemLich}. Locally, 
$$
\tilde{\mathcal D}^*\mathcal Du=g^{ j\bar k  }g^{p\bar q  }\nabla_{j}\nabla_p\nabla_{\bar k}\nabla_{\bar q }u\,.
$$ Now, applying \eqref{commnabla}, we have that 
$$
\tilde{\mathcal D}^*\mathcal Du=
\Delta^2_{\omega}u+ g^{ j\bar k  }\nabla_{j}({\rm Ric}^{(3)}(g)^{ \bar t }_{\bar k  }\nabla_{\bar t }u)=\Delta_{\omega}^2u+ g({\rm Ric}^{ch}(\omega), i\partial \bar \partial u)+ g(\bar \partial^*{\rm Ric}^{ch}(\omega), i\partial u)\,,
$$
 using the balanced condition to infer that ${\rm Ric}^{(3)}(\omega)={\rm Ric}^{ch}(\omega)$. The second statement is clear using \eqref{eqnuovolich} in Proposition \ref{thmveroL}. 
\end{proof}
Hence Lemma \ref{inuovolich} shows us that the operator $\tilde{\mathcal D}^*\mathcal D$ consists in a part of the linearized operator $\mathcal L $, and thus it gains  a great importance in the study of constant Chern scalar curvature balanced metrics. In view of this, we will characterize the kernel of such operator in two different ways. The first one is a characterization of $\ker \tilde{\mathcal D}^*\mathcal D$ as the set of  critical points of a suitably chosen energy, as the next lemma shows. 
\begin{lem}\label{varnostrolich}
Let $(M^n, \omega)$ be a compact balanced manifold. Then,  $u \in \ker \tilde{\mathcal D}^*\mathcal D$ if and only if 
$$
\langle \mathcal D u , \mathcal Dv \rangle_{L^2}=\frac12\int_M Q^2((\bar \partial u)^{\sharp }, (\partial v )^{\sharp})\frac{\omega^n}{n!}\,, \quad v \in C^{\infty}(M, \R)\,.
$$
 Then, straightforwardly,  $u\in\ker \tilde{\mathcal D}^*\mathcal D$ if and only if $u$ is a critical point of the following functional:
 $$
u\mapsto \int_M \left(|\mathcal Du|^2-\frac12Q^2((\bar \partial u)^{\sharp }, (\partial u )^{\sharp})\right)\frac{\omega^n}{n!}\,.
$$

\end{lem}
\begin{proof}
 We have that $u\in \ker \tilde{\mathcal D}^*\mathcal D$ if and only if, for any $v\in C^{\infty}(M, \R)$, 
 $$
0=\langle \tilde{\mathcal D}^*\mathcal Du, v \rangle_{L^2}= \langle \mathcal Du, \tilde {\mathcal D}v\rangle_{L^2}=\int_M (g^{i\bar j }g^{k\bar l }\nabla_{\bar j }\nabla_{\bar l }u \nabla_{k}\nabla_iv)\frac{\omega^n}{n!}\,.
$$
On the other hand, using that 
$$
\nabla_{k }\nabla_{i }v=\nabla_{i }\nabla_{k }v + T_{i k}^{ m }\nabla_{ m }v\,,
$$ we have that 
\begin{equation}\label{var0}
0=\langle \tilde{\mathcal D}^*\mathcal Du, v \rangle_{L^2}=\langle \mathcal Du, \mathcal Dv\rangle_{L^2} + \int_M (g^{i\bar j }g^{k\bar l }T_{ik }^{ m }
\nabla_{ m }v\nabla_{\bar j }\nabla_{\bar l }u)\frac{\omega^n}{n!}\,.\end{equation}
 Now, we observe that, using the balanced condition as in the proof of Lemma \ref{lemLich},
\begin{equation}\label{var1}
\int_M (g^{i\bar j }g^{k\bar l }T_{ik }^{ m }
\nabla_{ m }v\nabla_{\bar j }\nabla_{\bar  l}u)\frac{\omega^n}{n!}=-\int_M(g^{i\bar j }g^{k\bar l }\nabla_{\bar j }T_{ik}^{ m }
\nabla_{ m }v\nabla_{\bar l }u)\frac{\omega^n}{n!}-\int_M(g^{i\bar j }g^{k\bar l }T_{ik }^{ m }
\nabla_{\bar j }\nabla_{ m }v\nabla_{\bar l }u)\frac{\omega^n}{n!}\,.
\end{equation}
Thus, using the first Bianchi identity and \cite{STi}, we can infer that 
\begin{equation}\label{var2}
\begin{aligned}
-\int_M(g^{i\bar j }g^{k\bar l }\nabla_{\bar j }T_{ik }^{ m }
\nabla_{ m }v\nabla_{\bar  l}u)\frac{\omega^n}{n!}
=\langle\partial^*\partial \omega, i\partial u \wedge \bar \partial v \rangle_{L^2}- \frac12 \int_M Q^2((\bar \partial u )^{\sharp}, (\partial v)^{\sharp})\frac{\omega^n}{n!}\,.
\end{aligned}
\end{equation}
 On the other hand, using \eqref{formulazza1}, we obtain
 $$
 \begin{aligned}
-\int_M(g^{i\bar j }g^{k\bar l }T_{ik }^{ m }
\nabla_{\bar j }\nabla_{ m }v\nabla_{\bar l }u)\frac{\omega^n}{n!}
=\langle i\partial \bar \partial v, \iota_{( \partial u)^{\sharp}}\bar \partial \omega\rangle_{L^2}=-\langle  \partial^*\partial \omega,i \partial u \wedge  \bar \partial v  \rangle_{L^2}\,.
\end{aligned}
$$
 Hence, the claim is obtained combining \eqref{var1} and \eqref{var2} with \eqref{var0}.
\end{proof}
The second characterization of $\ker \tilde{\mathcal D}^*\mathcal D$ is in terms of the  kernel of a modified  Laplace operator on $T^{1,0}M$.  In order to achieve that, we consider the following  twisted Cauchy-Riemann operator:
 $$
 \bar\partial^T\colon T^{1,0}M\to \Lambda^{0, 1}(M)\otimes T^{1,0}M;
 $$ such that 
 $$
\bar \partial^{T}_{ W}Z=\bar \partial_WZ + g(Z, T(W, \cdot))^{\sharp}\,, \quad W\in T^{0,1}M,\,\, Z \in T^{1,0}M\,.
$$
 In components, since $(\bar \partial Z)_{\bar j }^i=\nabla_{\bar j }Z^i$, we have 
\begin{equation}\label{Doltt}
(\bar \partial^{T}Z)_{\bar j }^i=\nabla_{\bar j }Z^i + g^{i\bar p}T_{\bar j \bar p k} Z^k\,, 
\end{equation} which specialized when $Z=(\bar \partial u )^{\sharp}\,,$ with $u\in C^{\infty}(M, \R)$, gives
\begin{equation}\label{ullalla}
\begin{aligned}
(\bar \partial^{T}Z)_{\bar j }^i
=g^{i\bar l }\nabla_{\bar j }\nabla_{\bar l} u +g^{i\bar p}T_{\bar j\bar p  }^{\bar l}\nabla_{\bar l }u 
=g^{i\bar p }\nabla_{\bar p }\nabla_{\bar j }u\,.
\end{aligned}\end{equation}

\begin{lem}
 Let $(M, \omega)$ be a  compact balanced manifold. Then, we have that 
 $$
\tilde{\mathcal D}^*\mathcal Du=\bar \partial^*\left(\Delta^T_{ \bar \partial }(\bar \partial u)^{\sharp}\right)^{\flat}\,
,$$
 where $\Delta^T_{\bar \partial }:=\bar \partial^* \bar \partial^T\colon T^{1,0}M \to T^{1,0}M$. 
In particular, we have that
$$
\ker \mathcal D^*\mathcal D=\ker \tilde{ \mathcal D}^*\mathcal D \cap \{u\in C^{\infty}(M,  \R) \, \, |\, \, \iota_{(\bar \partial u )^{\sharp }}\bar \partial \omega=0\}\,.
$$
Moreover, $Z\in  \ker \Delta^T_{\bar \partial }$ if and only if $\alpha:=Z^{\flat}$ satisfies
\begin{equation}\label{gigiigi}
\Delta_{\bar \partial }\alpha+\iota_{iZ}{\rm Ric}^{ch}(\omega)+\frac12\Lambda^2(i\bar \partial \alpha \wedge \partial \omega)=0\,.
\end{equation}
Additionally, if  ${\rm Ric}^{ch}(\omega)=0$, we have that 
$$
(\ker \Delta^T_{ \bar \partial })^{\flat}\cap \ker \bar \partial =\mathcal H^{0,1}_{\bar \partial}:=\{\alpha\in\Lambda^{0,1}(M)\, \, |\, \, \Delta_{\bar \partial }\alpha=0\}\simeq H_{\bar \partial }^{0,1}(M, \C)\,.
$$
Furthermore, if ${\rm Ric}^{ch}(g)\le 0 $, $Z\in \ker \Delta^T_{\bar \partial}$ 
cannot be gradient. 
\end{lem}
\begin{proof}
 In components, using \eqref{ullalla}, we have that 
 $$
(\Delta^T_{\bar \partial }(\bar \partial u)^{\sharp})^k=-g^{k\bar p }g^{i\bar j }\nabla_i\nabla_{\bar p}\nabla_{\bar j}u\,.
$$
 Then, we can apply \eqref{commnabla} and obtain that 
 $$
(\Delta^T_{\bar \partial }(\bar \partial u)^{\sharp})^k=-g^{k\bar p }g^{i\bar j }\nabla_{\bar p }\nabla_{i}\nabla_{\bar  j}u- g^{i\bar j }g^{k\bar p }R_{\bar p i \bar j }^{\bar m }\nabla_{\bar m}u=\left(\left(-\bar \partial \Delta_{\omega}u+ \iota_{i(\bar \partial u )^{\sharp }}{\rm Ric}^{ch}(\omega)\right)^{\sharp}\right)^k\,,
$$
 giving that 
$$
\Delta^T_{\bar \partial }(\bar \partial u)^{\sharp}=\left(-\bar \partial \Delta_{\omega}u+ \iota_{i(\bar \partial u )^{\sharp }}{\rm Ric}^{ch}(\omega)\right)^{\sharp}\,.
$$ Then, it suffices to compute
$$
\partial^*(-\bar \partial \Delta_{\omega}u+ \iota_{i(\bar \partial u )^{\sharp }}{\rm Ric}^{ch}(\omega))=\Delta^2_{\omega}u +\bar \partial^*(\iota_{i(\bar \partial u )^{\sharp }}{\rm Ric}^{ch}(\omega)))\,.
$$
The conclusion follows straightforwardly repeating the computations done in \eqref{debbariota}  together with Lemma \ref{inuovolich}.

 The inclusion $\ker \mathcal D^*\mathcal D \subseteq\ker \tilde{ \mathcal D}^*\mathcal D \cap \{u\in C^{\infty}(M,  \R) \, \, |\, \, \iota_{(\bar \partial u )^{\sharp }}\bar \partial \omega=0\}$ is clear, using \eqref{holkillstorsion}. Conversely,  we know that $\iota_{(\bar \partial u )^{\sharp }}\bar \partial \omega=0$ implies $\iota_{(\bar \partial u )^{\sharp }}Q^2=0$, then, we can use Lemma \ref{varnostrolich} to conclude that $\mathcal D u =0$ giving the desired inclusion.
%

 In order to prove the third claim, we  observe that, using \eqref{Doltt}, 
$$
(\Delta^T_{\bar \partial }Z)^k=-g^{i\bar j }\nabla_i\nabla_{\bar j }Z^k- g^{i\bar j }g^{k\bar p }\nabla_i(T_{\bar j \bar p s}Z^s)\,.
$$
Now, choosing $\alpha=Z^{\flat}$, $Z\in T^{1,0}M$, we observe that 
$$
(\Delta^T_{\bar \partial }Z)^{\flat}_{\bar m }=-g^{i\bar j }\nabla_i\nabla_{\bar j }\alpha_{\bar m }- g^{i\bar j }\nabla_i(T_{\bar j \bar m s}Z^s)=-g^{i\bar j }\nabla_i\nabla_{\bar j }\alpha_{\bar m }- g^{i\bar j }\nabla_i(T_{\bar j \bar m }^{\bar q}\alpha_{\bar q})\,.
$$
 Then, we can conclude that 
 $$
(\Delta^T_{\bar \partial }Z)^{\flat}=\Delta_{\bar \partial }\alpha+ \iota_{iZ}{\rm Ric}^{ch}(\omega)+\frac12\Lambda^2(i\bar \partial \alpha \wedge \partial \omega)\,,$$
 by using \eqref{illaplace}.

The fourth claim is straightforward using \eqref{gigiigi}. The last claim is obtained following the same lines as in \eqref{nokerD}.
\end{proof}
 In Section \ref{examples}, we will discuss and explicitly compute the kernel of $\Delta_{\bar \partial }^T$ on the Iwasawa manifold and on the  Nakamura manifolds, see, respectively, Subsection \ref{iwa} and Subsection \ref{naka}.

\section{Proof of  the main Theorem}\label{proofimmain}

Having now convinced ourselves of the significance of this deformation, we can proceed with the proof of Theorem \ref{ansbal}.
We shall now introduce suitable weighted spaces, as done in \cite{BM}, as they will turn out to be the right spaces on which we are able to invert (uniformly) the operator $\tilde{\mathcal{L}}$. Since we can always assume,  up to rescaling,  that the neighbourhood of $x$ on which the $z$ coordinates are defined contains the region $\{|z|\leq 1\}$, we  define 
\[
\rho=\rho_\varepsilon(z):=
\begin{cases} 
\varepsilon^{p+q} & \text{on} \> |z|\leq \varepsilon^{p+q}; \\
\text{non decreasing} & \text{on} \> \varepsilon^{p+q} \leq |z|\leq 2\varepsilon^{p+q}; \\
|z| & \text{on} \> 2\varepsilon^{p+q} \leq |z| \leq 1/2; \\
\text{non decreasing} & \text{on} \> 1/2 \leq |z| \leq 1;\\
1 & \text{on} \> |z| \geq 1\,.
\end{cases}
\]
We then introduce, for all $b \in \mathbb{R}$, the \textit{weighted Hölder norm} as 
\[
\begin{aligned}
\lv u\rv_{C_{b, \ezi}^{k,\alpha}(\hM)}:= & \sum_{i=0}^k \sup_{\hM} |\rho^{b+i}\nabla_\varepsilon^iu|_\omega  
 +\underset{d_\varepsilon(x,y) < inj_\varepsilon}{\sup}\left|\min\left(\rho^{b+k+\alpha}(x),\rho^{b+k+\alpha}(y)\right)\frac{\nabla_\varepsilon^ku(x)-\nabla_\varepsilon^ku(y)}{d_\varepsilon(x,y)^\alpha}\right|_\omega,
\end{aligned}
\]
where $inj_\varepsilon$ is the injectivity radius of the metric $\omega$. Consequently, we define the corresponding \textit{weighted Hölder spaces} $C_{b, \ezi}^{k,\alpha}(\hM):=\{u\in C^{k}(\hM) \,\,  |\,\,  \lv u\rv_{C_{b, \ezi}^{k,\alpha}(\hM)}<\infty \}$, where $k\geq 0$, $\alpha \in (0,1)$ is the Hölder constant, and $\varepsilon$ indicates the dependence on the pre-gluing metric $\omega$ obtained in Subsection \ref{apprsol}. Hence, we can interpret $\tilde{\mathcal{S}}$ as 
$$\tilde{\mathcal{S}}: C_{b, \ezi}^{4,\alpha}(\hM) \rightarrow C_{b+4, \ezi}^{0,\alpha}(\hM)\,.$$

As in \cite[Theorem 1]{GS}, there is an obstacle given by the kernel of the operator $F_\omega$ (actually its limit - with respect to $\ezi$ - on $M_x$) in \eqref{opbalanced}. Nevertheless, we will see that it is possible to work orthogonally to this kernel, in order to ensure the invertibility of the operator. Indeed, we can introduce the functional space
$$\mathcal{V}_{\ezi}:= (\ker F_\omega)^{\perp_{L^2}} \subseteq C^{4, \alpha}_{b, \ezi}(\hM),$$
which inherits the Banach structure thanks to the fact that the topology induced by the weighted Hölder norm is finer then the $L^2$-topology, which guarantees that $\mathcal{V}_{\ezi}$ is a closed subspace of $C^{4, \alpha}_{b, \ezi}(\hM)$. We are then able to obtain the uniform invertibility of the linearized operator, 
 following the strategy in \cite{BM}. Before proving this result, however, we need a preliminary lemma, which will be central later.
    \begin{lem}\label{deformann}
        For all $v \in \ker F_{\tilde\omega}$, it exists $v_\ezi \in C^{4, \alpha}_{b, \ezi}(\hM)$ such that
        \begin{equation}\label{propext}
            v_\ezi \in \ker F_{ \omega_\ezi} \quad \text{and} \quad {v_\ezi}\equiv v\,, \quad \mbox{on \,\,} {\{|z|\geq 2\ezi^p\}}\,,
        \end{equation}
        for all $\ezi>0$.
    \end{lem}
    \begin{proof}
        For all $v \in \ker F_{\tilde\omega}$, we consider the same cut-off function $\chi_\ezi$ of Lemma \ref{flatcutoff} and call $R_\ezi:=\{\ezi^p < |z| <2\ezi^p\}$ the cut-off region, over which we consider the boundary value problem:
        \begin{equation}\label{bvstrong}
            \begin{cases}
                F_{\tilde\omega_\ezi}(u_{\ezi})=-F_{\tilde\omega_\ezi}(\chi_\ezi v) & \text{on}\>\> R_\ezi\,, \\
                u_{\ezi}=0 & \text{on} \>\> \p R_\ezi\,,
            \end{cases}
        \end{equation}
        for $u_{\ezi}$ smooth on $R_{\ezi}$.
        In order to solve problem \eqref{bvstrong}, we consider its weak formulation:
        \begin{equation}\label{bvweak}
            \begin{cases}
            B(\varphi,u_{\ezi}):=\langle \p\varphi,\p u_{\ezi}\rangle_{L_\ezi^2}-\frac{1}{n-1}\langle \varphi, |\p\tilde\omega_\ezi|_{\tilde\omega_\ezi}^2u_{\ezi}\rangle_{L_\ezi^2}=\langle \varphi, F_{\tilde\omega_\ezi}(\chi_\ezi v) \rangle_{L_\ezi^2}\,, &\>\> \varphi \in C_c^{\infty}(R_\ezi)\\
                u_{\ezi} \in W_{0,\ezi}^{1,2}(R_\ezi),
            \end{cases}
        \end{equation}
        where $L_\ezi^2$ identifies the $L^2$-product induced by $\tilde \omega_\ezi$, and same for the Sobolev space $W_{0,\ezi}^{1,2}(R_\ezi)$.\par
        In order to obtain a solution to problem \eqref{bvweak}, we notice that the bilinear form $B$ satisfies the Gårding inequality:
        \begin{equation}
            B(\varphi,\varphi) \geq \frac{1}{2}||\varphi||_{W_\ezi^{1,2}}^2-\left(\frac{\max |\p\tilde\omega_\ezi|_{\tilde\omega_\ezi}^2}{n-1}+\frac{1}{2}\right)||\varphi||_{L_\ezi^2}^2\,, \quad  \>\> \varphi \in W_{0,\ezi}^{1,2}(R_{\ezi}).
        \end{equation}
        This ensures us that we can apply \cite[Theorem 8.5]{Ag}, to obtain that problem \eqref{bvweak} has solution if and only if $F_{\tilde\omega_\ezi}(\chi_\ezi v)$ is orthogonal to $\ker F_{\tilde\omega_\ezi}$, where here $F_{\tilde\omega_\ezi}:W_{0,\ezi}^{2,2}(R_\ezi) \rightarrow L_\ezi^2(R_\ezi)$, which is clearly self-adjoint. On the other hand,  it is straightforward to notice that this last condition is indeed verified, ensuring us a solution $u_\ezi$.\par
        Now, if we extend $u_\ezi$ to the function $\tilde u_\ezi$, defined on the whole $M$ as identically zero outside of $R_\ezi$, it is clear that $\tilde u_\ezi$ solves (weakly) on $M$ the problem 
        $$F_{\tilde\omega_\ezi}(\tilde u_{\ezi})=-F_{\tilde\omega_\ezi}(\chi_\ezi v),$$ 
        ensuring that $\tilde u_\ezi$ is actually smooth on $M$, and hence also a classical solution of the latter equation.\par
        Finally, it is straightforward to see that $\tilde u_\ezi$ extends smoothly to $\hat u_\ezi$ on $\hM$ (by setting $\hat u_\ezi \equiv 0$ on the exceptional divisor), yielding a classical solution on $\hM$ of 
        $$F_{\omega_\ezi}(u)=-F_{\omega_\ezi}(\chi_\ezi v).$$ 
        Thus, the function 
        $$v_\ezi:= \chi_\ezi v + \hat u_\ezi$$
        is exactly the function we wanted.
    \end{proof}
 We are now ready to prove the main result of this subsection. 
\begin{prop}\label{injLbis}
For any  $b\in(0,2n-4)$, there exists $C>0$ such that,  for all  $u\in \mathcal V_{\ezi}$, we have 
$$\lVert u \rVert_{C^{4, \alpha}_{b, \ezi}(\hM)}\le C\lVert\tilde{\mathcal L }u\rVert_{C^{0, \alpha}_{b+4, \ezi}(\hM)}\,.
$$
\end{prop}
\begin{proof}
Suppose by contradiction that statement does not hold. Hence, we can find sequences $\{\ezi_k\}_{k\in \mathbb N}\subseteq \mathbb R_{>0}$ and $\{u_k\}_{k\in\mathbb N}$ such that $u_k\in \mathcal V_k:=\mathcal V_{\ezi_k}$, for all $k\in \N$, and 
\begin{equation}\label{hpass0bis}
    \ezi_k \underset{k \rightarrow +\infty}{\longrightarrow} 0\,, \quad ||u_k||_{C_{b, \ezi_k}^{4,\alpha}(\hM)}=1\,, \quad  \> k \in \mathbb N\,,
\end{equation}
and
\begin{equation}\label{hpassbis}
    ||\tilde{\mathcal{L}}u_k||_{C_{b+4, \ezi_k}^{0,\alpha}(\hM)}<\frac{1}{k}\,,  \quad  \> k \in \mathbb N.
\end{equation}
We will focus firstly on $M_x:=M\backslash \{x\}$. By applying Ascoli-Arzelà's Theorem, we have that $u_k \rightarrow u_\infty$ uniformly on compact subsets of $M_x$ in the sense of $C_b^{4,\alpha}$, with respect to the background metric $ \tilde{\omega}$. This implies, in particular, that on any compact subset of $M_x$, thanks to the fact that $ \tilde{\omega}$ is a Chern-Ricci flat balanced metric, it holds
\begin{equation}\label{convLbis}
    \mathcal{L}u_k\rightarrow -\Delta_{ \tilde{\omega}} F_{\tilde \omega}(u_\infty)\,, \quad  k\to \infty\,,
\end{equation}
i.e. $\mathcal{L}u_k$ converges uniformly on compact sets to a continuous function on $M_x$. If we then fix a point $y\in M_x$, in the region where $\rho\equiv 1$, condition \eqref{hpassbis} implies that
\[
\tilde{\mathcal{L}}u_k(y) \rightarrow 0,
\]
which, combined with equation \eqref{convLbis}, implies that the real sequence $\int_{\hM}u_k|\partial  \omega|^2\frac{\omega^n}{n!}$ has finite limit. Hence,  by Lebesgue's Theorem, we get
\begin{equation}\label{mezzolebesguebis}
    \int_{\hM}u_k|\partial \omega|^2\frac{\omega^n}{n!} \rightarrow \int_{M_x}u_\infty|\partial \tilde\omega|^2 \frac{\tilde{\omega}^n}{n!}\,,\quad  k \to \infty\,.
\end{equation}
If we now integrate $\tilde{\mathcal{L}}u_k$ on $M_x$, using equations \eqref{convLbis} and \eqref{mezzolebesguebis} and assuming $b<2n-4$, we obtain 
\begin{equation}
    0=\int_{M_x}\tilde{\mathcal{L}}u_\infty \frac{\tilde{\omega}^n}{n!}= -\int_{M_x}\Delta_{ \tilde{\omega}} F_{\tilde \omega}(u_\infty) \frac{\tilde{\omega}^n}{n!}+{\rm Vol}(M, \tilde \omega)\int_{M_x}u_\infty |\partial \tilde \omega|^2\frac{\tilde{\omega}^n}{n!} = {\rm Vol}(M, \tilde \omega)\int_{M_x}u_\infty|\partial \tilde \omega|^2 \frac{\tilde{\omega}^n}{n!},
\end{equation}
hence $\int_{M_x}u_\infty |\partial \tilde \omega|^2\frac{\tilde{\omega}^n}{n!}=0$. From this, we can infer that  $$
\Delta_{ \tilde{\omega}} F_{\tilde \omega}(u_\infty)=0\,,
$$
 from which follows that $u_\infty$ is such that $ F_{\tilde \omega}( u_\infty)\equiv c \in \mathbb R$. Now, recalling \eqref{opbalanced}, we again integrate the   equation $ F_{\tilde \omega}( u_\infty)\equiv c$ over the whole $M_x$  yielding that, since again $b<2n-4$,  
 $$
0=\int_{M_x}u_\infty |\partial \tilde \omega|^2\frac{\tilde{\omega}^n}{n!}=c{\rm Vol}(M, \omega)\,,
$$ which implies that $c=0$ and then $F_{\tilde \omega}(u_\infty)=0$ on $M_x$. 

Now, using that $u_{\infty} \in C^{4,\alpha}_{ b, \ezi}(M_x)$ and \eqref{opbalanced} again, we can conclude that 
$$
\Delta_{d}u_{\infty}\in C^{4, \alpha}_{ b, \ezi}(M_x)\,.
$$ Thanks to this, following the argument in \cite[Proposition 8.10]{Sz}, we have that $u_{\infty}\in C^{6, \alpha}_{ b-2, \ezi}(M_x)$. Iterating this process, we  can infer that $u_{\infty}\in C^{4+2j, \alpha}_{ b-2j, \ezi }(M_x)$ where $j\in \N$ is the first integer  such that $b-2j+1<0$. Now, we can extend $u_{\infty}$ to a function in $C^{4+2j, \alpha}(M)$ such that
$$
 u_{\infty}(x)=0 \,, \quad \Delta_{\omega}u_{\infty}(x)=0
$$   so that $u \in \ker F_{\tilde \omega}$ on the whole $M$. Now, elliptic regularity allows us to conclude that $u_{\infty}\in C^{\infty}(M, \R)$ such that $\int_M u_{\infty}|\partial \tilde\omega|^2\frac{\tilde\omega^n}{n!}=0$. Now, we want to use  that $u_{k}\in \mathcal V_k$ in order to conclude that $u_{\infty}$ is $L^2$-orthogonal to $\ker F_{\tilde \omega}$. If we show this, we will have that $u_{\infty}\in \ker F_{\tilde \omega }\cap (\ker F_{\tilde \omega})^{\perp}$ concluding that $u_{\infty}=0$.

 For all $v \in \ker F_{\tilde\omega}$ we can apply Lemma \ref{deformann} and obtain a corresponding $v_\ezi\in C^{4, \alpha}_{b, \ezi}(\hM)$. Then, denoting with $v_k:=v_{\ezi_k}$, for all $k \in \mathbb N$, using that  $u_k\in \mathcal V_k$,  it holds 
    \[
        \langle u_k, v_k \rangle_{\omega_{\ezi_k}}=0\,, \quad  \>\> k \in \mathbb N\,.
    \]
    This implies that, on any compact subset on $M_x$,  we have
    \[
    \langle u_\infty, v \rangle_{\tilde\omega}=0.
    \]
    Hence, considering an exhaustion of compact subsets of $M_x$ and using the fact that $u_\infty$ and $v$ are actually functions on $M$, we obtain
    \[
    \langle u_\infty, v \rangle_{\tilde\omega}=0
    \]
    on $M$, which means exactly that $$u_\infty \perp_{L^2} \ker F_{\tilde\omega}.$$
This allows to conclude that $u_{\infty}=0$, as explained above. 

We thus fix the compact set $M_c:=M\setminus \{|z|<1/2\}$, and focus on $A:=\{|z|<1/2\}$, on which we wish to obtain uniform convergence to zero.
For convenience, we shall shift to the \lq \lq large\rq\rq\, coordinates $\zeta$, i.e. the coordinates on the blow-up $\hat{X}$ defined outside the exceptional divisor. Recalling then that
\[
\zeta=\varepsilon^{-(p+q)}z \quad \text{and} \quad |z|=\varepsilon^{p+q}|\zeta|,
\]
we have the identification
\[
A \simeq \tilde{A}=\tilde{A}_\varepsilon:=\left\{|\zeta| <\frac{1}{2}\varepsilon^{-(p+q)}\right\}\subseteq \hat{X},
\]
and the last description will be the one we will use.\\
First of all, we shall rewrite $\rho$ with respect to $\zeta$ on $\tilde{A}$, giving
\[
\rho=
\begin{cases} 
\varepsilon^{p+q} & \text{on} \> |\zeta|\leq 1; \\
\text{non decreasing} & \text{on} \> 1 \leq |\zeta| \leq 2; \\
\varepsilon^{p+q}|\zeta| & \text{on} \> 2 \leq |\zeta| \leq 1/2\varepsilon^{-(p+q)}.
\end{cases}
\]
It follows that, going back to $\{u_k\}_{k\in \mathbb{N}}$ and recalling \eqref{hpass0bis}, we can infer, in particular, that on all $\tilde{A}_k:=\tilde{A}_{\ezi_k}$ the following holds
\[
|\rho^b u_k|\leq C.
\]
This suggests us to introduce the new sequence
\[
U_k:=\ezi_k^{b(p+q)}u_k,
\]
and using again \eqref{hpass0bis}, we obtain
\[
\begin{cases} 
|U_k|\leq C & \text{on} \> |\zeta|\leq 1; \\
|U_k|\leq C & \text{on} \> 1 \leq |\zeta| \leq 2; \\
|U_k|\leq C|z|^{-b}(\zeta) & \text{on} \> 2 \leq |\zeta| \leq 1/2\varepsilon_k^{-(p+q)}
\end{cases}
\]
and the same for its derivatives up to the fourth order.
These estimates for $U_k$ bring us to consider a new weight function $\tilde{\rho}=\tilde{\rho}_k$ on $\tilde{A}_k$ defined by
\[
\tilde{\rho}(\zeta)=
\begin{cases} 
1 & \text{on} \> |\zeta|\leq 1; \\
\text{non decreasing} & \text{on} \> 1 \leq |\zeta| \leq 2; \\
\lvert\zeta\rvert & \text{on} \> 2 \leq |\zeta| \leq 1/2\varepsilon_k^{-(p+q)},
\end{cases}
\]
 which gives that
\begin{equation}\label{wUbis}
|\tilde{\rho}^bU_k|\leq C,
\end{equation}
and estimates also for $\nabla^m U_k$,  for all $m=1,...,4$. Hence, since $\tilde{A}_k \rightarrow \hat{X}$, as $k\to \infty$,    using  again  Ascoli-Arzelà's Theorem, we obtain that $U_k \rightarrow U_\infty$,  as $k\to \infty$,  uniformly on compact sets of $\hat{X}$  in the sense of $\tilde{C}_b^{4,\alpha}:=C_b^{4,\alpha}(\tilde{\rho})$, where the later is the weighted Hölder space on $\hat{X}$ given by the weight $\tilde{\rho}$ and the metric $\omega_{BS}$.\\
On the other hand, on any compact subset of $\hat{X}$, for sufficiently large $k$, it holds
\begin{equation}\label{lichnbis}
\rho^{b+4}\mathcal L u_k =-\frac{1}{n-1}\tilde{\rho}^{b+4}\mathcal{D}^*\mathcal{D}U_k,
\end{equation}
where $\mathcal{D}^*\mathcal{D}$ is the Lichnerowicz operator corresponding to $\omega_{BS}$. Thus, since \eqref{hpassbis} holds, 
 by taking the limit in \eqref{lichnbis}, we obtain that $U_\infty$ is in the kernel of $\mathcal{D}^*\mathcal{D}$ with respect to the Burns-Simanca metric $\omega_{BS}$. Thus, applying  \cite[Proposition 8.10]{Sz}, we get that $U_\infty$ is necessarily constant, which needs to be zero as $U_\infty$ decays at infinity (from inequality \eqref{wUbis}). Hence, $U_k \to  0$ uniformly on compact sets of $\hat{X}$ in $\tilde{C}_b^{4,\alpha}$.\\
In order to conclude, we will show that $U_k$ admits a subsequence uniformly convergent to zero on the whole $\hat{X}$ in the sense $\tilde{C}_b^0$. This, combined with the scaled Schauder estimates, see for instance \cite[formula (6)]{BM}, will imply that also $U_k \to 0$ uniformly in $\tilde{C}_b^{4,\alpha}$.
On the other hand, this is equivalent to $u_k \to 0$ uniformly on $\{|z|< 1/2\}$ in $C_{b, \ezi}^{4,\alpha}$. Together with the fact  that $u_k$ converges uniformly to zero on $M_c$, it gives a contradiction with the fact that $||u_k||_{C_{b, \ezi}^{4,\alpha}(\hM)}=1$, for all $k \in \mathbb{N}$.\\
The final step of the proof will be to show that such subsequence necessarily exists. Indeed, if we assume by contradiction that such subsequence does not exist, then we can find a sequence $\{x_n\}_{n\in \mathbb N}\subseteq \hat{X}$ and  $\delta>0$ such that
\begin{equation}\label{Rkbis}
R_k:=|\zeta(x_k)|\rightarrow +\infty
\end{equation}
and
\begin{equation}
    \tilde \rho^b(x_k)|U_k(x_k)|\geq \delta\,,  \quad  k \geq 0\,.
\end{equation}
This last condition can be rewritten (up to choosing sufficiently large $k$) as 
\begin{equation}\label{assrescbis}
    R_k^b|U_k(x_k)|\geq \delta\,,  \quad  k \geq 0\,.
\end{equation}
If we then define $r_k:=|z(x_k)|$, we have that $r_k=\ezi_k^{p+q}R_k$, for all $k \in \mathbb N$, from which (up to subsequences) we see that we can only fall into two cases:
\begin{itemize}
    \item if $\lim_{k\rightarrow +\infty}r_k=r>0$, then it means that we can assume $x_k \rightarrow x_\infty$, which combined with the uniform convergence to zero on compact sets (of $M_x$) of the sequence $\{u_k\}_{k\in \mathbb N}$ gives
    \[
    0<\delta \leq R_k^b|U_k(x_k)|=r_k^b u_k(x_k) \rightarrow 0,
    \]
    i.e. a contradiction;
    \item if instead $\lim_{k\rightarrow +\infty}r_k=0$, we take $X'$ a copy of $\hat X$, and for all $k\geq 0$ we introduce the holomorphic maps
    \[ 
    \sigma_k:B_k \rightarrow A^*:=A\setminus\{0\},
    \]
    given by $\sigma_k(z'):=r_k z'$, where $B_k:=\{0<|z'|<r_k^{-1}/2\}\subseteq X'$. Using these,  we can define the metrics
    \[
    \theta_k:=r_k^{-2}\sigma_k^*\omega\,,
    \]
    and easily observe that $(B_k,\theta_k)\rightarrow (X',\omega_o)$, as $k\to \infty, $  where $\omega_o$ here denotes the flat metric induced by the coordinates $z'$. Then, it is natural to consider the functions on each $B_k$ given by
    \[
    W_k:=r_k^b\sigma_k^*u_k\,, \quad  k \in \mathbb N,
    \]
    and the pullback weight function
    \begin{equation}\label{pbpesobis}
    \rho'(z')=\sigma_k^*\rho(z')=
    \begin{cases}
        \ezi_k^{p+q} & \text{on} \> |z'|\leq R_k^{-1}\,;
        \\
        \text{non increasing} & \text{on} \> R_k^{-1}< |z'| < 2R_k^{-1}\,;  \\
        r_k|z'| & \text{on} \> 2R_k^{-1} \leq |z'| < \frac{1}{2}r_k^{-1}.
    \end{cases}
    \end{equation}

Now, if we pullback \eqref{hpass0bis} using $\sigma_k$, we immediately obtain that the sequence $\{W_k\}_{k\in \mathbb N}$ is uniformly bounded on compact sets in the $C_b^{4,\alpha}$ sense. Thus, by Ascoli-Arzela's Theorem, we can assume that $W_k\rightarrow W_\infty$, and, again, from pulling back \eqref{hpass0bis}, we obtain that $W_\infty$ is a $C^{4,\alpha}$-function on $X'$ decaying to infinity. Moreover, analyzing the pieces of the pullback
\[
\sigma_k^*(\mathcal{L}u_k)=\sigma_k^*\left(\Delta_\omega F_{\omega}(u_k)+n\frac{i\p\bpa(u_k\omega^{n-2})\wedge\text{Ric}^{ch}(\omega)}{\omega^n}-s^{ch}(\omega)F_{\omega}(u_k)\right)
\]
we can see that:
\begin{itemize}
    \item $\sigma_k^*\Delta_\omega(F_{\omega}u_k)=r_k^{-(b+4)}\Delta_{\theta_k}F_{k}(W_k)$, where $F_k:=F_{\theta_k}$;
    \item $\sigma_k^*\left(\frac{\text{Ric}^{ch} (\omega)\wedge i\partial \bar \partial (u_k\omega^{n-2})}{\omega^n}\right)=r_k^{-(b+4)}\left(\frac{\text{Ric}^{ch}(\theta_k)\wedge i\partial \bar \partial (W_k\theta_k^{n-2})}{\theta_k^n}\right)$;
    \item $\sigma_k^*(s^{ch}(\omega)F_{\omega}(u_k))=r_k^{-(b+4)}s^{ch}(\theta_k)F_k(W_k)$.
\end{itemize}
Moreover, it is easy to show that $F_k\to \frac{1}{n-1}\Delta_{\omega_o}$ and that, of course, $s^{ch}(\theta_k) \to s^{ch}(\omega_o)=0$, as $k \to\infty$. 
Hence, pulling back   \eqref{hpassbis}  with $\sigma_k$ and taking the limit in $k$, we obtain that $W_\infty$ is biharmonic on $X'$. Pulling back \eqref{hpass0bis} and recalling \eqref{pbpesobis}, we obtain that $W_\infty$ decays at infinity, implying necessarily that $W_\infty \equiv 0$ on $X'$. On the other hand, if we define the sequence $y_k:=\sigma_k^{-1}(x_k) \in X'$, it is straightforward to see that $|y_k|=1$, for all $k \in \mathbb N$. Hence, it can be assumed to be convergent to some $y_\infty$, which combined with the limit of pullback via $\sigma_k$ of \eqref{assrescbis}, implies $W_\infty(y_\infty)>0$, i.e. a contradicition with the fact that $W_\infty \equiv 0$.
\end{itemize}
Hence the thesis is proven.
\end{proof}
\begin{rmk}
    We stress how in the proof, the assumption of Chern-Ricci flatness of the metric $\tilde{\omega}$ has allowed us to face a significantly more approachable analytic problem by erasing the second degree terms involving the Chern-Ricci form and the Chern scalar curvature. Moreover, we suspect that the possibility to successfully work orthogonally to the kernel of the operator is connected to the fact that, in our Chern-Ricci flatness hypothesis, the kernel of the operator does not contain any holomorphic vector field, as we have seen in Lemma \ref{lemLich}. On top of this, Lemma \ref{lemLich} makes us hope that the result can be extended with additional analytic effort, at least in the case of non-positive constant Chern-scalar curvature balanced metrics, possibly exploiting the fact that the terms removed by the Chern-Ricci flatness assumptions are of (strictly) lower order.
\end{rmk}
 
The above estimate allows us to easily obtain the uniform invertibility.
\begin{lem}\label{invLbis}
The operator 
$$
\tilde{\mathcal  L}\colon \mathcal V_{\ezi}\to C^{0,\alpha}_{ b+4,\ezi}(\hM)
$$ is an isomorphism.
\end{lem}
\begin{proof}Thanks to Proposition \ref{injLbis}, we have that $\tilde{\mathcal L }$ is injective. Moreover, it is clear that $\tilde{\mathcal L }$ is elliptic and with the same index of $\Delta^2_\omega$, which is $0$. This automatically guarantees the claim.
\end{proof}

\subsection{Setting up the fixed point problem}\label{subsecopen}

We can now reformulate $\tilde{\mathcal{S}}(u)=0$ by considering the expansion 
$$
 s^{ch}(\omega_u)=s^{ch}(\omega)+ \mathcal L u + Q(u)\,,
$$ where $Q$ is the quadratic part of $s^{ch}(\omega_u)$. Then, \eqref{eqs3} can be rewritten as:
$$
s^{ch}(\omega)+ \tilde{\mathcal L }u + Q(u)=0\,,
$$ 
and hence, using Lemma \ref{invLbis}, we obtain that 
\begin{equation}\label{fix}
\mathcal N(u):= -\tilde{\mathcal L}^{-1}(s^{ch}(\omega)+ Q(u))=u\,.
\end{equation}
So, thanks to Banach's fixed-point Theorem, in order to find a solution to $\tilde{\mathcal{S}}(u)=0$, it is sufficient to show that $$
\mathcal {N}\colon\mathcal V_{\ezi}\to \mathcal V_{\ezi}
$$ is a contraction on a suitable open neighbourhood of zero in $\mathcal V_{\ezi}$.\par

In order to determine the open set we are looking for, we note that, if $\lVert\psi\rVert_{C^{4, \alpha}_{-2,\ezi}(\hM)}\le C\ezi^{\tau}$  for some $C,\tau>0$, then, 
\begin{equation}\label{ididibarbis}
\lVert i\partial \bar \partial(\psi\omega^{n-2}) \rVert_{C^{2, \alpha}_{0,\ezi}(\hM)}
\le C\lv\psi\rv_{C^{4, \alpha}_{-2, \ezi}(\hM)}\le C\ezi^{\tau}\,,
\end{equation} where the first inequality is due to the fact that $\lVert\omega^{n-2}\rVert_{C^{4,\alpha}_{0,\ezi}(\hM)}\le C$. Up to choosing $\ezi$ sufficiently small, this guarantees that $\omega_{\psi}^{n-1}>0$, hence provides a balanced metric, thanks to \cite{M}, as well as, using \eqref{ididibarbis}, 
\begin{equation}\label{diffbis}
\lv\omega_{\psi}^{n-1}-\omega^{n-1}\rv_{C^{2, \alpha}_{0,\ezi}(\hM)}=\lv i\partial \bar\partial(\psi\omega^{n-2})\rv_{C^{2, \alpha}_{0,\ezi}(\hM)}\le C\ezi^{\tau}\,.
\end{equation}
 Moreover, arguing as in \cite[Remark 2.8]{GS}, we can fix a point $y\in \hM$ and consider local holomorphic coordinates so that, in $y$, $\omega$ is the identity and $\omega_{\psi}$ is diagonal with eigenvalues $\lambda_i$. On the other hand, $\omega^{n-1}$ will be again the \lq\lq identity\rq \rq\, and $\omega_{\psi}^{n-1}$ will have eigenvalues $\Lambda_i$. But, thanks to \eqref{diffbis}, we know that $$
 \Lambda_i=1+O(\ezi^{\tau})\,,$$ which implies  that $\lambda_i=\left(\prod_{j\ne i}\Lambda_j\right)^{\frac{1}{n-1}}=1+O(\ezi^{\tau})\,.$ This last fact readily guarantees that 
\begin{equation}\label{diff2bis}
 \lv \omega_{\psi}-\omega\rv_{C^{2, \alpha}_{0, \ezi}(\hM)}\le C\ezi^{\tau}\,,
\end{equation} which,  in particular,  gives that $\omega_{\psi}\to \omega$, as $\ezi\to 0 .$
As in \cite{Sz} and \cite{GS}, we then consider the open set 
 $$
 U_{\tau}:=\{\psi\in \mathcal V_{\ezi}\quad | \quad \lv\psi\rv_{\alphabq(\hM)}\le C\ezi^{(p+q)(b+2)+ \tau}\}\,.
$$ We can readily note that, if $\psi\in U_{\tau}$, then 
\begin{equation}\label{suamarcezzbis}
\lv\psi\rv_{C^{4, \alpha}_{-2,\ezi}(\hM)}\le \ezi^{-(p+q)(b+2)}\lv\psi\rv_{\alphabq(\hM)}\le C\ezi^{\tau}\,,
\end{equation} where the first inequality is due to the fact that $\lv\psi\rv_{C^{k, \alpha}_{a, \ezi}(\hM)}\le \ezi^{(p+q)(-b+a)}\lv\psi\rv_{C^{k, \alpha}_{b, \ezi}(\hM)}$, for any $k\ge 0 $, $a\le b $, thanks to the definition of our weight. This inequality guarantees also that every $\psi \in U_\tau$ is not only small in the weighted sense, but it is so also in the standard sense, ensuring that our setting for the problem makes sense in this set.\\
We are thus left with the estimates to obtain that $\mathcal{N}$ preserves $U_\tau$ and it is a contraction on it.

\subsection{Weighted estimates}\label{subsecesti}
We first show that $\mathcal{N}$ contracts distances on $U_\tau$, which thanks to \eqref{fix}  and Lemma \ref{invLbis} reduces to showing that $Q$ contracts distances. Thus, fixed $\varphi_1, \varphi_2\in U_{\tau},$ the Mean value Theorem guarantees that there exists $t\in [0,1]$ such that, defined $\chi:=t\varphi_1+(1-t)\varphi_2\in U_{\tau}$,
 we have
$$
Q(\varphi_1)-Q(\varphi_2)=d_{\chi}Q(\varphi_1-\varphi_2)=(\mathcal L_{\chi}- \mathcal L )(\varphi_1-\varphi_2)\,.
$$   
We now need to compute $\mathcal L_{\chi}:=d_\chi\mathcal{S}$. As done in Subsection \ref{seceq},  we consider  the curve of Hermitian metrics  in $[\omega^{n-1}]_{BC}=[\omega_{\chi}^{n-1}]_{BC}$ defined by $\omega_{\chi, v}(s)^{n-1}:=\omega_{\chi}^{n-1}+si\partial \bar \partial (v\omega^{n-2})$, then, 
$$\mathcal L_{\chi}(v)=\frac{d}{ds}\Big|_{s=0}s^{ch}(\omega_{\chi, v}(s))\,.
$$ But, differentiating again \eqref{scal}, we obtain that 
$$
\mathcal L_{\chi}(v)\omega_{\chi}^n=n\frac{d}{ds}\Big|_{s=0}{\rm Ric}^{ch}(\omega_{\chi, v}(s))\wedge \omega_{\chi}^{n-1}+ n {\rm Ric}^{ch}(\omega_{\chi})\wedge \frac{d}{ds}\Big|_{s=0}\omega_{\chi}^{n-1}(s)- s^{ch}(\omega_{\chi})\frac{d}{ds}\Big|_{s=0}\omega_{\chi, v }(s)^{n}\,.
$$
As done in  \eqref{gigiKahler}, we  conclude that 
\begin{equation}\label{derchibis}
\begin{aligned}
\frac{d}{ds}\Big|_{s=0}\omega_{\chi, v}(s)^{n}= \frac{n}{n-1}i\partial \bar \partial (v\omega^{n-2})\wedge \omega_{\chi}\,,\quad  
\frac{d}{ds}\Big|_{s=0}{\rm Ric}^{ch}(\omega_{\chi, v}(s))= -i\partial \bar \partial F_\chi(v)\,,
\end{aligned}
\end{equation} where $F_{\chi}:=F_{\omega_{\chi}}$.
Then, we have that 
$$
\mathcal{L}_{\chi}(v)=-\Delta_{\omega_{\chi}}F_{\chi}(v)+n\frac{{\rm Ric}^{ch}(\omega_{\chi})\wedge i\partial \bar \partial (v\omega^{n-2}) }{\omega_{\chi}^n}- s^{ch}(\omega_{\chi})F_{\chi}(v)\,.
$$ Before going through the estimates, we need to explore the relation between the differential operators we are working with. First of all, we define the function
\begin{equation}\label{igiochino}
g(\chi):=\frac{\omega^n}{\omega_{\chi}^n}\,.
\end{equation} Then, for any  $v\in C^2(\hM)$, we have 
\begin{equation}\label{laplacesbis}
\Delta_{\omega_{\chi}}v=n\frac{i\partial \bar \partial v\wedge \omega_{\chi}^{n-1}}{\omega_{\chi}^n}=g(\chi)\left(\Delta_{\omega}v+ n\frac{i\partial \bar \partial v\wedge i\partial \bar \partial (\chi\omega^{n-2})}{\omega^n}\right)\,,
\end{equation} which gives us that
$$
\Delta_{\omega_{\chi}}v-\Delta_{\omega}v=(g(\chi)-1)\Delta_{\omega}v+ ng(\chi)\left(\frac{i\partial \bar \partial v\wedge i\partial \bar \partial (\chi\omega^{n-2})}{\omega^n}\right)\,.
$$For the sake of simplicity, we will denote 
\begin{equation}\label{defE}
E(v):=ng(\chi)\left(\frac{i\partial \bar \partial v\wedge i\partial \bar \partial (\chi\omega^{n-2})}{\omega^n}\right)
\end{equation} so that \eqref{laplacesbis} can be rewritten as 
\begin{equation}\label{laplaces2bis}
\Delta_{\omega_{\chi}}v=g(\chi)\Delta_{\omega}v+ E(v)\,.
\end{equation}
Moreover, we define 
$$
\begin{aligned}
G(v):=&\, n\frac{i\partial \bar \partial(v\omega^{n-2})\wedge {\rm Ric}^{ch}(\omega_{\chi})}{\omega_{\chi}^n}-n\frac{i\partial \bar \partial(v\omega^{n-2})\wedge {\rm Ric}^{ch}(\omega)}{\omega^n} + s^{ch}(\omega)F(v)- s^{ch}(\omega_{\chi})F_{\chi}(v)\\
=&\, n\frac{i\partial \bar \partial(v\omega^{n-2})\wedge \left(g(\chi){\rm Ric}^{ch}(\omega_{\chi})-{\rm Ric}^{ch}(\omega)-\frac{1}{n-1}(g(\chi)s^{ch}(\omega_{\chi})\omega_{\chi}-s^{ch}(\omega)\omega)\right)}{\omega^n}\,.
\end{aligned}
$$
Then, using \eqref{laplaces2bis} and these new notations, we have 
\begin{equation}\label{diffLbis}
\begin{aligned}
(\mathcal L_{\chi}-\mathcal L)(v)=&\, -\left(g(\chi)\Delta_{\omega}F_{\chi}(v)-\Delta_{\omega}F_{\omega}(u)\right) + E(F_{\chi}(v))+ G(v)\\
=&\, -g(\chi)\Delta_{\omega}(F_{\chi}- F_{\omega})(v)-(g(\chi)-1)\Delta_{\omega}F_{\omega}(v) + E(F_{\chi}(v) )+ G(v)\,.
\end{aligned}
\end{equation}
We will then breakdown the estimates in a series of smaller lemmas which will be used to conclude. The first one gives  estimates on the function $g$ defined in \eqref{igiochino}.
\begin{lem}\label{Lemmag}
Let $\chi\in U_{\tau}$, then
\begin{equation}\label{estig}
\lv g(\chi)-1\rv_{C^{2, \alpha}_{0,\ezi}(\hM)}\le C\ezi^{\tau}\,, \quad \lv g(\chi)\rv_{C^{2, \alpha}_{0,\ezi}(\hM)}\le 1+C\ezi^{\tau}\,.
\end{equation}
\end{lem}
\begin{proof}
Obviously, it is sufficient to prove the first inequality, since the second one can be recovered by that one using the triangle inequality and the fact that $\lv 1\rv_{C^{2, \alpha}_{0,\ezi}(\hM)}=1\,.$ In order to prove the first inequality in \eqref{estig}, we observe that 
$$
\begin{aligned}
\omega^n-\omega_{\chi}^n=&\, \omega^{n-1}\wedge(\omega-\omega_{\chi})+i\partial \bar \partial(\chi\omega^{n-2})\wedge (\omega-\omega_{\chi})-\omega\wedge i\partial \bar \partial(\chi\omega^{n-2}) \,.
\end{aligned}
$$ Now, from this, we have$$
\begin{aligned}
\lv \omega^n-\omega_{\chi}^n\rv_{C^{2, \alpha}_{0,\ezi}(\hM)}\le &\,  C(\lv\omega-\omega_{\chi}\rv_{C^{2, \alpha}_{0,\ezi}(\hM)}+ \lv i\partial \bar \partial(\chi \omega^{n-2})\rv_{C^{2, \alpha}_{0,\ezi}(\hM)}\lv\omega-\omega_{\chi}\rv_{C^{2, \alpha}_{0,\ezi}(\hM)}+ \lv i\partial \bar \partial(\chi\omega^{n-2})\rv_{C^{2, \alpha}_{0,\ezi}(\hM)})\,.
\end{aligned}
$$Thus, using  \eqref{ididibarbis},  \eqref{diffbis} and \eqref{diff2bis}, we have 
\begin{equation}\label{diffvol}
\lv \omega^n-\omega_{\chi}^n\rv_{C^{2, \alpha}_{0,\ezi}(\hM)}\le  C\ezi^{\tau}\,. 
\end{equation} Now, \eqref{diffvol} readily implies that 
$
\omega^n=\omega_{\chi}^n+ O(\ezi^{\tau})\,,
$  giving us the claim. 
\end{proof}
The next lemma shows the continuity of $E\colon C^{2, \alpha}_{b, \ezi}(\hM)\to C^{0,\alpha}_{b+2, \ezi}(\hM)$ and that its operator norm is bounded by $\ezi^{\tau}$, and it follows immediately from \eqref{estig} and \eqref{ididibarbis}.
\begin{lem}\label{lemmerrimobis}
 For $\ezi$ sufficiently small, we have that, for any $v\in C^{2, \alpha}_{b, \ezi}(\hM)$,
$$
\lv E(v)\rv_{C^{0,\alpha}_{b+2, \ezi}(\hM)}\le C\ezi^{\tau}\lv v\rv_{C^{2, \alpha}_{b, \ezi}(\hM)}\,.
$$
\end{lem}
Before showing the next lemma, we notice that, for any $v\in C^2(\hM)$, it holds
\begin{equation}\label{difflsbis}
\begin{aligned}
(F_{\chi}-F_{\omega})(v)
=&\, \frac{n}{n-1}\left(g(\chi)\left(\frac{i\partial \bar \partial (v\omega^{n-2})\wedge (\omega_{\chi}- \omega)}{\omega^n}\right)+ (g(\chi)-1)F_{\omega}(v))\right)\,.
\end{aligned}
\end{equation}
Again, the next lemma shows that $F_{\chi}-F_{\omega}\colon C^{4, \alpha}_{b, \ezi}(\hM)\to C^{2, \alpha}_{b+2, \ezi}(\hM)$ is a continuous operator with operator norm bounded by $\ezi^\tau$.
\begin{lem}\label{lemminobis}
For $\ezi$ sufficiently small,  for any $v\in C^{4, \alpha}_{b, \ezi}(\hM)$, we have that 
$$
\lv (F_{\chi}-F_{\omega})(v)\rv_{C^{2, \alpha}_{b+2, \ezi}(\hM)}\le C \ezi^{\tau}\lv v \rv_{\alphabq(\hM)}\,.
$$
\end{lem}
\begin{proof}
Thanks to \eqref{difflsbis} and Lemma \ref{Lemmag}, we can obtain that
$$
\begin{aligned}
\lv (F_{\chi}-F_{\omega})(v)\rv_{C^{2, \alpha}_{b+2, \ezi}(\hM)}\le&\,  C\left((1+\ezi^{\tau})\lv\omega_{\chi}-\omega\rv_{C^{2, \alpha}_{0,\ezi}(\hM)}\lv i\partial \bar \partial(v\omega^{n-2})\rv_{C^{2, \alpha}_{b+2,\ezi}(\hM)}+ \ezi^{\tau}\lv F_{\omega}(v)\rv_{C^{2, \alpha}_{0,\ezi}(\hM)}\right)\\
\le &\, C\left((1+\ezi^{\tau})\lv\omega_{\chi}-\omega\rv_{C^{2, \alpha}_{0,\ezi}(\hM)}\lv v\rv_{C^{4, \alpha}_{b,\ezi}(\hM)}+ \ezi^{\tau}\lv F_{\omega}(v)\rv_{C^{2, \alpha}_{0,\ezi}(\hM)}\right)\,.
\end{aligned}
$$ Now, we can use \eqref{estig}, \eqref{diff2bis} and the continuity of $F\colon C^{4, \alpha}_{b, \ezi}(\hM)\to C^{2, \alpha}_{b+2, \ezi}(\hM)$ to obtain that 
$$
\lv (F_{\chi}-F_{\omega})(v)\rv_{C^{2, \alpha}_{b+2, \ezi}(\hM)}\le C((1+\ezi^{\tau})\ezi^{\tau})\lv v\rv_{C^{4, \alpha}_{b,\ezi}(\hM)}+ \ezi^{\tau}\lv v\rv_{C^{4, \alpha}_{b,\ezi}(\hM)})\le C \ezi^{\tau}\lv v\rv_{C^{4, \alpha}_{b,\ezi}(\hM)}\,,
$$ concluding the proof. 
\end{proof}
 It remains to analize the operator $G$. In order to do so, we need  two more estimates. 
 \begin{lem}\label{lemmissimo}
 For $\ezi$ sufficiently small,  we have:
 $$
 \lv g(\chi){\rm Ric}^{ch}(\omega_{\chi})- {\rm Ric}^{ch}(\omega)\rv_{C^{0,\alpha}_{2, \ezi}(\hM)}\le C\ezi^{\tau}\,, \quad\lv g(\chi)s^{ch}(\omega_{\chi})\omega_{\chi}- s^{ch}(\omega)\omega\rv_{C^{0,\alpha}_{2, \ezi}(\hM)}\le C\ezi^{\tau} \,.
 $$
 \end{lem}
 \begin{proof}
  We have that
  $$
  \begin{aligned}
  g(\chi){\rm Ric}^{ch}(\omega_{\chi})- {\rm Ric}^{ch}(\omega)=&\, g(\chi)({\rm Ric}^{ch}(\omega_{\chi})- {\rm Ric}^{ch}(\omega))+ (g(\chi)-1){\rm Ric}^{ch}(\omega)\\
  =&\, g(\chi)i\partial \bar \partial \log g(\chi)+ (g(\chi)-1){\rm Ric}^{ch}(\omega)\,.
  \end{aligned}
  $$ On the other hand, we have 
  \begin{equation}\label{rigscalpes}
  \lv {\rm Ric}^{ch}(\omega) \rv_{C^{0,\alpha}_{2, \ezi}(\hM)}\le C\,, \quad  \lv s^{ch}(\omega) \rv_{C^{0,\alpha}_{2, \ezi}(\hM)}\le C\,.
  \end{equation} Indeed,  we know that $\omega=\omega_o+O(\lvert z\rvert^{m})$, implying that $\omega^n=\omega_o^n+O(\lvert z\rvert^m)$. This allows us to infer that
  $${\rm Ric}^{ch}(\omega)=O(\lvert z\rvert^{m-2})\,, \quad s^{ch}(\omega)=O(\lvert z\rvert^{m-2})\,.$$ Trivially, this gives that 
  $$
  \rho^2{\rm Ric}^{ch}(\omega)=O(\lvert z\rvert^{m})\,, \quad \rho^2s^{ch}(\omega)=O(\lvert z\rvert^{m}),
  $$ obtaining the claim.  Now, using \eqref{estig} and \eqref{rigscalpes}, we have that 
  \begin{equation}\label{rigido}
  \begin{aligned}
  \lv g(\chi){\rm Ric}^{ch}(\omega_{\chi})- {\rm Ric}^{ch}(\omega)\rv_{C^{0,\alpha}_{2, \ezi}(\hM)}\le&\, 
   C(1+\ezi^{\tau})\lv\log g(\chi) \rv_{C^{2,\alpha}_{0,\ezi}(\hM)}+ C\ezi^{\tau}\,.
  \end{aligned}
\end{equation}

But if we now recall again inequalities \eqref{estig}, we can use the Taylor expansion and obtain from \eqref{rigido} the first claim. 
 As for the second one, we observe that 
 $$
 g(\chi)s^{ch}(\omega_{\chi})\omega_{\chi}- s^{ch}(\omega)\omega=g(\chi)(s^{ch}(\omega_{\chi})- s^{ch}(\omega))\omega_{\chi}+ g(\chi)s^{ch}(\omega)(\omega_{\chi}-\omega)+ (g(\chi)-1)s^{ch}(\omega)\omega\,.
 $$ 
Moreover, using \eqref{estig}, \eqref{diff2bis} and \eqref{rigscalpes}, we have that 
 \begin{equation}\label{stimescal}
 \begin{aligned}
 \lv g(\chi)s^{ch}(\omega_{\chi})\omega_{\chi}- s^{ch}(\omega)\omega\rv_{C^{0,\alpha}_{2, \ezi}(\hM)}\le&\,  C(1+\ezi^{\tau})\lv (s^{ch}(\omega_{\chi})- s^{ch}(\omega))\omega_{\chi}\rv_{C^{0,\alpha}_{2, \ezi}(\hM)}+ C(1+\ezi^{\tau})\ezi^{\tau}+ C\ezi^{\tau}\\
 \le&\,  C\ezi^{\tau}+ C(1+\ezi^{\tau})\lv s^{ch}(\omega_{\chi})- s^{ch}(\omega)\rv_{C^{0,\alpha}_{2, \ezi}(\hM)}\lv\omega_{\chi}\rv_{C^{0,\alpha}_{0, \ezi}(\hM)}\,.
\end{aligned}
\end{equation}
Again, using \eqref{diff2bis}, we can infer  that 
$$
\lv\omega_{\chi}\rv_{C^{0,\alpha}_{0, \ezi}(\hM)}\le \lv\omega\rv_{C^{0,\alpha}_{0, \ezi}(\hM)} + \lv\omega_{\chi}- \omega\rv_{C^{0,\alpha}_{0, \ezi}(\hM)}\le C(1+\ezi^{\tau})\,,
$$ which put into \eqref{stimescal} gives that 
\begin{equation}\label{madonna}
\lv g(\chi)s^{ch}(\omega_{\chi})\omega_{\chi}- s^{ch}(\omega)\omega\rv_{C^{0,\alpha}_{2, \ezi}(\hM)}\le C\ezi^{\tau}+ C(1+\ezi^{\tau})^2\lv s^{ch}(\omega_{\chi})- s^{ch}(\omega)\rv_{C^{0,\alpha}_{2, \ezi}(\hM)}\,.\end{equation}
On the other hand, we have
$$
\begin{aligned}
s^{ch}(\omega_{\chi})-s^{ch}(\omega)
= & n \frac{{\rm Ric}^{ch}(\omega_{\chi})\wedge \omega^{n-1}}{\omega_{\chi}^{n}}+ n\frac{{\rm Ric}^{ch}(\omega_{\chi})\wedge i\partial \bar \partial (\chi \omega^{n-2})}{\omega_{\chi}^n}-s^{ch}(\omega)\\
=&\,(g(\chi)-1)s^{ch}(\omega)+ ng(\chi)\frac{i\partial \bar \partial \log g(\chi)\wedge \omega^{n-1}}{\omega^{n}}+ ng(\chi)\frac{{\rm Ric}^{ch}(\omega_{\chi})\wedge i\partial \bar \partial (\chi \omega^{n-2})}{\omega^{n}}\\
=&\, (g(\chi)-1)s^{ch}(\omega)+ g(\chi)\Delta_{\omega}\log(g(\chi))+ g(\chi)\frac{{\rm Ric}^{ch}(\omega_{\chi})\wedge i\partial \bar \partial (\chi \omega^{n-2})}{\omega^{n}}\,.
\end{aligned}
$$  Then, using again \eqref{estig} and \eqref{rigscalpes},
$$
\begin{aligned}
\lv s^{ch}(\omega_{\chi})- s^{ch}(\omega)\rv_{C^{0,\alpha}_{2, \ezi}(\hM)}\le&\,  C\ezi^{\tau}+ C(1+\ezi^{\tau})\ezi^{\tau} + C(1+\ezi^{\tau})\lv {\rm Ric}^{ch}(\omega_{\chi})\wedge i\partial \bar \partial(\chi\omega^{n-2})\rv_{C^{0,\alpha}_{2, \ezi}(\hM)}\\
\le &\, C\ezi^{\tau}+ C(1+\ezi^{\tau})\lv {\rm Ric}^{ch}(\omega_{\chi})\rv_{C^{0,\alpha}_{2, \ezi}(\hM)}\lv i\partial \bar \partial(\chi \omega^{n-2})\rv_{C^{0,\alpha}_{0, \ezi}(\hM)}\,.
\end{aligned}
$$ But, we have
\begin{equation}\label{ultime}
\begin{aligned}
\lv {\rm Ric}^{ch}(\omega_{\chi})\rv_{C^{0,\alpha}_{2, \ezi}(\hM)}\le&\,  \lv {\rm Ric}^{ch}(\omega)\rv_{C^{0,\alpha}_{2, \ezi}(\hM)} + \lv i\partial \bar \partial \log g(\chi)\rv_{C^{0,\alpha}_{2, \ezi}(\hM)}\le C(1+\ezi^{\tau})\,,\\
\lv i\partial \bar \partial(\chi \omega^{n-2})\rv_{C^{0,\alpha}_{0, \ezi}(\hM)}\le&\,  \lv \chi \omega^{n-2}\rv_{C^{2,\alpha}_{-2, \ezi}(\hM)}\le C\lv \chi \rv_{C^{4,\alpha}_{-2, \ezi}(\hM)}\le C\ezi^{\tau}\,, 
\end{aligned}
\end{equation} where the last inequality is due to \eqref{suamarcezzbis}.
 Putting \eqref{ultime} into \eqref{madonna}, we have the claim.
 \end{proof}
  Thus, using  Lemma \ref{lemmissimo}, we can conclude that 
  \begin{equation}\label{gdiv}
  \begin{aligned}
\lv G(v)\rv_{C^{0,\alpha}_{ b+4, \ezi}(\hM)}\le&\,  C\lv i\partial \bar \partial (v\omega^{n-2})\rv_{C^{0,\alpha}_{b+2}(\hM)}\left\lv g(\chi){\rm Ric}^{ch}(\omega_{\chi})-{\rm Ric}^{ch}(\omega)-\frac{g(\chi)s^{ch}(\omega_{\chi})\omega_{\chi}-s^{ch}(\omega)\omega}{n-1}\right\rv_{C^{0,\alpha}_{2, \alpha}(\hM)}\\
\le&\,  C\ezi^{\tau}\lv v\rv_{C^{4, \alpha}_{b, \ezi}(\hM)}\,.
\end{aligned}
\end{equation}
We are finally ready to prove that $\mathcal{N}$ is a contraction operator on $U_\tau$.
\begin{prop}\label{Nbal}
     For $\ezi$ sufficiently small and $b<2n-4$, the operator $\mathcal N$ is a contraction and $\mathcal N(U_{\tau})\subseteq U_{\tau}$.
\end{prop}
\begin{proof}
  Consider $v=\varphi_1-\varphi_2$ as above, 
  $$
  \lv\mathcal N(\varphi_1)- \mathcal N(\varphi_2) \rv_{C^{4, \alpha}_{b, \ezi}(\hM)}\le C\lv(\mathcal L_{\chi}- \mathcal L )(v)\rv_{C^{0,\alpha}_{b+4,\ezi}(\hM)}\,.
  $$ Using \eqref{diffLbis}, \eqref{estig}, Lemma \ref{lemminobis}, Lemma \ref{lemmerrimobis} and \eqref{gdiv} and the continuity of $\Delta_{\omega}\colon C^{4, \alpha}_{b,\ezi}(\hM)\to C^{2, \alpha}_{b+2}(\hM)$ and that of $F_{\omega}\colon C^{2, \alpha}_{b+2, \ezi}(\hM)\to C^{0,\alpha}_{b+4, \ezi}(\hM)$, we have 
  $$
  \lv(\mathcal L_{\chi}- \mathcal L )(v)\rv_{C^{0,\alpha}_{b+4,\ezi}(\hM)}\le  C\ezi^{\tau}\lv v\rv_{C^{4, \alpha}_{b, \ezi}(\hM)}
  $$ which, after choosing $\ezi$ sufficiently small, guarantees that $\mathcal N$ is a contraction.  Now fix $\varphi\in U_{\tau}$, we have that 
  $$
  \begin{aligned}
\lv\mathcal N(\varphi)\rv_{C^{4, \alpha}_{b, \ezi}(\hM)}\le&\,  \lv\mathcal N(0)\rv_{C^{4, \alpha}_{b, \ezi}(\hM)}+ \lv\mathcal N(\varphi)-\mathcal N  (0)\rv_{C^{4, \alpha}_{b, \ezi}(\hM)}\le\lv\mathcal N(0)\rv_{C^{4, \alpha}_{b, \ezi}(\hM)}+ C\ezi^{\tau}\lv\varphi\rv_{C^{4, \alpha}_{b, \ezi}(\hM)}\\
\le &\, \lv\mathcal N(0)\rv_{C^{4, \alpha}_{b, \ezi}(\hM)}+ C\ezi^{2\tau+ (p+q)(b+2)}\le \lv\tilde{\mathcal L}^{-1}(s^{ch}(\omega))\rv_{C^{4, \alpha}_{b, \ezi}(\hM)}+ C\ezi^{2\tau+ (p+q)(b+2)}\\
\le&\,   C\lv s^{ch}(\omega)\rv_{C^{0, \alpha}_{b+4, \ezi}(\hM)}+ C\ezi^{2\tau+ (p+q)(b+2)}\,.
\end{aligned}
$$ 
On the other hand, 
$$
\lv s^{ch}(\omega)\rv_{C^{0, \alpha}_{b+4, \ezi}(\hM)}\le C\ezi^{p(m+ b+2)}\,,
$$ 
from which it follows
$$
\lv\mathcal N(\varphi)\rv_{C^{4, \alpha}_{b, \ezi}(\hM)}\le C\ezi^{p(m+b+2)}+ C\ezi^{2\tau+ (p+q)(b+2)}\le C\ezi^{\min\{\tau, pm-q(b+2)- \tau\}}\ezi^{\tau+(p+q)(b+2)}\,.
$$  It is then sufficient to notice that $\tau$ can be chosen such that $pm-q(b+2)>\tau >0$, giving us the claim.
\end{proof}

Hence also Theorem \ref{ansbal} is proven.\par
The construction done to prove Theorem \ref{ansbal} can also be used in the case in which the chosen points are not smooth, provided the resolutions of the singularity model at such points satisfy some extra conditions. More precisely, we need to impose the following:
\begin{enumerate}
\item \label{hyporb} for any $x\in M$, let $G_x\subseteq {\rm SU}(n)$ acting freely on $\C^n\backslash \{0\}$ so that a suitable  neighbourhood of $p$ is biholomorphic to a neighbourhood of $\C^n/G_x$. $\C^n/G_x$ has a ALE resolution $(X, \omega_{{\rm ALE}})$, where $\omega_{\rm ALE}$ is a scalar flat ALE K\"ahler metric on $X$ such that, away from the exceptional divisors, takes the following form: 
\begin{equation}\label{asymp}
\omega_{{\rm ALE}}=\omega_o+i\partial \bar \partial \gamma\,, \quad \gamma(r)= O(r^{4-2n})\,.
\end{equation}
\end{enumerate}
 Once Condition \eqref{hyporb} is satisfied, we can repeat all the steps substituting the Burns-Simanca metric with $\omega_{{\rm ALE}}$ and conclude.\par
It is also clear that this setting can be considered in the case of an orbifold admitting crepant resolutions, since, as we recalled in Subsection \ref{preli}, the singularity resolution model carries Kähler Ricci-flat ALE metrics with fast decay. In particular, one can easily adapt the proof of Theorem \ref{ansbal} (where the key fact in repeating the proof is Lemma \ref{deformann})  to prove Theorem \ref{glueorb} which gives an extension of \cite[Theorem 1]{GS} to the general case of balanced Chern-Ricci flat orbifolds admitting crepant resolutions.

\section{Non-positive trace deformation}\label{nonpostr}
The second ansatz we will be considering is based on assuming  the existence of a suitable $(n-2,n-2)$-form to restrict the deformation argument to an easier subspace of the balanced class. More specifically, we will assume that ${M}$ is endowed with $\tilde\Omega\in \Lambda_{\R}^{n-2,n-2}(M)$   and such that
$$
\tilde\omega \wedge \tilde\Omega>0 \quad \mbox{ and }\quad \Lambda_{\tilde\omega}^{n-1}(i\partial \bar \partial \tilde\Omega)\le0\,.
$$ 

Then, we can consider $\ezi>0$ sufficiently small and a small neighbourhood  of $x$, which will be identified with $B(0,1)\subset \C^n$,  with holomorphic coordinates $z$ on $M$. As before, we can consider the  cut-off function $\chi\colon[0,1]\to [0,1] $ from Lemma \ref{flatcutoff} and consider 
$$
\tilde{\Omega}_{\ezi}=(1-\chi_\ezi(|z|))\tilde{\Omega}+ \chi_\ezi(|z|)\omega_o^{n-2}\in \Lambda_{\R}^{n-2, n-2}(M)
$$ 
where again $\omega_o$ is the flat metric induced by $z$ on $B(0,1)$. 
As for $\tilde\omega_{\ezi}$, $\tilde \Omega_{\ezi}$ coincides with the $(n-2)$-th power of the flat metric in a small neighbourhood of $x$, hence we can repeat the strategy to construct $\omega$, and glue together $\tilde{\Omega}_{\ezi}$ with $\ezi^{2(n-2)(p+q)}\omega_{BS, \ezi}^{n-2}$ on the flat region \eqref{regione}, obtaining  $\Omega\in \Lambda_{\R}^{n-2, n-2}(\hM)$.
\begin{rmk}\label{radOm}
It is easy to check that $\omega\wedge\Omega>0$. Indeed, it is straightforward from the construction that the only region in which we have to check this is the cut-off region $\{\varepsilon^p \leq |z| \leq 2\varepsilon^p\}$, in which we have
\[
\omega\wedge\Omega=(\omega_o+O(|z|))\wedge((1-\chi)\tilde{\Omega}+\chi\omega_o^{n-2})=(1-\chi)\omega_o\wedge\tilde{\Omega}+\chi\omega_o^{n-1}+O(|z|),
\]
which is positive,  for sufficiently small $\ezi$. Indeed,  up to the decaying term,  it is a pointwise convex combination of positive forms. This, thanks to the work of Michelsohn (\cite{M}), also implies that $\omega\wedge\Omega=(\omega')^{n-1}$, where $\omega'$ is a Hermitian metric on $\hM$.\par
On the other hand, the condition $\Lambda_{\tilde{\omega}}(i\p\bpa\tilde{\Omega})\le0$ might not be preserved, but, as we will see, we will just need it on the base manifold.
\end{rmk}
Thanks to this construction, we can thus choose $\varphi=u\Omega$ with $u \in C^{\infty}(\hM, \R)$ such that $\omega_u^{n-1}:=\omega_{u\Omega}^{n-1}>0$, and along with assuming $f(u\Omega)=u$, we are able to turn again the operator $\tilde{\mathcal S}$ into an operator taking smooth functions in input defined as:
\begin{equation}\label{opfunctions}
\tilde{\mathcal{S}}(u)=s^{ch}(\omega_u)-s^{ch}( \tilde{\omega})-\int_{\hM} u\frac{\omega^n}{n!}\,.
\end{equation}
Let us start by writing again the linearized operator $\tilde{\mathcal{L}}$ implementing this new ansatz:
\begin{equation}\label{pippoLOm}
    \tilde{\mathcal L}(u):=\tilde{\mathcal L}(u \Omega)=-\Delta_\omega F_{\omega,\Omega}(u)+n\frac{\text{Ric}^{ch}(\omega)\wedge i\p\bpa(u\Omega)}{\omega^n}-s^{ch}(\omega)F_{\omega,\Omega}(u),
\end{equation}
where \begin{equation}\label{FOme}
F_{\omega,\Omega}(u):=F_\omega(u\Omega)=\frac{n}{n-1}\frac{\omega\wedge i\p\bpa(u\Omega)}{\omega^n}.
\end{equation}\par
 Now, the proof of Theorem \ref{immain} is considerably simpler than that of Theorem \ref{ansbal}. This is mainly due to the fact that there is no need to restrict our attention on suitably chosen subspaces. Indeed, the following lemma asserts  that the operator $F_{\omega, \Omega}$ has trivial kernel, when restricted to zero-mean value functions, on compact manifolds.

 \begin{lem}\label{lemmaF}
 Let $(M^n, \tilde\omega)$ be a compact balanced manifold and let $\tilde\Omega\in \Lambda_{\R}^{n-2, n-2}(M)$ satisfying \eqref{hypOm}. If 
 $$
 F_{\tilde \omega, \tilde \Omega} \colon C^{\infty}(M, \R)\to C^{\infty}(M, \R)\,,  \quad  F_{\tilde \omega, \tilde \Omega}(u):=\frac{n}{n-1}\frac{\tilde \omega\wedge i\partial \bar \partial (u\tilde \Omega)}{\tilde \omega^n}\,, 
$$then,  there are no non-trivial functions $u\in C^{\infty}(M, \R)$ such that 
$$
F_{\tilde \omega, \tilde \Omega}(u)=c\,,\quad c\in \R\, \quad \mbox{and } \int_M u \frac{\tilde \omega^n}{n!}=0\,.
$$
 Then, the restriction of $F_{\tilde \omega, \tilde \Omega}$ to smooth functions with zero $\tilde\omega$-mean value is injective.
 \end{lem}
\begin{proof}
Expanding the definition of $F_{\tilde\omega, \tilde \Omega}$, we have 
$$
 F_{\tilde \omega, \tilde \Omega}(u)=\frac{n}{n-1}\left(\frac{i\partial \bar \partial u \wedge \tilde \Omega \wedge\tilde \omega}{\tilde \omega^n}+ 2\Re \left(\frac{i\partial u \wedge \bar \partial \tilde \Omega \wedge \tilde \omega}{\tilde \omega^n}\right)+ u \frac{i\partial \bar \partial \tilde \Omega \wedge \tilde \omega}{\tilde \omega^n}\right)\,.
$$
 On the other hand, we know that  $\tilde \Omega\wedge \tilde \omega$ is a positive $(n-1, n-1)$-form. Then, thanks to a result in \cite{M}, there exists a Hermitian metric $\tilde\omega'$ such that $\tilde \omega'^{n-1}=\tilde \Omega \wedge \tilde \omega$. This implies that 
 $$
n\frac{i\partial \bar \partial u \wedge \tilde \Omega \wedge\tilde \omega}{\tilde \omega^n}=n\frac{i\partial \bar \partial u \wedge \tilde \omega'^{n-1} }{\tilde \omega^n}=\frac{\tilde \omega'^{n}}{\tilde \omega^n}\Delta_{\tilde \omega'}u\,.
$$Moreover, 
 $$
\frac{i\partial \bar \partial \tilde \Omega\wedge \tilde \omega}{\tilde \omega^n}
=\frac{1}{n!(n-1)!}\tilde g(i\partial \bar \partial \tilde \Omega, \tilde \omega^{n-1})=\frac{1}{n!(n-1)!}\Lambda_{\tilde \omega}^{n-1}(i\partial \bar \partial \tilde \Omega)\,.
$$ 
 Then, we can conclude that  
\begin{equation}\label{expF}
 F_{\tilde \omega, \tilde \Omega}(u)=\frac{1}{n-1}\frac{\tilde \omega'^{n}}{\tilde{\omega}^n}\Delta_{\tilde \omega'}u + \frac{2n}{n-1}\Re \left(\frac{i\partial u \wedge \bar \partial \tilde \Omega \wedge \tilde \omega}{\tilde\omega^n}\right)+ \frac{u}{(n-1)((n-1)!)^2}\Lambda_{\tilde \omega}^{n-1}(i\partial \bar \partial \tilde \Omega)\,,
\end{equation}  which is a second order elliptic operator with non-positive zero order coefficient. 
Now, if   $u\in C^{\infty}(M, \R)\backslash\{0\}$ with $\int_Mu\frac{\tilde \omega^n}{n!}=0$, then $\max_Mu>0$. On the other hand, thanks to the maximum principle, being  a solution of 
$$
F_{\tilde \omega, \tilde \Omega}(u)=c, \quad c>0\,,
$$ allows us to conclude that $u$ is  constant. Using again  the condition $\int_Mu\frac{\tilde\omega^n}{n!}=0$,  we conclude $u=0$.
The case in which $c<0$ can be deduced by the above, applying the same argument to $v=-u$ satisfying 
$$
F_{\tilde \omega, \tilde \Omega}(v)=-c>0\,, 
$$obtaining the claim.\end{proof}

 \begin{prop}\label{injLbisOmega}
 For any $b\in(0,2n-4)$, there exists $C>0$ such that,  for all  $u\in C^{4, \alpha}_{b, \ezi}(\hM)$, we have  $$\lVert u \rVert_{C^{4, \alpha}_{b, \ezi}(\hM)}\le C\lVert\tilde{\mathcal L }u\rVert_{C^{0, \alpha}_{b+4, \ezi}(\hM)}\,.
 $$ Then, $$
 \tilde{\mathcal L}\colon C^{4, \alpha}_{b, \ezi}(\hM)\to C^{, \alpha}_{b+4, \ezi}(\hM)
 $$ is a isomorphism. 
 \end{prop}
 \begin{proof} The proof of the first part follows the ideas of the proof of Lemma \ref{invLbis}. In this case, we have that $u_{\infty}\in C^{4,\alpha}_b(M_x)$ is such that 
$$
F_{\tilde\omega, \tilde \Omega}(u_{\infty})=c\,, \quad c \in \R\,,\quad \int_{M_x}u_{\infty}\frac{\tilde\omega^n}{n!}=0\,.
$$
 Now, using \eqref{expF}, we can use a bootstrap argument to infer that $F_{\tilde \omega, \tilde \Omega}(u_{\infty})=c$ on the whole $M$, then, in particular, it is smooth. Now, we can conclude that $u_{\infty}=0$, using Lemma \ref{lemmaF}\,. As regards the second part, one may notice that the index of $\tilde{\mathcal L}$ is equal to that of $\Delta_\omega\circ \Delta_{\omega'}$, where $\omega'$ is the Hermitian metric such that $\Omega\wedge \omega=(\omega')^{n-1}$, as in Remark \ref{radOm}, which is $0$. On the other hand, the first part of the statement is telling that $\tilde{\mathcal L}$ is injective and thus, since of  index $0$,  surjective. 
\end{proof}
Now, the rest of the proof of Theorem \ref{immain} goes as that of Theorem \ref{ansbal}. Indeed, we can now reformulate $\tilde{\mathcal{S}}(u)=0$ as  the following fixed point problem:
$$
\mathcal N(u):= -\tilde{\mathcal L}^{-1}(s^{ch}(\omega)+ Q(u))=u\,
$$
 where 
 $$\mathcal N\colon C^{4, \alpha}_{b, \ezi}(\hM )\to C^{4, \alpha}_{b, \ezi}(\hM)\,.
 $$
\begin{prop}
     For $\ezi$ sufficiently small and $b<2n-4$, the operator $\mathcal N$ is a contraction and $\mathcal N(U_{\tau})\subseteq U_{\tau}$,  where $$
 U_{\tau}:=\{\psi\in \alphabq(\hM)\quad | \quad \lv\psi\rv_{\alphabq(\hM)}\le C\ezi^{(p+q)(b+2)+ \tau}\}\,.
$$ 
\end{prop}
\begin{proof}
First of all, we notice that
 $$
 \lVert\Omega\rVert_{C^{4,\alpha}_{0,\ezi}(\hM)}\le C\,.
 $$
 Indeed, we  recall that 
 $$
 \Omega=\begin{cases}
 \ezi^{2(n-2)(p+q)}\omega_{BS, \ezi}^{n-2} \quad & |z|\le \ezi^p\,;\\
 \tilde \Omega' \quad & \ezi^p< |z| <  2\ezi^p\,;\\
 \tilde \Omega \quad & |z|\ge 2\ezi^p\,,
 \end{cases}\, \quad \omega=\begin{cases}\ezi^{2(p+q)}\omega_{BS, \ezi} \quad & |z|\le\ezi^p\,;\\
  \omega_{\ezi}' \quad & \ezi^p< |z| <  2\ezi^p\,;\\
 \tilde \omega \quad & |z|\ge 2\ezi^p\,. 
 \end{cases}
 $$ Of course, $|\nabla^k\Omega|_{\omega}\le C $, if $|z|\ge 2\ezi^p$ and $0\le |z|\le \ezi^p$, for all $k=0,\ldots, 4.$ On the other hand, both $\tilde \Omega' $ and $\omega_{\ezi}'$ depend on $\ezi$ just for their domain of definition and so we can infer that $|\nabla^k\Omega|_{\omega}\le C$ on the whole $\hM$,  giving the claim. Once, we have this, the proof is analogous to that of Proposition \ref{Nbal}, using an easily adapted version of the estimates found in Subsection \ref{subsecesti}.  
\end{proof}
Hence Theorem \ref{immain} is proven.\par
In the same fashion as discussed in the last part of Section \ref{baldef}, we can prove the same result even for singular points provided Condition \eqref{hyporb} is satisfied. 
In particular, again one can repeat the proof of Theorem \ref{immain} (where the key fact consists in repeating the proof of Proposition \ref{injLbisOmega})  to prove Theorem \ref{glueorbOm} which can be considered as a variation of \cite[Theorem 1]{GS} and Theorem \ref{glueorb} with this new ansatz.
In Section \ref{examples} we will discuss some examples in which Theorem \ref{immain} and Theorem \ref{glueorbOm} can be applied.

\section{Examples}\label{examples} 
In this section, we will describe families of compact Chern-Ricci flat balanced manifolds over which we can apply Theorem \ref{immain}, Theorem \ref{glueorbOm} and Theorem \ref{glueorb} (obviously along Theorem \ref{ansbal}, which can be applied to all compact Chern-Ricci flat balanced manifolds). Some examples will also allow us to show the fact that the kernel of the operator $F_{\tilde\omega}$ can either be zero or positive dimensional on the same manifold, strictly depending on the choice of the metric. We shall also see that the hypothesis of Theorem \ref{immain} identify a strict subset of compact Chern-Ricci flat balanced manifolds.

\subsection{Non-positive trace examples}\label{exOmega}
A very special case in which we can apply Theorem \ref{immain} is when $i\p\bpa\tilde \Omega=0$. This setting can arise in multiple scenarios, one of which is given by the case in which the balanced manifold $(M^n,\tilde{\omega})$ carries also \textit{astheno-Kähler} metrics (introduced by Jost and Yau in \cite{JY}), i.e. Hermitian metrics with fundamental form $\eta$ such that
$$i\partial \bar \partial \eta^{n-2}=0\,,$$
whose coexistence with the balanced structure was shown (by Fino, Grantcharov and Vezzoni in \cite{FGV} and by Latorre and Ugarte in \cite{LU}) not to force the existence of a Kähler metric. In this case, the $(n-2)$-th power of $\eta$ is trivially satisfying conditions \eqref{hypOm}, thus they are very natural to be considered.\par
Once compact balanced manifolds admitting also an astheno-Kähler structure are found, it is not hard to find they carry Chern-Ricci flat balanced metrics, as we can see in the following remark.
\begin{rmk}
    The existence of Chern-Ricci flat balanced metrics in a given balanced class when the manifold carries astheno-Kähler metrics depends only on the first Bott-Chern class. Indeed, in \cite{STW}, generalizing the same result on K\"ahler manifolds in \cite{TW1}, the authors proved, as a direct consequence of the solvability of suitable complex Monge-Ampère equations,  that  on a compact balanced manifold $(M^n, \tilde \omega )$ admitting an astheno-K\"ahler metric with $c_1^{BC}(M)=0$, we can always find a Chern-Ricci flat balanced metric in $[\tilde\omega^{n-1}]_{BC}$.
\end{rmk}

A first explicit example in complex dimension $4$ for this setting is the following.

\begin{exa}[Fino, Grantcharov, Vezzoni \cite{FGV}]
    Consider the $T^2$-principal bundle $\pi:M\rightarrow T^6$, where $T^6$ has the standard complex structure with holomorphic coordinates $(z_1,z_2,z_3)$, and with characteristic classes 
    \[
        a_1:=dz_1\wedge d\bar{z}_1+dz_2\wedge d\bar{z}_2-2dz_3\wedge d\bar{z}_3 \quad \text{and} \quad a_2:=dz_2\wedge d\bar{z}_2-dz_3\wedge d\bar{z}_3.
    \]
    We  then consider the Kähler metrics
    \[
        \eta_1:=dz_1\wedge d\bar{z}_1+dz_2\wedge d\bar{z}_2+dz_3\wedge d\bar{z}_3 \quad \text{and} \quad \eta_2:=dz_1\wedge d\bar{z}_1+dz_2\wedge d\bar{z}_2+5dz_3\wedge d\bar{z}_3
    \]
    on $T^6$, the connection $1$-forms $\theta^j$, $j=1,2$, such that $d\theta^j=\pi^*a_j$, and define
    \[
    \tilde\omega_j:=\pi^*\eta_j+\theta^1\wedge\theta^2, \quad j=1,2,
    \]
    which correspond respectively to a balanced and an astheno-Kähler metric on $M$. Moreover, denoting with $z_0$ the holomorphic coordinate on $T^2$,  it is straightforward to notice that
    \[
    \Theta:=d z_0 \wedge dz_1 \wedge d z_2 \wedge d z_3,
    \]
    defines a global holomorphic volume form, from which we also see that $\tilde\omega_1$  is Chern-Ricci flat, hence satisfying all the hypothesis are into place to apply Theorem \ref{immain}.
\end{exa}

A further family of examples was given in any complex dimension $n\geq 4$ as follows.

\begin{exa}[Latorre, Ugarte \cite{LU}]
For $n\geq 4$, consider the nilpotent Lie group and the left-invariant complex structure identified by the left-invariant $(1,0)$-coframe satisfying the following structure equations:
$$
d\varphi^{i}=0\,, \quad i=1,\ldots, n-1\,, \quad d\varphi^{n}=\sum_{i=1}^{n-1}a_i\varphi^{i\bar i }\,,
$$ where $a_1,\ldots, a_{n-1} \in \R$ such that $a_i\ne0$, for all $i=1,\ldots, n-1$, and $\sum_{i=1}^{n-1}a_i=0.$ Choosing $a_i\in \Q$, for all $i=1,\ldots, n$ guarantees the existence of a co-compact lattice, giving a compact nilmanifold.
On top of this, we know that nilmanifolds have vanishing first Bott-Chern class, as any left-invariant metric is Chern-Ricci flat (see \cite[Proposition 2.1]{LV}). Thus we are in position to apply Theorem \ref{immain}.\\
Moreover, in the same paper the authors produced another suitable family of $4$-dimensional nilmanifolds depending on three parameters in $\mathbb Q(i)$, where one can find Chern-Ricci flat balanced and astheno-K\"ahler metrics, identified by the following structure equations:
$$
 d\varphi^1=d\varphi^2=d\varphi^3=0\,, \quad d\varphi^4=A\varphi^{12}+ B\varphi^{13}+ C\varphi^{23}+ \varphi^{1\bar 1}+ \varphi^{2\bar 2 }- 2\varphi^{3\bar 3}\,, \quad A, B, C\in \Q(i).
$$
\end{exa}

To cover also the case of complex dimension $3$, we need to follow a different path, as we can see in the following remark.
\begin{rmk}
    We notice that no compact non-Kähler examples are known in complex dimension $3$, and we do not expect them to exist. Indeed, on threefolds, astheno-Kähler metrics are exactly corresponding to SKT metrics, which in light of the Fino-Vezzoni conjecture, are not expected to coexist with balanced structures in this setting.
\end{rmk}
Not having astheno-Kähler metrics is not, however, a big problem, mainly because the positivity of the form $\Omega$ for the ansatz is not a necessary condition, as we only need its positivity when wedged with the metric on the base. With this in mind, we are able to identify an interesting class of examples in dimension $n\geq 3$, still satisfying the condition $i\p\bpa\tilde \Omega=0$, without necessarily having $\tilde \Omega$ to be the $(n-2)$-th power of an astheno-Kähler metric. This class is given by compact balanced manifolds admitting a  holomorphic submersion with $1$-dimensional fibres.

\begin{prop}\label{propfiber}
Let $\pi\colon M^n\to X^{n-1}$ be  a holomorphic submersion and assume  $(M, \tilde \omega)$ is a compact balanced manifold. Then, there exists $\tilde\Omega\in \Lambda_{\R}^{n-2, n-2}(M)$ satisfying \eqref{hypOm}.
\end{prop}
\begin{proof} 
Thanks to \cite[Proposition 1.9]{M}, we know that $X$ admits balanced metrics. Let us fix   $\omega_X$  a balanced metric on  $X$.
 So, we can consider  $\tilde \Omega=\pi^*\omega_X^{n-2} $. Thus,  we have 
$
d\tilde \Omega=0
$.
Then, we just need to check if $\tilde\omega \wedge \tilde \Omega>0.$ We then fix a point $p\in M$ and choose holomorphic coordinates  $\{z_1, \ldots, z_n \}$ on a neighbourhood of $p$ such that $\{z_1, \ldots, z_{n-1}\}$ are holomorphic coordinates on a neighbourhood of $\pi(p)$ in $X$ and $z_n$ is the  holomorphic coordinate of the fibre over $p$, see for instance \cite[Lemma 5.6]{T}, such that 
$$
\tilde \omega_{n\bar n}=1\,, \quad \pi^*\omega_X=i\sum_{i=1}^{n-1}dz^{i\bar i} \quad \mbox{and } \tilde \omega_{i\bar k }=\lambda_i\delta_{ik}\,, \quad i,k=1,\ldots, n-1\,. 
$$
From this, we immediately have that 
$$
 \begin{aligned}
\tilde \Omega=\pi^*\omega_X^{n-2}=&\, i^{n-2}(n-2)!\sum_{i=1}^{n-1}dz^{1\bar 1\ldots \hat i \hat{\bar i }\ldots n-1 \overline{ n-1} }\,.
\end{aligned}
$$ 
Then, it is easy to see that 
$$
\begin{aligned}
\tilde \omega \wedge \tilde \Omega=i^{n-1}(n-2)! \left(\sum_{i=1}^{n-1}dz^{1\bar 1 \ldots \hat i \hat{\bar i }\ldots n\bar n} + \left(\sum_{i=1}^{n-1} \lambda_i \right)dz^{1\bar 1\ldots  \hat n\hat{\bar n} }
+ \sum_{i=1}^{n-1}\tilde \omega_{i\bar n }dz^{1\bar 1\ldots i\hat {\bar i} \ldots \hat n\bar n }-\tilde \omega_{n\bar i  }dz^{1\bar 1\ldots \hat i {\bar i}\ldots  j \hat{\bar j} \ldots n\hat{\bar n} }\right)\,.
\end{aligned}
$$
 Thus, $\tilde \omega \wedge \tilde \Omega$ is represented in $p$ by the matrix
$$
B=\frac{1}{n-1}\begin{pmatrix}{\rm Id} & \tilde \omega_{i\bar n }\\
\tilde \omega_{n\bar i } & \left(\sum_{i=1}^{n-1}\lambda_i\right)
\end{pmatrix}
$$ which is positive definite if and only if 
$$
\det(B)=\frac{1}{(n-1)^n}\sum_{i=1}^{n-1}\left(\lambda_i-\lvert\tilde \omega_{i\bar n }\rvert^2\right)>0\,.
$$ On the other hand, using that $\omega $ is a metric, we have that $\lambda_i-|\tilde \omega_{i\bar n}|^2>0$, for all $i=1, \ldots, n-1$, concluding the proof.
\end{proof}
This proposition is thus very interesting, as in most  examples gives us an particular choice of $\tilde \Omega$, since we frequently have an explicit balanced metric on the base. In the following subsections, we shall describe a series of examples which we will also see to be satisfying the hypothesis of Proposition \ref{propfiber}.\par

Finally, we can also recover an example on threefolds in a case where $i\p\bpa\tilde \Omega\neq 0$
\begin{exa}[Angella, Guedj, Lu \cite{AGL}]
    Consider the six-dimensional nilmanifolds (corresponding to the case \cite[Example (Ni)]{AGL}) identified by the structure equations
    \[
        d\varphi^1=d\varphi^2=0, \quad d\varphi^3=\rho\varphi^{12}+\varphi^{1\bar 1}+\lambda\varphi^{1\bar 2}+D\varphi^{2\bar 2},
    \]
    for $\{\varphi^1,\varphi^2,\varphi^3\}$ a coframe of invariant $(1,0)$-forms, $\rho \in \{0,1\}$, $\lambda \in \mathbb{R}_{\geq 0}$ and $D \in \mathbb{C}$ such that $\text{Im}(D)\geq 0$. When the corresponding Lie algebra is $\mathfrak{h}_2$, $\mathfrak{h}_3$, $\mathfrak{h}_4$ or $\mathfrak{h}_5$ (in the notation of \cite{Sal, U, UV}), these manifolds carry both a balanced metric $\tilde\omega$ (which is automatically Chern-Ricci flat thanks to the nilmanifold structure) and  a plurinegative metric $\omega'$, i.e. $i\partial \bar \partial \omega' \le 0 $, making $\tilde\Omega=\omega'$ a natural choice to apply Theorem \ref{immain}.
\end{exa}
By a recent work of George, see  \cite[Theorem 1.3]{Ge}, a compact balanced manifold $(M, \tilde \omega)$ with $c_1^{BC}(M)=0$, admitting a metric $\alpha$ such that $i\partial \bar \partial \alpha^{n-2}\le 0 $, always admits a balanced Chern-Ricci flat metric in the balanced class of $\tilde \omega$. This allows us to apply Theorem \ref{immain} in this wider setting.

\subsection{Torus fibrations on Calabi-Yau surfaces}\label{gp}
Other very interesting examples are given by the constructions of Goldstein and Prokushkin (see \cite{GP}), which Fu and Yau (in \cite{FuY}) showed to be highly relevant in the study of the Hull-Strominger system. In order to briefly recall the construction, let $M$ be the total space of a $T^2$-bundle over a Calabi-Yau surface $(S,\omega_S)$ with holomorphic volume $\Theta\in \Lambda^{2,0}(S)$. Choose $S$ such that it admits closed $2$-forms $\omega_P$ and $\omega_Q$ such that $\omega_P+i\omega_Q \in \Lambda^{2,0}(S)$, and $[\omega_P/2\pi],[\omega_Q/2\pi] \in H^2(S,\mathbb{Z})$. Then we can find a $(1,0)$-form $\theta$ on $M$ such that $d\theta=\omega_P+i\omega_Q$ and $i\theta\wedge\bar\theta>0$. Now, $\theta\wedge\Theta$ defines a holomorphic volume form on $M$, along with $\pi^*\omega_S+i\theta\wedge\bar\theta$, which corresponds to a Chern-Ricci flat balanced metric. Finally, choosing $\omega_P$ and $\omega_Q$ such that either one identifies a non-zero cohomology class guarantees that $M$ does not carry Kähler metrics. 
\begin{rmk}
    The form $\pi^*\omega_S$ satisfies conditions \eqref{hypOm}, and thus is a natural choice of $\tilde \Omega$ in order to apply Theorem \ref{immain}.
\end{rmk}
\begin{rmk}
    If $S$ is chosen from a certain class of algebraic K3 surfaces (see \cite{BTY}), then it admits a finite quotient with isolated singularities which can be resolved without discrepancy. On top of this, the action producing the quotient also preserves the metric structure described above, along with the form $\pi^*\omega_S$, making this quotient a perfect example where to apply Theorem \ref{glueorb} and Theorem \ref{glueorbOm}.
\end{rmk}

Let us conclude by focusing on two specific locally homogeneous examples whose  structure will allow us to construct examples suitable to apply Theorems \ref{immain}, \ref{glueorbOm} and \cite[Theorem 1]{GS}. 

\subsection{The Iwasawa manifold}\label{iwa}

Recall that the Iwasawa manifold $M={\rm Heis}(3, \C)/{\rm Heis}(3, \Z[i])$ is the unique complex parallelizable nilmanifold of complex dimension $3$. 

\begin{exa}\label{exIwa}
The center of ${\rm Heis}(3, \C)$ is given by $\mathbb{C}$, whose natural action descends to a $T^2$-action on $M$, which gives rise to (see for a more general assertion \cite{Ro}) a holomorphic principal $T^2$-bundle structure
\begin{equation}\label{bundleiw}
\pi\colon M \to T^4.
\end{equation}
Moreover, the complex parallelizable nilmanifold structure once again guarantees that any left-invariant Hermitian metric is balanced and Chern-Ricci flat.  Thus, we are in the position to apply Proposition \ref{propfiber} to find $\tilde \Omega$ satisfying \eqref{hypOm} and then apply Theorem \ref{immain}. We also have an explicit choice of $\tilde \Omega$, as we have the flat torus metric $\omega_{T^4}$ that naturally allows us to choose  $\tilde \Omega=\pi^*\omega_{T^4}$.
\end{exa}

Let us now recall the standard coframe of invariant (with respect to the Heisenberg group operation) $(1,0)$-forms:
\begin{equation}
\varphi^1:=dz^1, \quad \varphi^2:=dz^2, \quad \varphi^3:=dz^3-z^2dz^1,
\end{equation} 
which satisfy the following structure equations:
\begin{equation}\label{eqstruIwa}
d\varphi^1=d\varphi^2=0\,,\quad d\varphi^3=\varphi^1\wedge \varphi^2\,.
\end{equation}
Using \cite[Lemma 2.5]{UV}, we are led to  infer that, with such a frame, any left-invariant balanced metric on ${\rm Heis}(3, \C)$ is biholomorphically isometric to $$\tilde \omega_t:=\frac{i}{2}(\varphi^1\wedge\bar \varphi^{ 1}+\varphi^2\wedge\bar \varphi^{ 2}+t^2\varphi^3\wedge\bar \varphi^{ 3})\,, \quad t^2> 0 \,,
 $$
which descends to a Chern-Ricci flat balanced metric on the Iwasawa manifold $M$, for every $t\neq 0$. We can thus just focus on the family $\tilde \omega_t$.

\begin{rmk}\label{iwansbal}
    For the metrics $\tilde \omega_t$,  an easy computation shows that $\partial \tilde \omega_t=i\frac{t^2}{2}\varphi^1\wedge \varphi^2\wedge \varphi^{\bar 3}$, implying $|\p\tilde \omega_t|_{\tilde \omega_t}^2\frac{\omega_t^3}{3!}=i\bpa\tilde \omega_t\wedge\p\tilde\omega_t=t^2\frac{\tilde \omega_t^3}{3!}$, from which we get $|\p\tilde \omega_t|_{\tilde \omega_t}^2=t^2$.
 Hence,  the equation $F_{\tilde \omega_t}(u)=0$, recalling   \eqref{opbalanced},   becomes
\begin{equation}\label{keriw}
\Delta_{\tilde \omega_t} u +\frac{t^2}{2}u=0\,.
\end{equation}
We claim that  $u\equiv 0$ is the unique solution of \eqref{keriw}, for any $t> 0 $ outside a countable set. Indeed, first of all we notice that the bundle map  
$$
\pi\colon (M, \tilde \omega_t)\to (T^4, \omega_{T^4})\,, \quad \omega_{T^4}=\frac{i}{2}(\varphi^1\wedge \bar \varphi^{ 1}+ \varphi^2\wedge \bar \varphi^{ 2} )
$$ 
corresponding to the projection on the first two coordinates, is a Riemannian submersion, for any $t>0$. Moreover, we observe that  the fibres are minimal, as the mean curvature vector of the fibres has to be left-invariant. Then, the claim is equivalent to prove that $Z({\rm Heis}(3, \C))$ is minimal in ${\rm Heis}(3, \C)$. This is trivially true since the Levi-Civita connection computed on central vector fields is identically zero (see also \cite[Example 3.4]{E}). Now, we can integrate equation \eqref{keriw} along the fibres of $\pi$, and obtain 
\begin{equation}\label{mediakeriw}
    \Delta_{\omega_{T^4}}\hat{u}+\frac{t^2}{2}\hat{u}=0\,,  
\end{equation}
where $\hat{u}(x):=\int_{T_x}u\omega_{T^2, t}$ is the average of $u$ along the fibres, $\omega_{T^2, t}$ is the metric obtained by rescaling the flat metric on $T^2_x$ by a factor $t^2$ (which in this case coincides with the metric induced on the fibres by $\omega_t$) and $\omega_{T^4}$ is the flat metric on the base (again coinciding with the one induced on the base by $\omega_t$).\\
Here, equation \eqref{mediakeriw} gives us two possibilities: either $\frac{t^2}{2}\notin {\rm Spec}(T^4,\omega_{T^4} )$ or $\frac{t^2}{2}\in {\rm Spec}(T^4,\omega_{T^4} )$. Choosing then $t$ in order to land in the first case, puts us in the position to conclude that $\hat u \equiv 0$. This allows
 us  to apply \cite[Corollary 1.7]{Bord}, telling us that, in our setting, the $k$-th eigenvalue corresponding to eigenfunctions with vanishing average along the fibres is bounded from below by the $k$-th eigenvalue of $\Delta_{\omega_{T^2, t}}$, which is equal to $4\pi^2 t^2k^2>t^2/2$. Hence, for these values of $t$, $u$ is necessarily vanishing, and thus equation \eqref{keriw} has only the trivial solution.
The case in which $\frac{t^2}{2}\in {\rm Spec}(T^4,\omega_{T^4} )$ does not satisfy condition $\ker F_{\tilde \omega_t}=\{0\}$, since $u=f\circ \pi$, with $f\colon T^4\to \R$ an eigenfunction of $\Delta_{\omega_{T^4}}$ with eigenvalue $\frac{t^2}{2}$ is clearly contained in $\ker F_{\tilde\omega_t}$, which makes it of positive dimension.
\end{rmk}

\begin{exa}
    The orbifold quotient constructed by Sferruzza and Tomassini, see \cite{ST},  provides an example on which we can apply Theorem \ref{glueorbOm}. Indeed, the action the authors consider preserves the flat torus metric $\omega_{T^4}$ on the base, making $\pi^*\omega_{T^4}^2$ a suitable choice for $\tilde \Omega$ also on the orbifold quotient and hence allowing us - through Theorem \ref{glueorbOm} - to produce Chern-Ricci flat balanced metrics on the crepant resolution. Moreover, it is worth underlining that the action also preserves the metrics $\omega_t$ considered in Remark \ref{iwansbal}, showing that the kernel of the operator $F_\omega$ can be either zero or non-zero also in the orbifold setting.
\end{exa}

\begin{rmk}
    The fact that for the same metric (and corresponding balanced class) we are able to achieve our constructions with both ansatz, suggests the expected fact that Chern-Ricci flat balanced metrics, as well as  constant Chern scalar balanced metrics, might have very large moduli space inside the balanced class. Hence, in order to be able to geometrize such classes, we forsee the need to introduce additional constraint, possibly on the torsion.
\end{rmk}

 Let us  conclude this subsection  showing that  on  $(M, \tilde \omega_1)$, we have
 $$\ker \Delta_{\bar \partial }^T=\langle \varphi^{\bar 1},\varphi^{\bar 2}, \varphi^{\bar 3} \rangle\,.$$
 First of all, we  observe that, since the symmetrization process,  introduced by Belgun in \cite{Bel},  commutes with $*$,$ \bar \partial $, $\partial $, every element in $\ker \Delta^T_{ \bar \partial }$ has to be left-invariant. Furthermore, it is easy to see that $$
(\ker  \Delta^T_{\bar \partial })^{\flat}\cap \ker\bar \partial =\mathcal H_{\bar \partial}^{0,1} =\langle \varphi^{\bar 1}, \varphi^{\bar 2}\rangle\,.
$$
 Moreover, one can easily see  
 that $\Delta_{\bar \partial }\varphi^{\bar 3 }=\varphi^{\bar 3 }$ while  $\frac12\Lambda^2(i\bar \partial  \varphi^{\bar 3 }\wedge \partial \tilde\omega_1)=-\varphi^{\bar 3}$, giving that 
$$
\Delta_{\bar \partial }\varphi^{\bar  3}+\frac12\Lambda^2(i\bar \partial \varphi^{\bar 3 } \wedge \partial \tilde\omega_1)=0\,.
$$
 This allows us to obtain the claim.
\subsection{Nakamura Manifolds}\label{naka}
The final example we want to discuss is the one of Nakamura manifolds, as constructed by Cattaneo and Tomassini in \cite{CT}.\par
To briefly recall the construction of Nakamura manifolds, fix $M \in \text{SL}(n,\mathbb Z)$ to be diagonalizable and let $P \in \text{GL}(n,\mathbb R)$ such that $P^{-1}MP=\text{diag}(e^{\lambda_1},...,e^{\lambda_n})$, where  $\lambda_i\in \R$ are  such that 
 \begin{equation}\label{condizsl}
    \sum_{i=1}^n \lambda_i =0\,.
    \end{equation} 
    From here, we can consider the group action $\rho:\mathbb C \rightarrow \text{GL}(n,\mathbb C)$ given by 
    $$\rho(w):=\text{diag}\left(e^{\frac{\lambda_1}{2}(w+\bar w)},\ldots,e^{\frac{\lambda_n}{2}(w+\bar w)}\right),$$
    which allows  to construct the semidirect product $G_M:=\mathbb C \ltimes_\rho \mathbb C^n$. Then, we build the lattices
    \[
    \Gamma_\tau':=\mathbb Z \oplus i\tau\cdot\mathbb Z \quad \text{and} \quad \Gamma_P'':=\mathbb{Z}^n \oplus  iP\cdot \mathbb{Z}^n,
    \]
    where $\tau \in \mathbb R \setminus \{0\}$. After noticing that $\Gamma_{P,\tau}:=\Gamma_\tau'\ltimes_\rho \Gamma_P'' \leq G_M$, we  define Nakamura manifolds as the quotients
    \[
    N=N_{M,P,\tau}:=G_M/\Gamma_{P,\tau}.
    \]
    These $(n+1)$-dimensional compact manifolds inherit the left-invariant coframe 
    \[
    \varphi^0:=dw\,, \quad \varphi^j:=e^{-\frac{\lambda_j}{2}(w+\bar w)}dz_j\,, \quad  j=1, \ldots, n,
    \]
    of $(1,0)$-forms from $G_M$  satisfying the following structure equations:
    \[
    d\varphi^0=0, \quad d\varphi^j=-\frac{\lambda_j}{2}(\varphi^0+\bar\varphi^0)\wedge\varphi^j\,,\quad  \, j=1, \ldots, n\,.
    \]
   It can  then be easily checked that 
    \[
    \tilde\omega_t:=\frac i2\left(t^2\varphi^0\wedge\bar\varphi^0+\sum_{j=1}^n\varphi^j\wedge\bar\varphi^{j}\right) \quad \text{and} \quad \Theta:=\varphi^0\wedge \ldots \wedge \varphi^n
    \]
    define respectively a family of balanced metrics and a holomorphic volume form of constant 
    $\tilde\omega_t$-norm on $N$, making $\tilde\omega_t$ also Chern-Ricci flat, and $N$ a Calabi-Yau manifold.
    \begin{exa}\label{nakansbal}
    Following the discussion in \cite[Section 5]{CT}, any Nakamura manifold inherits a  holomorphic $T^{2n}$-bundle over a $2$ dimensional torus. A fundamental difference with the previous example is that now the metrics induced by $\tilde\omega_t$ on the fibres are not equal, but they vary depending on the base parameter $w$. Nevertheless, it is straightforward to notice (by fixing a fiber, rescaling the coordinates and using \eqref{condizsl}) that the Laplacians induced on each fibre all share the same eigenvalues, which are also the ones of the flat metric on $T^{2n}$. Hence, we can repeat the argument used in Example \ref{iwansbal} to obtain once again that the operator $F_{\tilde\omega_t}$ has vanishing kernel up to a countable set of values for $t$.
    \end{exa}

On the other hand, Theorem \ref{immain} \emph{cannot} be applied to these manifolds, as we explain in the following remark.
\begin{rmk}\label{noOmnak}
Thanks to the symmetrization process, we can reduce ourselves to work with invariant forms.
We will prove the claim in complex dimension $3$. In this case, the structure equations are:
$$
d\varphi^0=0\,, \quad d\varphi^1=-\frac{\lambda}{2}(\varphi^0+\bar \varphi^{ 0})\wedge \varphi^1\,, \quad d\varphi^1=\frac{\lambda}{2}(\varphi^0+\bar \varphi^{ 0})\wedge \varphi^1\,, \quad \lambda \in \R\backslash\{0\} \,.
$$
Using these, one can easily prove that 
\begin{equation}\label{ippzero}
i\partial \bar \partial(\varphi^{i}\wedge \bar \varphi^{j})=0\,, \quad i<j\,,
\end{equation} while 
\begin{equation}\label{ippnonzero}
i\partial \bar \partial (\varphi^{0}\wedge \bar \varphi^0)=0\,, \quad i\partial \bar \partial (\varphi^{1}\wedge \bar \varphi^1)=i \lambda^2\varphi^0\wedge \bar \varphi^0\wedge \varphi^1\wedge \bar \varphi^1\,, \quad i\partial \bar \partial (\varphi^{2}\wedge \bar \varphi^2)=i \lambda^2\varphi^0\wedge \bar \varphi^0\wedge \varphi^2\wedge \bar \varphi^2\,.
\end{equation} Then, given a general left-invariant $\tilde \Omega\in \Lambda_{\R}^{1,1}(N)$, we can write it as:
$$
\tilde \Omega=i\sum_{i,j=0}^2a_{i\bar j }\varphi^i\wedge \bar \varphi^j\,.
$$ Then, using \eqref{ippzero} and \eqref{ippnonzero}, we have that 
$$
i\partial \bar \partial \tilde \Omega=i^2\lambda^2(a_{1\bar 1 }\varphi^0\wedge \bar \varphi^0\wedge \varphi^1\wedge \bar \varphi^1+ a_{2\bar 2}\varphi^0\wedge \bar \varphi^0\wedge \varphi^2\wedge \bar \varphi^2)\,,
$$
and hence
 $$
\Lambda^2_{\tilde\omega_t}(i\partial \bar \partial \tilde \Omega)\frac{\tilde\omega_t^3}{3!}=2i\partial \bar \partial \tilde \Omega \wedge \tilde\omega_t\,.
$$
 On the other hand, we easily see that 
 $$
i\partial \bar \partial \tilde \Omega \wedge \tilde\omega_t=i^3\lambda^2(a_{1\bar 1}+ a_{2\bar 2})\varphi^0\wedge \bar \varphi^0\wedge \varphi^1\wedge \bar \varphi^1\wedge\varphi^2\wedge \bar \varphi^2=\lambda^2(a_{1\bar 1}+ a_{2\bar 2})\frac{\tilde\omega_t^3}{3!}\,.
$$From which we conclude that $\Lambda^2_{\tilde\omega_t}(i\partial \bar \partial \tilde \Omega)=\lambda^2(a_{1\bar 1}+ a_{2\bar 2})$. Now, $\Lambda^2_{\tilde\omega_t}(i\partial \bar \partial \tilde \Omega)\le 0 $ if and only if $a_{1\bar 1}+ a_{2\bar 2}\le 0 $.  But, computing $ \tilde\omega_t\wedge \tilde \Omega$ and imposing its positivity, we  need to have  $a_{1\bar 1}+ a_{2\bar 2}>0$ which is absurd, giving the claim.
\end{rmk}
In the three dimensional case, fixing $\tilde \omega_1$ as the reference balanced metric,  we can easily infer that 
$$
\ker \Delta_{\bar \partial }^T=\langle\varphi^{\bar 0}, \varphi^{\bar 1}, \varphi^{\bar 2} \rangle\,.
$$
Indeed, we have that
$
\mathcal H_{\bar\partial }^{0,1}=\langle\varphi^{\bar 0 }\rangle
$,
and one can easily obtain the claim verifying that 
$$
\frac12\Lambda_{\tilde \omega_1}^{2}(i\bar \partial \varphi^{\bar j}\wedge \partial \tilde \omega_1)=-\frac{\lambda^2}{4}\varphi^{\bar j}=-\Delta_{\bar \partial }\varphi^{\bar j }\,, \quad j=1,2\,.
$$


\begin{thebibliography}{99}

\bibitem[Ag]{Ag} S. Agmon, B. F. Jones, G. W. Batten, \emph{Lectures on elliptic boundary value problems}, Princeton, New Jersey: D. Van Nostrand Company, (1965).


\bibitem[AB1]{AB2} L. Alessandrini, G. Bassanelli, {\it A class of balanced manifolds}, Proc. Japan Acad. {\bf 80} (2004), Ser. A, No. 1, 6--7.

\bibitem[AB2]{AB1} L. Alessandrini, G. Bassanelli, {\it Modifications of compact balanced manifolds}, C. R. Acad. Sci. Paris Sér. I Math. {\bf 320} (1995), No. 12, 1517--1522.

\bibitem[AB3]{AB3}
L. Alessandrini, G. Bassanelli, \emph{
Small deformations of a class of compact non-Kähler manifolds}, Proc. Am. Math. Soc. {\bf 109} (1990),  No. 4, 1059--1062.

\bibitem[AI]{AI} B. Alexandrov, S. Ivanov, {\it  Vanishing theorems on Hermitian manifolds}, Differential Geom. Appl. {\bf 14} (2001), No. 3,  251--265.

\bibitem[AGF]{AGF} B. Andreas, M. Garcia-Fernandez, {\it Solutions of the Strominger System via Stable Bundles on Calabi-Yau Threefolds}, Commun. Math. Phys. \textbf{315} (2012), No. 1,  153--168.

\bibitem[ACS]{ACS} D. Angella, S. Calamai, C. Spotti, \emph{On Chern-Yamabe problem}, Math. Res. Lett. {\bf 24} (2017), No. 3, 645--677.

\bibitem[ACS2]{ACS2}
D. Angella, S. Calamai, C. Spotti, \emph{Remarks on Chern-Einstein Hermitian metrics,} Math. Z. {\bf 295} (2020), No. 3-4, 1707--1722.

\bibitem[AGL]{AGL}
D. Angella, V. Guedj, C.H. Lu, \emph{Plurisigned Hermitian metrics}, 
Trans. Am. Math. Soc. {\bf 376} (2023), No. 7, 4631--4659.

\bibitem[AO]{AO} D. Angella, A. Otiman, \emph{A note on compatibility of special Hermitian structures}, arXiv preprint arXiv:2306.02981 (2023).







\bibitem[AP]{AP} C. Arezzo, F. Pacard, \emph{Blowing up and desingularizing constant scalar curvature Kähler manifolds}, Acta Math. {\bf 196} (2006), No. 2, 179-228.

\bibitem[A]{A} T. Aubin, \emph{Equations du type de Monge-Ampère sur les varietes kähleriennes compactes}, C. R. Acad. Sci. Paris {\bf 283} (1976) 119--121.



\bibitem[BTY]{BTY} M. Becker, L.-S. Tseng, S.-T. Yau, \emph{New Heterotic non-Kähler geometries}, Adv. Theor. Math. Phys. {\bf 13} (2009), No. 6,  1815–-1845.

\bibitem[BV]{BV}
L. Bedulli, L. Vezzoni, \emph{A parabolic flow of balanced metrics,} J. Reine Angew. Math. {\bf 723} (2017), 79--99.

\bibitem[Bel]{Bel}
F. Belgun, \emph{On the metric structure of non-K\"ahler complex surfaces,} Math. Ann. {\bf 317} (2000), No.1,  1--40.

\bibitem[BM]{BM} O. Biquard, V. Minerbe. \emph{A K\"ummer construction for gravitational instantons}, Commun. Math. Phys.  {\bf 308} (2011), No. 3,  773--794.

\bibitem[B]{Bis}
J.-M. Bismut, \emph{A local index theorem for non-K\"ahler manifolds},  Math. Ann. {\bf 284} (1989), No. 4, 681--699.

\bibitem[Bord]{Bord} M. Bordoni, \emph{Spectral estimates for submersions with fibers of basic mean curvature}, An. Univ. VestTimi. Ser. Mat.-Inform. {\bf 44} (2006), No. 1, 23--36.



\bibitem[C]{C} E. Calabi, \emph{Extremal Kähler metrics}, Seminar on Differential Geom., Ann. of Math. Stud. {\bf 16} (1982),  259--290.

\bibitem[CT]{CT} A. Cattaneo, A. Tomassini, \emph{$\p\bpa$-lemma and $p$-Kähler structures on families of solvmanifolds}, Math. Z. {\bf 308} (2024), No. 3, 17 p.


\bibitem[CDS]{CDS} X. Chen, S. Donaldson and S. Sun, \emph{Kähler–Einstein metrics on Fano manifolds. I, II, III}, J. Amer.
Math. Soc. {\bf 28} (2015), 183--197, 199--234, 235--278.

\bibitem[CRS]{CRS} I. Chiose, R. Răsdeaconu, I. Şuvaina,  \emph{Balanced manifolds and SKT metrics}, Ann. Mat. Pura Appl. {\bf 201} (2022), No. 5, 2505--2517 .





\bibitem[CPY2]{CPY2} T. C. Collins, S. Picard, S.-T. Yau, \emph{The Strominger system in the square of a Kähler class}, arXiv preprint arXiv:2211.03784 (2022).





\bibitem[DS]{DS} R. Dervan, L. M. Sektnan, \emph{Extremal metrics on Fibrations}, Proc. London Math. Soc.(3) {\bf 120} (2020), No. 4,  587--616.

\bibitem[D1]{D1} S. K. Donaldson, \emph{Scalar curvature and stability of toric varieties}, J. Differ. Geom. {\bf 62} (2002), No. 2,  289--349.

\bibitem[E]{E}
P. Eberlein, \emph{Geometry of 2-step nilpotent groups with a left invariant metric. II,} Trans. Am. Math. Soc. {\bf 343}  (1994), No. 2, 805--828.


\bibitem[FP]{FP}
T. Fei, D.H. Phong,
\emph{Unification of the Kähler-Ricci and anomaly flows},  Differential geometry, Calabi-Yau theory, and general relativity. Lectures given at conferences celebrating the 70th birthday of Shing-Tung Yau at Harvard University, Cambridge, MA, USA, May 2019. Somerville, MA: International Press. Surv. Differ. Geom. {\bf 23} (2020), 89--103.

\bibitem[FeY]{FeY} T. Fei, S.-T. Yau, {\it Invariant solutions to the Strominger system on complex Lie groups and their quotients}, Comm. Math. Phys. {\bf 338} (2015), No. 3, 1183--1195.  


\bibitem[F1]{F1}
 J. Fine, \emph{Constant scalar curvature Kähler metrics on fibred complex surfaces},
J. Differ. Geom. {\bf 68} (2004), No. 3, 397--432.

\bibitem[F2]{F2} J. Fine, \emph{Fibrations with constant scalar curvature Kähler metrics and the CM-line bundle}, Math. Res.
Lett. {\bf 14} (2007), No. 2,  239--247.


\bibitem[FGV]{FGV} A. Fino, G. Grantcharov, L. Vezzoni, \emph{Astheno–Kähler and balanced structures on fibrations}, Int. Math. Res. Not. {\bf 22} (2019), No. 22,  7093--7117.

\bibitem[FV]{FV} A. Fino, L. Vezzoni, \emph{Special Hermitian metrics on compact solvmanifolds}, J. Geom. Phys. {\bf 91} (2015), 40--53.




\bibitem[FLY]{FLY} J.-X. Fu, J. Li and S.-T. Yau, \emph{Balanced metrics on non-Kähler Calabi-Yau threefolds}, J. Diff. Geom. {\bf 90} (2012), No. 1,  81--129.

\bibitem[FWW]{FWW} J. Fu, Z. Wang, D. Wu, \emph{Form-type Calabi-Yau equations}, Math. Res. Lett. {\bf 17} (2010), No. 5, 887--903.


\bibitem[FuY]{FuY} J. Fu, S.-T. Yau, \emph{The theory of superstring with flux on non-Kähler manifolds and the complex Monge-Ampère equation}, J. Differ. Geom. {\bf 78} (2008), No. 3, 369--428.

\bibitem[FG]{FG}
E. Fusi, G. Gentili, 
\emph{Special metrics in hypercomplex geometry},
arXiv preprint arXiv:2401.13056 (2024).

\bibitem[GI]{gigis}
G. Ganchev, S. Ivanov, \emph{Holomorphic and Killing vector fields on compact balanced Hermitian manifolds}, Int. J. Math.  {\bf 11} (2000), No. 1, 15--28.

\bibitem[GF]{GF} M. Garcia-Fernandez, {\it Lectures on the  Strominger system}, Trav. Math. {\bf 24} (2016), 7--61.

\bibitem[GFGM]{GFGM} M. Garcia-Fernandez, R. Gonzalez-Molina, \emph{Harmonic metrics for the Hull-Strominger
system and stability}, arXiv preprint arXiv:2301.08236 (2023). 

\bibitem[G1]{G} P. Gauduchon, {\it Fibr\'e hermitiens \`a endomorphisme de Ricci non n\'egativ}, Bulletin de la S.M.F. {\bf 105} (1977), 113--140.

\bibitem[G2]{gauconn}
P. Gauduchon,   \emph{Hermitian connections and Dirac operators},  Boll. Un. Mat. Ital. B (7) {\bf 11} (1997), No. 2, suppl.,
257--288.


\bibitem[G3]{Gau} P. Gauduchon, \emph {La 1-forme de torsion d’une vari\'et\'e hermitienne compacte}, Math. Ann. {\bf 267} (1984), No. 4, 495--518.

\bibitem[G4]{Gaulaplace}
 P. Gauduchon, \emph{Le théorème de l’excentricité nulle,} C. R. Acad. Sci. Paris Sér. A-B {\bf 285} (1977), No. 5, A387--A390.
 
 \bibitem[Ge]{Ge}
 M. George, \emph{Volume forms on balanced manifolds and the Calabi-Yau equation}, arXiv preprint  arXiv:2406.00995 (2024).

\bibitem[GiP]{GiP}
F. Giusti, F. Podestà, \emph{Real semisimple Lie groups and balanced metrics,} Rev. Mat. Iberoam. {\bf 39} (2023), No. 2, 711--729.

\bibitem[GS]{GS} F. Giusti, C. Spotti, \emph{A Kümmer construction for Chern-Ricci flat balanced manifolds}, Math. Z. {\bf 308} (2024), No. 61, 30 p. 

\bibitem[GP]{GP} E. Goldstein and S. Prokushkin, \emph{Geometric model for complex non-Kähler manifolds with $SU(3)$-structure}, Comm. Math. Phys. {\bf 251} (2004), No. 1, 65--78.

\bibitem[GGP]{GGP}
D. Grantcharov, G. Grantcharov, Y.S. Poon,\emph{ Calabi-Yau connections with torsion on toric bundles},  J. Differ. Geom. {\bf 78} (2008), No. 1, 13--32.

\bibitem[Hu]{Hu} C. Hull, \emph{Superstring compactifications with torsion and space-time supersimmetry}, In Turin 1985 Proceedings "Superunification and Extra Dimensions" (1986), 347--375.

\bibitem[Huy]{Huy}
D. Huybrechts, \emph{Complex geometry. An introduction,} Universitext. Berlin: Springer, (2005),  309 p.

\bibitem[JY]{JY}
J. Jost, S.-T. Yau, \emph{A nonlinear elliptic system for maps from Hermitian to Riemannian manifolds and rigidity theorems in Hermitian geometry,}
Acta Math. {\bf 170} (1993), No. 2, 221--254.


\bibitem[J]{J} D. D. Joyce, \emph{Compact manifolds with special holonomy}, Oxford Mathematical Monographs, Oxford University Press, Oxford, 2000.


\bibitem[LU]{LU}
A. Latorre, L. Ugarte, \emph{On non-K\"ahler compact complex manifolds with balanced and astheno-K\"ahler metrics}, 
C. R., Math., Acad. Sci. Paris,  {\bf 355} (2017), No. 1, 90--93.

\bibitem[LV]{LV}
J. Lauret, E.A. Rodríguez-Valencia, \emph{On the Chern-Ricci flow and its solitons for Lie groups, }
Math. Nachr. {\bf 288} (2015), No. 13, 1512--1526.

\bibitem[LeB]{LeB} C. LeBrun, \emph{Counter-examples to the generalized positive action conjecture}, Comm. Math. Phys. {\bf 118} (1988), No. 4, 591--596.

\bibitem[LeU]{LeU}
M. Lejmi, M. Upmeier, \emph{Integrability theorems and conformally constant Chern scalar curvature metrics in almost Hermitian geometry,} Communications in Analysis and Geometry {\bf 28} (2020), No. 7, 1603--1645. 




 \bibitem[LiY]{LiY}
K. Liu, X. Yang, \emph{Ricci curvatures on Hermitian manifolds}, 
 Trans. Am. Math. Soc. {\bf 369} (2017), No. 7, 5157--5196.

\bibitem[M]{M} M. L. Michelsohn, {\it On the existence of special metrics in complex geometry}, Acta Math. {\bf 149} (1982), 261--295.

\bibitem[NZ]{NZ}
L. Ni,  F. Zheng, \emph{On Hermitian manifolds whose Chern connection is Ambrose-Singer}, Trans. Am. Math. Soc. {\bf 376} (2023), No. 9, 6681--6707.

\bibitem[PPZ]{PPZ} D. H. Phong, S. Picard, X. Zhang. \emph{Geometric flows and Strominger systems}, Math. Z. {\bf 288} (2018), No. 1,  101--113.

\bibitem[Ro]{Ro}
S. Rollenske, \emph{Dolbeault cohomology of nilmanifolds with left-invariant complex structure,} 
Complex and differential geometry. Conference held at Leibniz Universität Hannover, Germany, September 14–18, 2009. Proceedings. Berlin: Springer Proceedings in Mathematics {\bf 8} (2011), 369--392 .

\bibitem[Sal]{Sal}
 S. Salamon, \emph{Complex structures on nilpotent Lie algebras,} J. Pure Appl. Algebra {\bf 157} (2001), No. 2-3,  311--333.
 
\bibitem[ST]{ST} T. Sferruzza, A. Tomassini, \emph{Dolbeault and Bott-Chern formalities: deformations and $\partial\bar{\partial}$-Lemma},  J. Geom. Phys. {\bf 175} (2022), No. 104470, 19 pp.


\bibitem[Sim]{Sim} S. R. Simanca, \emph{K\"ahler metrics of constant scalar curvature on bundles over $\mathbb{CP}^{n-1}$}, Math. Ann. {\bf 291} (1991), No. 2, 239--246.

\bibitem[Sto]{Sto}
J. Stoppa, \emph{K-stability of constant scalar curvature Kähler manifolds},
Adv. Math. {\bf 221} (2009), No. 4, 1397--1408.

\bibitem[STi]{STi}
 J. Streets,  G. Tian, 
\emph{Hermitian curvature flow},  
J. Eur. Math. Soc.  {\bf 13} (2011), No. 3, 601--634.



\bibitem[S]{S} A. Strominger, \emph{Superstrings with torsion}, Nucl. Phys. B {\bf 274} (1986), No. 2,  253--284.







\bibitem[Sz1]{Sz} G. Székelyhidi, \emph{An introduction to extremal Kähler metrics}, Graduate Studies in Mathematics. Vol. 152. American Mathematical Soc. (2014).


\bibitem[Sz2]{Sz2}
G. Székelyhidi, \emph{Blowing up extremal Kähler manifolds. II}, 
Invent. Math. {\bf 200} (2015), No. 3, 925--977.

\bibitem[Sz3]{Sz1}
G. Székelyhidi,
\emph{On blowing up extremal K\"ahler manifolds}, 
Duke Math. J. {\bf 161} (2012), No. 8, 1411--1453.

\bibitem[STW]{STW} G. Székelyhidi, V. Tosatti, B. Weinkove, \emph{Gauduchon metrics with prescribed volume form}, Acta Math. {\bf 219} (2017), No. 1,  181--211.


\bibitem[Ti]{Ti} G. Tian, \emph{Kähler–Einstein metrics with positive scalar curvature}, Invent. Math. {\bf 130} (1997), No. 1,  1--37.

\bibitem[Ti1]{Ti1} G. Tian, \emph{K-stability and Kähler–Einstein metrics}, Commun. Pure Appl. Math. {\bf 68} (2015), 1085--1156; Corrigendum, {\bf 68} (2015), 2082--2083.

\bibitem[T]{T}
V. Tosatti, \emph{{\rm KAWA} lecture notes on the K\"ahler-Ricci flow}, Ann. Fac. Sci. Toulouse Math. {\bf 27} (2018), No. 2, 285--376.



\bibitem[TW1]{TW1} V. Tosatti, B. Weinkove, \emph{The Monge-Ampère equation for (n-1)-plurisubharmonic functions on a compact Kähler manifold}, J. Amer. Math. Soc. {\bf 30} (2017), No. 2, 311--346.

\bibitem[U]{U}
L. Ugarte, \emph{Hermitian structures on six-dimensional nilmanifolds.} Transform. Groups {\bf 12} (2007), No. 1,  175--202.

\bibitem[UV]{UV}
L. Ugarte, R. Villacampa, \emph{Balanced Hermitian geometry on 6-dimensional nilmanifolds,} Forum Math. {\bf 27} (2015), No. 2, 1025--1070.

\bibitem[Us]{Us}
Y. Ustinovskiy, \emph{The Hermitian curvature flow on manifolds with non-negative Griffiths curvature},  Am. J. Math., {\bf 141} (2019), No. 6,   1751--1775.

\bibitem[Va]{Va} I. Vaisman, \emph{Locally conformal Kähler manifolds with parallel Lee form}, Rend. Mat. (6) {\bf 12} (1979), No. 2, 263--284.

\bibitem[V]{vez}
L. Vezzoni, \emph{A note on canonical Ricci forms on 2-step nilmanifolds}, 
Proc. Am. Math. Soc. {\bf  141} (2013), No. 1, 325--333.

\bibitem[Y1]{Y1} S.-T. Yau, \emph{Metrics on complex manifolds}, Science in China Series A Mathematics {\bf 53} (2010), 565--572.

\bibitem[Y2]{Y3}
S.-T. Yau, \emph{On the Ricci curvature of a compact K\"ahler manifold and the Complex Monge- Ampère equation, I,} Commun. Pure Appl. Math.  {\bf 31} (1978), 339--411.


\bibitem[Y3]{Y2}
S.-T. Yau, \emph{Open problems in geometry,} Proc. Symposia Pure Math. {\bf 54} (1993), 1--28.


\end{thebibliography}
\end{document}